\documentclass[11pt,draft,reqno]{amsart}
  \usepackage{amsmath}
\usepackage{amssymb}
\usepackage{amsfonts}
\usepackage{eucal}
\usepackage{xcolor}
    \makeatletter
     \def\section{\@startsection{section}{1}%
     \z@{.7\linespacing\@plus\linespacing}{.5\linespacing}%
     {\bfseries
     \centering
     }}
     \def\@secnumfont{\bfseries}
     \makeatother
\usepackage[full]{textcomp}

\newtheorem{theorem}{Theorem}[section]
\newtheorem{lemma}[theorem]{Lemma}
\newtheorem{proposition}[theorem]{Proposition}
\newtheorem{corollary}[theorem]{Corollary}
\theoremstyle{definition}
\newtheorem{definition}[theorem]{Definition}

 \theoremstyle{remarks}
\newtheorem{remarks}[theorem]{{\bf  Remarks}}
\newtheorem{remark}[theorem]{{\bf  Remark}}
\numberwithin{equation}{section}
 \usepackage{a4wide}
\font\sevenrm =cmr10 at  7  pt
\def\ddate {\sevenrm \ifcase\month\or January\or
February\or March\or April\or May\or June\or July\or
August\or September\or October\or November\or December\fi\! {\the\day}, \!{\sevenrm\the\year}}

\def \a{{\alpha}}
\def \b{{\beta}}
\def \D{{\Delta}}
\def \d{{\delta}}
\def \e{{\varepsilon}}
\def \g{{\gamma}}

\def \k{{\kappa}}
\def \l{{\lambda}}

\def \o{{\omega}}
\def \O{{\Omega}}
\def \p{{\varphi}}
\def \t{{\vartheta}}

\def \m{{\mu}}
\def \s{{\sigma}}

\def \noi{{\noindent}}
\def \qq{{\qquad}}
\def \dd{{\rm d}}
\def\F{{\mathcal F}}
\def\A{{\mathcal A}}
\def\P{{\mathcal P}}
\def\={\,=\,}

\def\E{{\mathbb E \,}}

\def\P{{\mathbb P}}

\def\R{{\mathbb R}}

\def\Z{{\mathbb Z}}

\def\Q{{\mathbb Q}}

\def\N{{\mathbb N}}

\def\C{{\mathbb C}}
 \def\beq{\begin{equation}}
\def\eeq{\end{equation}}
  
\def\ben{\begin{eqnarray}}
\def\een{\end{eqnarray}}\font\sevenrm= cmr10 at 9 pt
\def\ddate {\sevenrm \ifcase\month\or January\or
February\or March\or April\or May\or June\or July\or
August\or September\or October\or November\or December\fi\! {\the\day}, \!{\sevenrm\the\year}}

  \scrollmode

at  8 pt

  at 9,7pt
    at  8,2 pt
  \scrollmode
 at 9,9  pt
 at 8,2  pt
 \font\tf=cmbxsl10 at 9,5pt

 \font\tf=cmbxsl10 at 9pt
 \font\sevenrm= cmr10 at 7,3 pt

\begin{document}

\vskip 1.5cm

 \scrollmode

  \newcommand{\sm}{\vskip} \newcommand{\ens}{\enspace} \

 \title[Divisibility and primality in   random walks]
  {Divisibility and primality in random walks}
 
   \author{ Michel J.\,G.\, WEBER}
  \address{IRMA, UMR 7501, Universit\'e
Louis-Pasteur et C.N.R.S.,   7  rue Ren\'e Descartes, 67084
Strasbourg Cedex, France.
 \vskip 1 pt   E-mail:    {\tt  michel.weber@math.unistra.fr}}
 \begin{abstract}    In this paper we study the divisibility and primality properties of the Bernoulli random walk. We improve or extend some of our divisibility results  to wide classes of iid or independent non iid random walks. We also obtain new primality results       for  the Rademacher random walk. We study the value distribution of divisors of the random walk in the Cram\'er model, and obtain a general estimate of a similar kind to that of the Bernouilli model. 
   Earlier   results on  divisors and quasi-prime numbers in the Bernoulli model are recorded, as well as  some other recent    in the Cram\'er random model,  based on  an estimate due to Selberg.
\end{abstract}
\maketitle

\tableofcontents
\section{Introduction.}\label{s1}

In this paper we study      divisibility and primality properties of the most  fundamental     integer lattice random walks: the      Bernoulli random walk and Rademacher random walk. This is continuing previous works in  \cite{W5}, \cite{W5a}, \cite{W1}, \cite{W4},     \cite{W7},  
   \cite{W6}, where these questions
     have been thoroughly investigated, as well as other questions related to their  asymptotic structural independence. 
 We next extend the study to  the Cram\'er random walk on which Cram\'er's     model built. This  is an important example of independent, non identically distributed random walk.  A     thorough   probabilistic study 
 of this   model was recently made in \cite{W8}.
     The case of a general square integrable lattice random walk was very recently  investigated. We prove in  \cite{W10},  \cite{W9} sharp divisibility results    valid for quite large classes of independent,   identically distributed  (iid) lattice   random walks, and obtain also divisibility results in independent   non identically distributed   case. The used tools are new in this context, and built on  a remarkable coupling method and a recent improvement of an important necessary condition of arithmetical type for the applicability of the local limit theorem. We   refer to Part IV of our 2009's book   \cite{W2}, the  Chapters 11    and  13, which  are  devoted to divisors, and their applications in Analysis, especially Fourier analysis. 

\vskip 2   pt  Throughout  let     $\rho\in ]0,1[$ and $\mathcal B(\rho)=\{B_n(\rho), n\ge 1\}$ denotes  the binomial random walk,   $B_n(\rho)=\sum_{k=1}^n \b_k(\rho)$, for each $n$, where  $\b(\rho)=
\{\b_i(\rho) , i\ge 1\}$ is a sequence of iid  binomial random variables ($\P\{\b_i(\rho)=0\} =1-\P\{\b_i(\rho)= 1\}=\rho $).  When $\rho =1/2$ we use the simplified notation $\mathcal B =\{B_n , n\ge 1\}$, $B_n =\sum_{k=1}^n \b_k $, for each $n$,    $\b =
\{\b_i  , i\ge 1\}$ 
  being iid Bernoulli random variables.
    Let also $\{R_n, n\ge 1\}$, $R_n=\sum_{k=1}^n \e_k$,  where  $\e=
\{\e_i , i\ge 1\}$ is  a sequence of iid 
Rademacher random variables ($\P\{ \e_i =\pm 1\} =1/2$). Further let   $\mathcal S= \{S_n, n\ge 3\}$, $S_n=\sum_{i=3}^n \xi_i$, be the random walk associated to the Cram\'er model,   $\xi_i$,  $i\ge 3$, being independent
  random variables defined by $\P\{\xi_i= 1\}=  1-\P\{\xi_i= 0\}= \frac1{\log i}$. The basic probability space on which these random variables are defined is noted $(\O,\A,\P)$, and the  corresponding expectation symbol $\E$.
  \smallskip\par
  
\smallskip\par 

\smallskip\par
 

\medskip\par
 {\it Notation and convention.}    
-- Throughout the paper,  let  
$p$   denotes an arbitrary prime number,   the letter $C$   a universal constant whose value may change at each occurence, and   $C_{\a, \b, \ldots}$     a constant depending only on the parameters ${\a, \b, \ldots}$. 
\vskip 2 pt 
\noi -- We agree that $\sum_{\emptyset}=0$, $\sup_{\emptyset}=0$.

\bigskip\par
\section{Preamble on summation methods.}\label{s2}

In background  of  the study of the divisibility/primality questions   related to these models are often are summation methods, which may be  of   particular efficiency, as they are linked between them through   Tauberian results. We first begin with some important facts. The uniform  model  (the uniform distribution on $\{0,1,2\ldots, n\}$, for each $n$), also called the   Kubilius model,  is related to Ces\`aro's summation method, whereas the binomial model is  to Euler
  summation method, and there exist Tauberian theorems 
  comparing   summation  methods, in particular these two methods.  Therefore   fine arithmetical results can be   directly used, for one summation  method (one random model), depending on the rate of convergence in this one, to answer a question concerning another random model.  
\begin{definition} A  set  
$A\subset 
\N$   has  Ces\`aro density $\l$ of order 1 ($C_1$--density $\l$),  if 
$$\lim_{n\to \infty}\frac{\#\{ j\le n: j\in A\}}{n}= \l.
$$
\end{definition}Often, the most readily available information about a set  of integers is that it has $C_1$-Ces\`aro density with a given rate of convergence. \begin{definition}  A subset $A$ of $  \N$ is said to have {\it Euler density  $\l$ 
with parameter
$\varrho $} (in short ${ E}_\varrho$ density $\l$) if 
\begin{equation*} 
\lim_{n\to \infty}\sum_{j\in A} {{n}\choose{j}} \varrho^j (1-\varrho)^{n-j}= \l.
\end{equation*} 
\end{definition}
Diaconis and Stein proved in \cite{DS}
(Theorem 1) the following fundamental result:
\begin{theorem}\label{ds}For any $A\subset \N$, and $\varrho\in ]0,1[$ the following assertions are equivalent:
 \begin{eqnarray*} \label{dschar} ({\rm i})  & &\qq \hbox{\it $A$ has  ${E}_\varrho$ density $\l$},
  \cr   ({\rm ii})  & &\qq\lim_{t\to \infty} e^{-t}\sum_{j\in A} \frac{t^j}{j!}=\l,
 \cr ({\rm iii})  & &\qq \hbox{\it  for all $\e>0$}, \quad    \lim_{n\to \infty} \frac{\#\{j\in A : n\le j< n+\e\sqrt n\} }{\e\sqrt n} 
 =\l . 
 \end{eqnarray*}
\end{theorem} 
 
\begin{remark} \rm The implication (ii)$\Rightarrow$(iii) was in fact
 first established by Jurkat \cite{J}, and a
proof   using an extension of Wiener Tauberian theorem was given by
Moh in \cite{Mo}, not quoted in the paper. 
\end{remark}\vskip 2pt 
An important consequence presented by the authors as the principal  result of the paper,  is   that:
\begin{equation} \label{E.rho} \hbox{\it  The existence and value of ${  E}_\varrho$ density does not depend on
the value of $\varrho$}.  
\end{equation}  
A study of limit results in the present setting, for binomial random walks,  therefore reduces to the case of the Bernoulli random walk. 
As a consequence of  Theorem 3  in Diaconis and Stein  \cite{DS}, we have
\begin{theorem} If $A\subset \N$ has $E_\varrho$ density $\l$, then $A$ has $C_1$ density $\l$. Conversely the following implication  in which   $g$ is assumed to be regularly varying at infinity with exponent $\a$, $-1<\a\le -1/2$, holds true 
\begin{equation} \label{ 2}   
   \Big|\frac{\#\{ j\le n: j\in A\}}{n}- \l\Big|\le g(n)\qquad\!\!\!\!\!\Longrightarrow  \qquad\!\!\!\!\!  \Big|\P\{B_n(\varrho) \in A\} - \l\Big|\le K g(n)n^{1/2}. 
 \end{equation}
 \end{theorem}
See also Hardy \cite{H1},    Theorem 149.  
 Let
$\{\a_k, k\ge 0\}$ be a  sequence of reals, $\l$ real and $\e(n)$ be positive reals. More generally, it follows from  Lemma \ref{Eulerbound} that 
 \begin{eqnarray*} \label{4.61c}\Big|\frac{1}{n}\sum_{j=0}^{n-1}\a_j- \l \Big|\le \frac{\e(n)}{n^{ 1/2} }
\quad \Longrightarrow \quad  \Big| \sum_{h=0}^n   2^{-n}  
{{n}\choose{h}}\a_h-\l \Big|  \le   \frac{C}{ \sqrt n}\, \max_{1\le \ell\le n} (\e(\ell) \sqrt \ell)     .\end{eqnarray*}

Let us illustrate this by giving an application to {\it $k$-free} numbers in $\mathcal B(\rho)$, which is implicitly included in \cite{DS}, (Corollary 4).
   
 \begin{definition}An integer $s$ is said {\it $k$-free} if and only if it is not divisible be the $k$-th power of any prime. 
 \end{definition}Leahey and Nymann proved in \cite{LN}  among other papers, by using a formula analogous to 
    \begin{equation*}
    1-\P\{ B_n(\varrho)\ \hbox{is prime}\}= - \sum_{2\le d\le n }  \m(d) \P\{ d|B_n(\varrho)\},  
\end{equation*}  $\m$ being M\"obius function, that
\begin{equation}\label{Bnprime}   
\lim_{n\to \infty}\P\big\{ B_n(\varrho)\ k\hbox{-{\rm free}}\big\}= \frac{1} {\zeta(k)},\end{equation}    where $\zeta$ is the Riemann zeta function.
       A much more precise result holds:
\begin{corollary}  {\it For $k\ge 2$ integer,}
 \begin{equation}\label{k.free.DS} \Big|\P\big\{ B_n(\varrho)\ k\hbox{-{\rm free}}\big\}  - \frac{1}{\zeta(k)} \Big|\le E(n) , 
\end{equation}
where
$$ E(n)\le C\left( n^{\frac1k-1}\cdot  e^{ -Ak^{-8/5}\frac{(\log 
n^{{3/5}}}{(\log\log  n) ^{1/5}}}\right).$$
\end{corollary} 
\begin{proof} Let ${\mathcal Q}_k$,  ${\mathcal Q}_k(x)$ respectively denote the set of $k$-free integers, and the number of $k$-free integers $\le x$ ($k\ge 2$, integer).  By a result of Walfisz \cite{Wal},
  \begin{equation}\label{k.free} 
{\mathcal Q}_k (x)= 
\frac{x}{\zeta(k)} + {\mathcal  O} 
\big( x^{\frac1k}. e^{ -Ak^{-8/5}.\frac{(\log 
x^{{3/5}}}{(\log\log  x) ^{1/5}}}\big)
 . 
\end{equation}
Take $A= {\mathcal Q}_k$  and note that 
$$\Big|\frac{\#\{ j\le n: j\in A\}}{n}- \frac{1}{\zeta(k)} \Big|\le g(n)$$
with 
$$g(n)= n^{\frac1k-1}\cdot e^{ -Ak^{-8/5}.\frac{(\log 
n^{{3/5}}}{(\log\log  n) ^{1/5}}},$$
Clearly $g(n)$   is regularly varying at infinity with exponent $\a$, $-1<\a\le -1/2$. Thus \eqref{k.free.DS} follows.
\end{proof}
 \medskip\par
  For $0$-density sequences, we have the following complementary result, which does not appear  in the literature.
\begin{proposition}
Let $A$ be a set of integers of (natural) density $0$. Then for every $p>1$,
\begin{equation}\label{hardy1} \lim_{N\to \infty}\frac{1}{N}\sum_{n=1}^N \P\{ B_n(\varrho)\in A\}^p =0.\end{equation} 
 \end{proposition} 
 This follows from Hardy \cite[ineq.\,3]{H}: {\it Let $p>1$. Then for $0<\varrho <1$, $a_n\ge 0$,}
 \begin{equation}\label{hardy} \sum_{n=1}^\infty \left(\sum_{m=0}^n   {{n}\choose{m}}\varrho^m(1-\varrho)^{n-m}a_m\right)^p\le \varrho^{-1}  \sum_{n=1}^\infty a_n^p.
 \end{equation} 


  \bigskip\par
 We refer to Apostol \cite{A}, McCarthy \cite{MC}, Montgomery and Vaughan \cite{MV}, Tenenbaum \cite{Te1}   as  bibliographical sources on arithmetical functions, and for results on arithmetical properties of $r$-uples of integers, co-primality,  pairwise co-primality,   least common multiple, and associated arithmetical functions.   We cite  T\'oth's elaborated survey \cite{To} for   sharp limit results of this kind with  estimates of the remainder,  Fernandez and Fernandez \cite{FF}, 
 for  limiting moments and distribution,   for instance of  {gcd}$ (X^n_1, \ldots, X^n_n)$, {lcm}$(X^n_1, \ldots, X^n_n)$ of  arrays of independent uniformly distributed random variables. These specialized arithmetical functions 
are not considered here.  In each 
case, however, what is investigated is often the behavior for $n$ large of $\E f(T_n)^r$, $r\ge 1$, where $f(n)$ is one of these  arithmetical functions, $f$ may be varying with $n$ (for arrays of random variables), $T_n= Y_1+\ldots+Y_{\nu(n)}$  (the expectation symbol $\E$ corresponding to the underlying probability space on which these random variables $Y_j$ are defined). 
Recall that any arithmetical function $f(n)$ can be represented in the form
$$ f(n) =\sum_{d|n} \Phi(d).$$
The function $\Phi(d)$ is defined by the M\"obius inversion formula
$$ \Phi(n)= \sum_{d|n}\mu\Big(\frac{n}{d}\Big) f(d)  .$$
For instance for the Bernoulli model,
$$\E f(n) =\sum_{d|n} \P\{d|B_n\}\,\Phi(d).$$
See also sub-section \ref{sub.bin.way.back}. Hence the interest of having at disposal sharp estimates of  $\P\{d|B_n\}$ when both $d$ and $n$ vary.
 \vskip 2 pt
When $\Phi(d)$ is such that the series $\sum \Phi(d)/d$ is absolutely convergent, we recall that by Wintner's Theorem
$$ \lim_{N\to\infty}\frac{1}{N}\sum_{n=1}^Nf(n)= \sum_{d=1}^\infty \frac{\Phi(d)}{d}.$$
See  Postnikov \cite{Po} p.\,138. 
For many arithmetical functions $\Phi(d)$ is simple. Here are some examples of mostly studied functions. Let  $\mathcal P$ be the set of prime numbers and let the letter $p$ always denote a prime number.  
  \begin{eqnarray*}
    \s_s(n) = \sum_{d|n} d^s &&\qq\hbox{\rm the sum of $s$-powers of divisors of $n$, $s\in \R$;}
  \cr   J_s(n)=\sum_{d|n} d^s \m\Big(\frac{n}{d} \Big) &&\qq\hbox{\rm the generalized Euler totient function, $s>0$;}
\cr    w(n)=\sum_{p|n}\frac{\log p}{p} &&\qq\hbox{\rm Davenport's function;}
\cr \Phi(n)=\sum_{d|n} \frac{\log d}{d} &&\qq\hbox{\rm Erd\"os and  Zaremba function;}
\cr 
 \omega(n) = \sum_{p|n,p\in \P} 1 &&\qq\hbox{\rm the prime divisor function;}
   \cr  \Omega(n) = \sum_{p|n,p\in \P} p &&\qq\hbox{\rm   the sum of  prime divisor function;}
 \cr\theta(n)= 2^{\o(n)} &&\qq\hbox{\rm    counting the number of prime divisors of  $n$.}\end{eqnarray*}   
   The classical Euler totient function $\p(n)$ counting the number of integers $k $ less than $n$ and such that $(k,n) =1$ is $J_1(n)$, and $\s_0(n)$ is the divisor function $d(n)$ counting the number of divisors of $n$. 

\vskip 2 pt To the best of our knowledge, {\it all}     divisibility results for a random walk existing in the literature are {\it in}-probability results, except for \cite{W6}. It is a pending question to know whether it is possible to derive almost sure   counterparts, from existing {\it in}-probability results using almost sure convergence criteria, which is a far more complicated task.     These criteria are   not  familiar to number theorists. 
A random walk  often has   natural  structural independence properties. Therefore  a study of almost sure divisibility/primality  in a random walk is possible, at the price of many additional technicalities. That's make the research in this domain  attractive. In the case of the Bernoulli random walk, it is not complicated to show   that   the system 
 $$ \big\{ \chi\{d|B_n\},    \    2\le d\le n,\ n\ge 1\big\}$$
is a mixing system, namely        the   correlation
function   satisfies for each     $d$ and $\d$ greater than 2, 
 \begin{equation}
 \label{delta.def.lim}
  \lim_{m-n\to \infty }  \P\big\{ d|B_n  ,    \d|B_m\big\}-\P \{ d|B_n \}\P \{    
\d|B_m \}= 0.
\end{equation}
The estimation of the correlation 
 \begin{equation}
 \label{delta.def}
  {\Delta} \big( (d,n), (\d, m)\big)= \P\big\{ d|B_n  ,   \d|B_m\big\}-\P \{ d|B_n \}\P \{    
\d|B_m \}.
\end{equation}
  was made in \cite{W4},   and is substancially improved in    
  \cite{W14}. We have the following formulation
  \begin{eqnarray}\label{6.4}& & {\bf \Delta} \big( (d,n),  (\d, m)\big)  =
\cr & &\qq \quad{1\over d\d}  \sum_{j=1}^{d-1}\sum_{h=1}^{\d-1}  e^{ i\pi  (  {j\over d}n +   {h\over \d}m) }  \cos^{m-n}{ \pi   {h\over
\d} }\, \Big(\cos^n{ \pi  (  {j\over d} +  
{h\over \d})   }-\cos^{n} {
\pi j\over d}     \cos^{n} {
\pi h\over \d}\Big),
 \end{eqnarray}  
which   can be more simplified    using symmetries.

\medskip  \par

The paper is organized as follows. In Sections \ref{s2}-\ref{s6} we  study the divisors in the Bernoulli random walk. We also study examples. Hitherto, the   divisibility or primality properties of general   i.i.d.   or independent non i.i.d. lattice   random walks, were not much explored and the results we present are mostly the first. In Section \ref{s8}, we describe the value distribution of divisors in the Rademacher random walk and present some applications; we essentially study primality properties. In Section \ref{s9}, we study the value distribution of divisors in the Cram\'er random walk, and record some primality   results obtained by the author. Complements to previous sections are gather in Appendix.

\section{Divisibility in the Bernoulli random walk.}
 \label{s3}
We study  the probability
  $ \P \{d|B_n \}$ when both $d$ and $n$ vary.    We first prove a basic but  handy estimate. Next we state a sharper and nearly optimal estimate using Theta function, earlier established by the author in \cite{W1}. 
  Applications are given to: 
   $\E \s_{-1}(B_n)$, $\E d(B_n)$,     $\P\big\{  B_n\, , B_m \ \hbox{co-prime} \big\}$, $\P\{ B_n\, \hbox{$z$-quasi-prime} \}$, $\P\{ B_n\,\hbox{  prime} \}$. We will  also estimate the probability $\P\{D|B_nB_m\}$ for $n<m$. The study of the divisors of $B_nB_m$  is one piece of the estimation of the correlation appearing in \eqref{delta.def} when $n$ and $m$ are relatively near,  details are given in Section \ref{s6}.

\subsection{Basic estimates.}
\begin{proposition}   \label{Lemma44}
  For any integers   
 $d\le n$,
  \begin{equation}\label{multiple1}\big|\P \{ d | B_n \} -\frac{1}{d}\big|
   \le      {2\over d}\sum_{ 1\le j<d/2}
e^{-        {
 2n j^2/  d^2}  } . 
\end{equation}
\end{proposition}
\begin{proof} From the formula 
\begin{equation}\label{basic}d\,\chi(d| B_n)=\sum_{j=0}^{d-1} e^{2i\pi  {j \over d}B_{n }},
\end{equation}
  we get by integration,
\begin{equation}\label{basic.i} \P\big\{ d | B_n\big\}={1\over d}\sum_{j=0}^{d-1} e^{i\pi n{j\over d}}\big(\cos { \pi j\over d}\big)^{n}.
  \end{equation}
 As $e^{i  n(\pi-x)} \cos^n    (\pi-x)  =(-1)^ne^{-i 
nx}  (-1)^n\cos^n   x= e^{-i 
nx}  \cos^n   x  $, we have in fact  by distinguishing the case $d$ even from the case $d$ odd,   the simplified form
\begin{equation}\label{Lemma41} \P\big\{ d | B_n\big\} ={1\over d}+ {2\over d}\sum_{ 1\le j<d/2}
(\cos{ \pi n{j\over d}})\big(\cos { \pi j\over d}\big)^{n}.
\end{equation}
    Thus   
    $$\big|\P\big\{ d | B_n\big\} -\frac{1}{d}\big|
   \le      {2\over d}\sum_{ 1\le j<d/2}
e^{n\log (  1-2 \sin^2{ \pi j\over 2d})} .$$
     As $\sin x\ge (2/\pi )x$, $0\le x\le \pi/2 $, it follows that $\log (  1-2 \sin^2{ \pi j\over 2d})\le -2 \sin^2{ \pi j\over 2d}\le  -2  ({
  j\over  d}) ^2   $. Thus 
\begin{eqnarray*} \big|\P\big\{ d | B_n\big\} -\frac{1}{d}\big|
  &\le &    {2\over d}\sum_{ 1\le j<d/2}
e^{-        {
 2n j^2/  d^2}  }.
\end{eqnarray*}   
\end{proof} 
This   implies estimate in   Example 2 in Diaconis and Stein \cite{DS}    without using Taylor's expansion or Ramus' formula.  See also Remark \ref{rem.Ex.2}.
Since $\sum_{j\ge 1} e^{- { 2n j^2/  d^2}}\le
Cd/\sqrt n $,   we  deduce  from Proposition \ref{Lemma44} the useful estimate below.
\begin{corollary}\label{c.useful.est}
\begin{equation*}   \sup_{1\le d\le n}\,\big|\P\big\{d|B_n\big\}- {1\over d}  \big| \, \le \,   {C \over\sqrt n} .
\end{equation*}
where $C$ is an absolute constant. 
\end{corollary}

  \bigskip \par Proposition \ref{Lemma44}  is   already sharp enough  to
imply  the   useful bound below.
 \begin{proposition} \label{unifd}  $$\sup_{n\ge 1} \sum_{d<\sqrt n} \big|\P\big\{  d|   B_{ n}
\big\}-{1\over d} 
\big|<\infty.$$
\end{proposition}\begin{proof}
  Let $A $ and   $m$ be positive {\it reals}. Then for $d\ge 1$,
$$ \int_{m\over d+1}^{m\over d}e^{-At^2}{  d t\over t}\ge e^{-A({m\over d})^2}\int_{m\over d+1}^{m\over d} {  d t\over t}=e^{-A({m\over
d})^2}\log (1+{ 1\over d})\ge  { 1\over 2d}e^{-A({m\over
d})^2} , $$
since $\log 1+x\ge \frac{x}{2}$, if $0\le x\le 1$. Writing  $n=\nu^2$ and letting $m=\nu j$, $A=2$  gives
$${1\over d}\sum_{j=1}^d e^{-2({\nu j\over
d})^2}\le 2 \sum_{j=1}^d  \int_{\nu j\over d+1}^{\nu j\over d}e^{-2t^2}{  d t\over t}.$$
Let 
 $ H= \sum_{1\le d<\nu} \big\{{1\over d}\sum_{1\le j\le d} e^{-2({\nu j\over
d})^2}\big\}$.  
By permuting the order of summation, we get
\begin{eqnarray*}H&\le& 2\sum_{1\le d<\nu}\Big\{\sum_{1\le j\le d}   \int_{\nu j\over d+1}^{\nu j\over d}e^{-2t^2}{  d t\over t}\Big\}
=2\sum_{1\le j<\nu}\sum_{j\le d<\nu} \int_{\nu j\over d+1}^{\nu j\over d}e^{-2t^2}{  d t\over t}
\cr &\le &2
\sum_{1\le j<\nu} \int_{  {\nu j\over \nu+1}  }^{\nu  }e^{-2t^2}{  d t\over t}=  2 \int_{  {\nu +1 \over \nu }}^{\nu  }\Big(\sum_{j=1}^{\min(\nu ,
{\nu +1 \over \nu }t)}1\Big)  e^{-2t^2}{  d t\over t}  
\cr &\le & 4\int_{{ 1 \over 2 } }^\infty  e^{-2t^2}  
d t  = C. \end{eqnarray*}
Consequently, by using (\ref{multiple1}),
 \begin{equation*}  
  \sum_{d<\sqrt n} \big|\P\big\{  d|   B_{ n}
\big\}-{1\over d} 
\big|\le 2\sum_{d<\sqrt n} {1\over d}\sum_{ 1\le j<d/2}
e^{-        {
 2n j^2/  d^2}  }\le 2H\le C, \end{equation*}
as claimed.\end{proof}

{Some applications}:
\vskip 3 pt
\noi(1) {\it An estimate  of $\E d(   B_{ n})$.} Proposition \ref{unifd}  implies the following estimate.
\begin{proposition}[Weber \cite{W11},\,Prop.\,6] \label{dbn} For  some universal constant $C$, we have 
$$ \sup_{n\ge 2}\, \big| \E d(   B_{ n})- \log n \big| \le C   .$$  \end{proposition}
We refer to the cited article  for the proof.
\smallskip \par
\noi This estimate is certainly improvable. The same question arises with $\E d^2(   B_{ n})$.  We remark that \begin{eqnarray}\E d^2(B_n)&=&
\sum_{   d \le n\atop    \d  \le n}\P\{ [d,\d]|B_n\}\, =\, 
 \sum_{ [d,\d]\le n}\P\{ [d,\d]|B_n\}
.
\end{eqnarray}
 A relevant result is    \begin{theorem}[Shi-Weber \cite{Sh.We},\,Th.\,2.1] \label{shi.weber}
  {\rm (i)}   If $0<\s< 1$, then
 \begin{eqnarray*}
 \sum_{ [d,\d]\le N}  \frac{1}{[d,\d]^{\s }}
 &=&\frac 3{\pi^2} \cdot \frac 1{1-\s}N^{1-\s} (\log N)^2+\mathcal O\left( N^ {1-\s}(\log N)\right).
\end{eqnarray*}
 
{\rm (ii)} Further, if $\s =1$,
 \begin{eqnarray*}
 \sum_{ [d,\d]\le N}  \frac{1}{[d,\d]}
 &=&\frac 1{\pi^2} (\log N)^3+\mathcal O\left((\log N)^2\right).
\end{eqnarray*}
\end{theorem}
\begin{remark}[Large zeta values imply large values of a very small selection of Dirichlet polynomials]\rm Combined with the arithmetic perturbation technic of Dirichlet polynomials     introduced in \cite{Sh.We},  Proposition 1.1, 
  large zeta values can be  compared with large values of a very small selection of Dirichlet polynomials. For any $N\ge1$,  and any  real numbers $\s,\e, t$,
  \begin{equation}\label{e2}\sum_{n=1}^N \frac{1}{n^{\s +it}}  =\sum_{[d,\d]\le N}   \frac{J_\e (d)J_\e (\d)}{[d,\d]^{{\s +2\e+it}}}
\sum_{m=1}^{ \frac{N}{[d,\d]}}            \frac{1} {m^{{\s+2\e +it}}}.
\end{equation}
    Let $\s_0>0$. 
    Thus by Theorem \ref{shi.weber}-(ii), for any   $\s\ge \s_0$,  $|t|\ge 1$, $\e> 0$,\begin{eqnarray*}\big|\zeta(1 + it) \big|&\le& C(\e)\, (\log |t|)^3  \max_{M\in\mathcal E({|t|})} \Big|\sum_{m=1}^{M} \frac{1}{m^{{1+2\e +it}}}\Big|+ C  |t|^{-\s},
\end{eqnarray*}
where $\mathcal E(|t|))\subset [1, |t|]$ and 
$\#(\mathcal E(|t|))=\mathcal O\big( \log |t|\big)$. Hence a much smaller selection than the one derived from Abel's summation. These pointwise implications extend to Dirichlet polynomials indexed on intervals, and have by integration standard in-measure counterparts. This is displayed elsewhere. These questions are relevant in the recent work of     Matom\"aki and  Ter\"av\"ainen \cite{MT}.
\end{remark}
\vskip 3 pt
\noi(2) {\it Infinite M\"obius inversion.}  The following application of Proposition \ref{dbn} to infinite M\"obius inversion is proved in \cite{W11}.
\begin{theorem}\label{t2} Let $g:\N\to \R^+$. Assume that  for some $0<\e<1$,
\begin{equation}\label{locvar[g]}    \g_g(\e)= \sup_{n\ge 1} \sup_{\frac{|2m-n |}{\sqrt{n\log\log n}}\,\le 1+\e}\frac{ g(m) }{ g(n) }<\infty.
\end{equation}
Then the infinite M\"obius inversion 
\begin{equation}\label{HW} 
g(x) =\sum_{m=1}^\infty f(mx)\qq \Leftrightarrow \qq f(x) =\sum_{n=1}^\infty \m(n)g(nx),
\end{equation} holds for any $x$, as soon as
\begin{equation}\label{imi[g]} \sum_{n\ge 1} g(n)  \log n<\infty.\end{equation} 
\end{theorem}
See also Weber \cite{W13}, Remark 2,  for a rectification  of  Theorem 270 in Hardy et   Wright \cite{HW}, more precisely, the assertion that {\it under} assumption $ {\sum_{m,n=1}^\infty} |f(mnx)|<\infty$,     the implication    $(\Longleftarrow)$  in \eqref{HW} holds,     is not clear.\medskip \par \noi
(3) {\it Estimates of $\E \s_{-1}(B_n)$ and $\P\big\{  B_n\,{\it and}\, B_m \ {\it coprime} \big\}$.}
    Recall that $\s_s(k)$ (section \ref{s2})  
is a multiplicative function and  that $\s_{-s}(n)=n^{-s} \s_s(n)  $. \begin{corollary}
$$\Big|\E \s_{-1}(B_n)-   \zeta(2) \Big| \le\  C\,{  \log n\over\sqrt n} .$$
\end{corollary}

\begin{proof}Writing that $\s_{-1}(B_n)=\sum_{d\le n }\frac{1}{d}\chi\{d|B_n\}$ since $B_n\le n$, we have 
$$  \E \s_{-1}(B_n)= \sum_{d\le n }\frac{1}{d}\,\P\{d|B_n\}.  $$
Now $$\E \s_{-1}(B_n)- \sum_{d\le n}{1\over d^2}=
\sum_{d\le n}{1\over d}\big( \P\big\{d|B_n\big\}  - {1\over d }\big). $$
Whence,
$$\big|\E \s_{-1}(B_n)- \sum_{d\le n}{1\over d^2}\big|\le
{C \over\sqrt n}\ \sum_{d\le n}{1\over d}\le\  C\,{  \log n\over\sqrt n} .$$
Consequently,
  $$\Big|\E \s_{-1}(B_n)-   \zeta(2) \Big| \le\  C\,{  \log n\over\sqrt n} .$$
\end{proof}
\begin{corollary} We have for $n>m\ge 2$,
\begin{eqnarray*} \P\big\{  B_n\,{\it and}\, B_m \ {\it coprime} \big\}  &=&\sum_{d\le m\wedge (n-m)} {\m(d)\over d^2}\cr & &+\mathcal O\Big(   {\log m \over  \sqrt{m\wedge (n-m)}} +\Big({m\wedge (n-m)\over m\vee (n-m)}\Big)^{1/2}\Big).
\end{eqnarray*}
Moreover,
$$\lim_{m\to\infty\atop  {n }/{m} \to\infty}\P\big\{ (B_n,B_m)=1\big\}
 =\frac{1}{\zeta(2)}
 .$$
 \end{corollary}
 \begin{proof} By using M\"obius inversion formula,
$$\P\big\{ (B_n,B_m)=1\big\} =\sum_{d\le 
m\wedge (n-m)}\m(d)\P\{d|S_m\}\,\P\{d|S_{n-m}\}$$
$$=\sum_{d\le m\wedge (n-m)}\m(d)\Big(\P\big\{d|S_m\big\}-{1\over d}+{1\over d}\Big)\,\Big(\P\big\{d|S_{n-m}\big\}-{1\over d}+{1\over d}\Big) $$
$$=\sum_{d\le m\wedge (n-m)}\m(d)\Big(\P\big\{d|S_m\big\}-{1\over d}+{1\over d}\Big)\,\Big(\P\big\{d|S_{n-m}\big\}-{1\over d}+{1\over d}\Big) $$
$$=\sum_{d\le m\wedge (n-m)}\m(d)\Big( {1\over d}+\mathcal O\big({ 1\over\sqrt m}\big)\Big)\, \Big({1\over d}+\mathcal O\big({ 1\over\sqrt {n-m}}\big)\Big) $$
$$=\sum_{d\le m\wedge (n-m)} {\m(d)\over d^2}+\mathcal O\Big(\sum_{d\le m}
{1\over d }\big({1\over  \sqrt{ m}}+{1\over  \sqrt{ n-m}}\big)+\Big({m\wedge (n-m)\over \sqrt{m( n-m)}}\Big)\Big)$$
$$=\sum_{d\le m\wedge (n-m)} {\m(d)\over d^2}+\mathcal O\Big(   {\log m(\sqrt m+\sqrt{ n-m} )\over  \sqrt{ m(n-m)}} +\Big({m\wedge (n-m)\over m\vee (n-m)}\Big)^{1/2}\Big).$$
For instance: 
\vskip 2pt --- If  $(\log m)^2<n-m\le m$, 
\begin{eqnarray*} \Big|\P\big\{  B_n\,{\it and}\, B_m \ {\it coprime} \big\}  -\sum_{d\le  n-m } {\m(d)\over d^2}\Big|  & = & \mathcal O\Big(   {\log m \over  \sqrt{ n-m }} +\Big({ n-m\over m }\Big)^{1/2}\Big).
\end{eqnarray*}

--- If  $n\ge 2m$,  
\begin{eqnarray*} \Big|\P\big\{  B_n\,{\it and}\, B_m \ {\it coprime} \big\}  -\sum_{d\le m  } {\m(d)\over d^2}\Big|  & = & \mathcal O\Big(   {\log m \over  \sqrt{  m }} +\Big({  m\over n-m }\Big)^{1/2}\Big).
\end{eqnarray*}
This implies the second claim.\end{proof}
 
This suggests to study the following problem: Let $n_1<n_2<\ldots< n_k$ be positive integers such that 
$$\min\{ n_j/n_i, 1\le i<j \le k\}$$
is large. To estimate the probability that $B_{n_1}, B_{n_2},\ldots, B_{n_k}$ are mutually coprime, namely
$$\P\big\{ (B_{n_i}, B_{n_j})=1, \forall 1\le i<j \le k\big\}.$$ 
 That problem is related to a nice notion: for a subset of an integer lattice, to be {\it visible from the origin}, in this case the   set of random points   $(x,y)$ in $\N^2$, $x\neq y$, with coordinates in $\{B_{n_1}, B_{n_2},\ldots, B_{n_k}\}$. Referring to section 3.8 in Apostol \cite{A}, we say  that two lattice points $P$ and $Q$ are mutually visible  if the segment which join  them contains no lattice points other than the endpoints $P$ and $Q$. We also recall that two lattice points $(a,b)$ and $(m,n)$ are mutually visible if  and only if $(a-m, b-n)= 1$, or  equivalently  if and only if $(a-m,b-n)$ is visible from the origin. The standard case presented in \cite{A} has clear link with Farey series, and Dirichlet's easy   estimate on the number of corresponding visible  points is a particular case of a much more general result having a direct link with Riemann Hypothesis  (RH). See   \cite{W13} for more.
\smallskip \par

 \begin{remark}\rm Let $X$ and $Y$ be independent random variables having geometric distribution with parameter $1-e^{-\b}$ with $\b>0$: $\P\{X=n\}=(1-e^{-\b})e^{-\b(n-1)}$ for all $n\in\N$. Bureaux and Enriquez \cite{BE} have shown that 
 that error term
 $$ \P\{(X,Y)=1\}= \frac{1}{\zeta(2)}+\frac{\b}{\zeta(2)}+\mathcal O(\b^{3/2-\e}), \quad \hbox{for each}\ \e>0\ \hbox{as}\ \b\downarrow 0, $$
 is equivalent to RH. 
 
 A natural question arises from this nice result. It seems possible (although we have not investigated it) to   estimate the probability $\P\{X\,\hbox{prime}\}$ with accuracy, up to some extend. Would the accuracy of the error term be also sensitive? 
 \end{remark}

\subsection{The local limit theorem for binomial sums}
The first local limit theorem    was  established  nearly  three centuries ago, in the binomial case. We refer to     Szewczak    and  Weber \cite{SW}, Th. 1.1 for the statement   under the  form below. 
 \begin{theorem}[{\rm De Moivre--Laplace} (1733)]
\label{moivre}  Let
$0<\rho<1$. Let $X$ be such that
${\mathbb P}\{X=1\}=p=1-{\mathbb P}\{X=0\}
$. Let
$X_1, X_2,\ldots$ be independent copies of
$X$ and let $S_n=X_1+\ldots +X_n$, $n\ge 2$. Let also  $0<\g<1$.
Then for all $k$ such that $|k-np|  \le \g n\rho(1-\rho)$ and $    n \ge \max(\rho/(1-\rho),(1-\rho)/\rho)$, letting
$ x= \frac{k-n\rho}{\sqrt{  n\rho(1-\rho)}}$,
we have
 \begin{equation*} {\mathbb P}\{B_n(\rho)=k\}  \ =\
 \frac{e^{-  \frac{x^2}{  2  } }}{\sqrt{2\pi n\rho(1-\rho)}} \   e^E ,
 \end{equation*}
with
$|E|\le   \frac{3|x|+ 2 |x|^3}{(1-\g)\sqrt{ n\rho(1-\rho)}} +  \frac{1}{ 4n\min(\rho,1-\rho )(1 -  \g      )}$.

\vskip   3 pt  If $x= o(n^{1/6})$,  then
\begin{equation*}\Big| {\mathbb P}\{B_n(\rho)=k\}- \frac{1 }{\sqrt{2\pi n\rho(1-\rho)}}\,  e^{-  \frac{(k-n\rho)^2}{  2 n\rho(1-\rho) }}\Big|
\ \le  \   \frac{1 }{\sqrt{2\pi n\rho(1-\rho)}}\ e^{-  \frac{(k-n\rho)^2}{  2 n\rho(1-\rho) }}\Big( \frac{c_1 |x|^3}{\sqrt{ n}}+ \frac{c_2|x|}{\sqrt n} +  \frac{c_3}{ n}\Big).
\end{equation*}
The constants $c_i$, $i=1,2,3$ are explicit and depend on $p$ only. \end{theorem}

A useful  consequence is: {\it For any $b>0$, $n\ge 2$,}
\beq\label{binomial.llt.1}
  \sup_{k:|k-n\rho|\le \sqrt{b n\log n}}\Big|{\mathbb P}\{B_n(\rho)=k\}- \frac{1 }{\sqrt{2\pi n\rho(1-\rho)}}\,  e^{-  \frac{(k-n\rho)^2}{  2 n\rho(1-\rho) }}\Big|
\, \le  \,  C_b \Big( \frac{ \log^{3/2} n}{\sqrt n}\Big).  
\eeq 

\subsection{A concentration inequality   for binomial sums.}\label{sub2.2}
 The following  is Th.\,2.3   in MacDiarmid \cite{di}, which addresses a more general case.
\begin{lemma} \label{di.1}
Let $X_1, \dots, X_n$     be independent random variables, with $0 \le X_k \le 1$ for each $k$.
Let $S_n = \sum_{k=1}^n X_k$ and $\mu = \E S_n$. Then for any $\epsilon >0$,
 \begin{eqnarray*}
{\rm (a)}\quad  &&
 \P\big\{S_n \ge  (1+\epsilon)\mu\big\}
    \le  e^{- \frac{\epsilon^2\mu}{2(1+ \epsilon/3) } } ,
    \cr {\rm (b)} \quad & &\P\big\{S_n \le  (1-\epsilon)\mu\big\}\le    e^{- \frac{\epsilon^2\mu}{2}}.
 \end{eqnarray*}
\end{lemma}

\subsection{Sharper uniform estimate and  Theta     function.}

  One might believe in view of Proposition \ref{unifd} that the central comparison term  is $\frac1d$. But this is {\it not} so. The right comparison term -{\it with a sharper remainder term\,}- turns up to be more complicated, and is equal to ${\Theta(d,n)\over d}$ where $\Theta(d,n)$ is the Theta     function, 
  \begin{equation} \label{Theta}\displaystyle{ \Theta (d,n)  =  \sum_{\ell\in \Z} e^{in\pi{\ell\over   d }-{n\pi^2\ell^2\over 2 d^2}},     }
   \end{equation}
  $d\ge 1,\ n\ge 1$ being integers.  
\medskip \par The following Theorem proved in Weber \cite{W1}, 
provides a uniform estimate on the whole range of  divisors.
  \begin{theorem}[\cite{W1},\,Theorem II]\label{estPdlBn} 
{  \rm (1)}    We have the following uniform estimate: 
\begin{equation}\label{unif.est.theta}\sup_{2\le d\le n}\Big|\P\big\{d|B_n\big\}- {\Theta(d,n)\over d}  \Big|= {\mathcal O}\Big(\frac{(\log n)^{5/2}}{n^{3/2}}\Big).  
\end{equation}


{  \rm (2)} Further,
  for any $\a>0$,  
\begin{equation}\label{unif.est.theta.alpha} \sup_{d<  \pi   \sqrt{   n \over 2\a\log n}}\big|\P\big\{  d|   B_{ n} \big\}-{1\over d} 
\big|= {\mathcal O}_\e\big(n^{-\a+\e }\big),\qq \quad (\forall \e>0)  
\end{equation}   
 and for any $0<\rho<1 $, 
\begin{equation}\label{unif.est.theta.rho}  \sup_{d<  (\pi/\sqrt 2) n^{(1-\rho)/2} }\big|\P\big\{  d|   B_{ n} \big\}-{1\over d} 
\big|= {\mathcal O}_\e\big(e^{-(1-\e) n^\rho}\big),\qq \quad (\forall 0<\e<1). 
\end{equation} 
 \end{theorem} 
\vskip 2 pt The uniform estimate \eqref{unif.est.theta} covers the   critical region of divisors $d\sim \sqrt n$, and   we don't know  any  other   estimate of this kind covering this case.   This estimate  also admits a version  for binomial sums, which follows by simply copying the proof made in the Bernoulli case.
 
 \begin{remark}\label{unif.est.theta.alpha.rem}\rm  The proof we give is intrinsic, in particular we don't see how to  obtain this estimate from an approach using the uniform model.
 \end{remark}
\begin{remarks} \label{rem.Ex.2}\rm  --- Note  the
considerable variation of the error term between \eqref{unif.est.theta.alpha} and \eqref{unif.est.theta.rho}. 
These latter estimates are easy to prove and are true for binomial sums.  
In fact (\cite{DS},\,Example 2),
 \begin{equation}\label{multiple} \Big|\P\big\{d|B_n(\varrho) \}- \frac{1}{d}\Big|\le 
 e^{-        {
 8n \varrho(1-\varrho)  /  d^2}  }.
\end{equation} 
\vskip 3 pt
 \noi ---  Lemma 2.1 in      Cilleruelo,   Fern\'andez and Fern\'andez \cite{CFF} is only a weak  form of it. 
\vskip 3 pt
\noi  --- Letting next $h(x)$ be increasing to infinity with $x$ and such that $h(x)\le  \sqrt x $, it follows that
 \begin{equation}\label{multiple.h} \sup_{d< \sqrt{  n/ h(n)}}\Big|\P\big\{d|B_n(\varrho) \}- \frac{1}{d}\Big|\le 
 e^{-    8  
   \varrho(1-\varrho) h(n)  }.
\end{equation} 
   For instance,  
for any $\a>0$,   
 \begin{equation}\label{unifdivalpha}  \sup_{d< \sqrt{  n/ (\a \log n)}}\Big|\P\big\{d|B_n(\varrho) \}- \frac{1}{d}\Big|\le 
 n^{-    8  \a
   \varrho(1-\varrho)  },  
\end{equation} and  for any $0<\b<1 $,   
\begin{equation}\label{unifdivro} \sup_{d<   n^{(1-\b)/2} }\big|\P\big\{  d|   B_{ n} \big\}-{1\over d} 
\big|\le e^{- 8   
   \varrho(1-\varrho) n^\b}. 
\end{equation}
\end{remarks} 

   \begin{remark} \rm    
Poisson summation formula: for   $x\in \R,\  0\le
\d\le 1
$, 
\begin{equation}\label{poisson}\sum_{\ell\in \Z} e^{-(\ell+\d)^2\pi x^{-1}}=x^{1/2} \sum_{\ell\in \Z}  e^{2i\pi \ell\d -\ell^2\pi x},
 \end{equation}  
 implies that   \begin{equation}\label{poisson.comp.theta.llt}
 \frac{\Theta(d,n)}{ d}  =\sqrt{  { 2\over \pi
n}}  \sum_{ 
z\equiv 0\, (d)} e^{-{ (2z-n)^2\over 2 n}}.
 \end{equation} 
 \end{remark}
 \begin{remark} \rm   Formulas \eqref{unif.est.theta}, \eqref{poisson.comp.theta.llt}
show that the general problem of estimating $\P\big\{d|B_n\big\}$  reduces to the one of   estimating the $d$ series
\begin{equation}\label{vd.div.series}
\sum_{j }e^{-\frac{(jd-\rho)^2}{2n}},
\end{equation}
the summation being either over   even rational numbers $j$, or over   odd rational numbers $j$,   $\rho$ denoting   the residue class modulo $d$ of $n$.

An indication of the behavior of these series when $d$ and $n$ {\it simultaneously} vary, is given by  the already sharp estimate, 
\begin{equation}\label{Theta.dmu} 
 \big|{\Theta (d,n)\over d}-{1\over d}\big| \ \le \ \begin{cases}   { C\over d}
 e^{ - {n \pi^2 \over 2d^2}} & \qq \text{   if}\quad d\le \sqrt n,\cr        {C \over\sqrt n}   & \qq \text{  if}\quad \sqrt n\le d\le n.
 \end{cases}
 \end{equation}

We also quote the classical bound in the study of   inclusion theorems related to  summation methods. 
\begin{lemma} [\cite{H},\,Th.\,140] \label{exp.square}
\begin{equation}\sum_{k\in \Z} e^{-ck^2/n} =\sqrt{\Big(\frac{n\pi}{c}\Big)} +\mathcal O(1)
\end{equation}
when $n\to \infty$, uniformly in any interval of positive $c$.
\end{lemma}
Indeed by applying Poisson summation formula    with $x=\frac{\pi n }{c}$ and $\d=0$, \begin{equation} \sum_{k\in \Z} e^{-k^2\pi x^{-1}}=x^{1/2}   \sum_{k\in \Z}e^{-k^2\pi x}=\sqrt{\Big(\frac{n\pi}{c}\Big)} +\sum_{|k|\ge 1 } e^{-k^2n /c}  
    \end{equation} 
 We next use the  estimate below. Let $a$ be any positive real. Then,  
\begin{equation}\label{(2.17)} \sum_{H=1}^\infty e^{-aH^2}\le \ \ \begin{cases}   3e^{-a} & \qq {\rm  if}\quad \  a\ge 1,\cr
       {3 /\sqrt a}   & \qq {\rm  if}\quad  \ 0<a\le 1 .
       \end{cases}
       \end{equation}
Here $a=n/c\ge 1 $ for $n$ large, in which case
\begin{equation}\sum_{k\in \Z} e^{-ck^2/n}=\sqrt{\Big(\frac{n\pi}{c}\Big)} +\mathcal O(  e^{-n/c}), 
    \end{equation} 
 thereby implying the claimed estimate with a much better remainder.
 This   is however too weak to imply estimate \eqref{Theta.dmu} in which  $d$   and $n$ may vary simultaneously.
As to \eqref{(2.17)}, this is elementary since \begin{eqnarray*}
\sum_{H=1}^k e^{-aH^2}&\le &e^{-a } +\sum_{H=2}^k e^{-a H^2 }  \le  e^{-a } + \int_{1}^\infty
  e^{-au^2}du
 \cr 
   (u=v/\sqrt{2a})\qq &  =  & e^{-a } + {1\over \sqrt{ 2a}}\int_{\sqrt{ 2a}}^\infty
  e^{-v^2/2}dv \cr &\le  &e^{-a }\Big(1 + \sqrt{{\pi\over   2a }} \Big),  \end{eqnarray*}
 where   we used the Mill's ratio bound $ e^{ x^2/2} \int_x^\infty e^{-t^2/2}\, dt  \le \sqrt{{\pi\over 2}}     $,
 valid for all $x\ge 0$. Thus $\sum_{H=1}^\infty e^{-aH^2}\le e^{-a } (1 + \sqrt{{\pi/  2  }}  )\le    3e^{-a}$   if  $a\ge 1$, and
$\sum_{H=1}^\infty e^{-aH^2}\le   (1/ \sqrt a )\big(1 + \sqrt{{\pi/   2  }} \big)\le      3/ \sqrt a  $,  if  $a\le 1$.  
\end{remark} 

\vskip 4 pt
\begin{remark}\rm
A study of the special case   of square-free divisors  is made in Shi-Weber \cite{Sh.We1}.
 \end{remark} 

\medskip \par
  In the same 2007's paper \cite{W1}, the distribution of the smallest prime divisor of $B_n$ was further studied.  Some notation: let $(a,b)$ denote the greatest common divisor of the integers $a$ and $b$. Recall that    $P^-(m )$ denotes    the smallest
prime divisor  of the integer $m>1$ (by convention $P^-(1)=+\infty$). We refer to   Tenenbaum \cite{Te}, III.6, for a study of integers free of   small divisors, and to the deep and instructive survey of Hildebrand and Tenenbaum \cite{HT}, and III.5 in the above, concerning    integers free of large divisors. 

More precisely, we    derived   from Theorem \ref{estPdlBn}  and elementary Erathost\`ene's sieve,   sharp estimates of  the probability $\P \{ P^-( B_n ) > \zeta  \}$,  
providing   the right central comparison term  together with an already   satisfactory and exploitable error term. 
  \begin{theorem}[\cite{W1}, Th.\,I]\label{P-Bn}
  There exist a positive real $c>0$ and   constants $C_0$, $\zeta_0$ such that for $n$ large
enough,   we have the
 following estimate
$$
\Big|\hskip1pt\P\big\{ P^-( B_n ) > \zeta \big\}-     { e^{-\gamma} \over  \log
\zeta }\,\Big|  \le  { C_0 \over   \log^2 \zeta } \qq \qquad (\,\zeta_0\le  \zeta\le n^{c/\log\log n}\,) $$
 where    $\g$ is  Euler's constant.
\end{theorem}
\begin{remark}\rm Sharpenings of Theorem \ref{P-Bn} are the object of a separate work.
\end{remark} 
We illustrate this on
some examples.

 
 \vskip 3 pt  \noi (1) {\it Estimate of $\P\{ B_n\,\hbox{  prime} \}$.}
We begin  with clarifying   the link between this probability and the Prime Number Theorem (PNT).
We first remark that
$$\P\Big\{ \big|B_n-\frac{n}{2}\big|>\frac12 \sqrt{3n\log n}\Big\}\ll \frac{1}{n}.$$
 Next by the local limit theorem
 there exists a numerical constant $C_0$ such that for all positive $n$
 \begin{eqnarray*}  \sup_{k}\, \Big|  {\mathbb P}\big\{\mathcal B_n=k\}
 -\sqrt{\frac{2}{\pi n}} e^{-{ (2k-n)^2\over
2 n}}\Big|\ \le \
  \  \frac{C_0}{n^{3/2}}  .
  \end{eqnarray*}
 Accordingly, 
 $$\P\{ B_n\,\hbox{  prime} \}=\sqrt{\frac{2}{\pi n}} \sum_{p\in\mathcal P\cap [\frac{n}{2}-\frac12 \sqrt{3n\log n},\frac{n}{2}+\frac12 \sqrt{3n\log n}]}   e^{-{ (2p-n)^2\over
2 n}}+\mathcal O\Big(\frac{1}{n}\Big) .$$
  The summation index is of type $\mathcal P\cap [x, x+y]$, $y\asymp \sqrt{x\log x}$; thus its cardinality relies upon extensions of the PNT  and cannot in turn be estimated.
  The best   unconditional estimate (slightly improved by Heath-Brown) is due to Huxley who showed in \cite{Hu}    that the PNT extends to intervals of the type $[x,x+x^\t]$,   $x^{\frac{7}{12} }  \le \t\le \frac{x}{(\log x)^4}$, namely
\begin{equation*}
\#\{[x,x+x^\t]\cap \mathcal P\} \sim  {x^\t}/{ \log x }.
\end{equation*} 
  Assuming   RH, the best result known to us states as follows,
\begin{equation}\label{RH.pi}\pi(x)-\pi(x-y)\, =\, \int_{x-y}^x \frac{\dd t}{\log t}+ 
\mathcal O\Big( x^{\frac{1}{2}}\log \Big\{ \frac{y}{x^{\frac{1}{2}}\log x}\Big\}\Big)
\end{equation}
for $y$ in the range 
$ 2x^{\frac{1}{2}}\log x \le y\le x$. Hence for $M\ge 2$ fixed,
\begin{equation}\pi(x)-\pi(x-Mx^{\frac{1}{2}}\log x)
\sim\,  x^{\frac{1}{2}} \big\{M+ 
\mathcal O\big(  \log M\big)\big\}.\end{equation}

See Heath-Brown \cite{HB}.
The interested reader will fruitfully refer to Montgomery and Soundararajan \cite{MS},\,p.\,595, also to Montgomery and Vaughan \cite{MV}, a model of   workbook. Accordingly any  result based on an assumption on the size's order of $\P\{ B_n\,\hbox{  prime} \}$ like for instance $\ge ( \log n)^{-\a}$, $\a>0$  -which has an easy interpretation on the gap between consecutive primes- is of no peculiar interest.
\vskip 3pt On the other hand, it follows from Remark \ref{rem.Snprime.M}   that  for some constant $C>0$,
\begin{equation}\label{bd1}(\log n) \, \P\{ B_n\,\hbox{prime} \}\ge C,
\end{equation}
for all   integers $n$ in a  set  of density 1.  \bigskip \par

 By using Theorem \ref{P-Bn}, we get  the following remarkable compact formula with sharp rate of convergence. 
\begin{theorem}\label{eta(n;y)Bprime}
 \begin{equation*}\P\{ B_n\,\hbox{\rm prime} \} \,=\,  \sum_{d\le n\,,\,  P^+(d )\le \sqrt n}{  \m(d)\over d}\Theta(d,n)+ {\mathcal O}\Big({ \log^{5/2} n  \over n}\Big).\end{equation*} 
 Further
\begin{equation*}\P\{ B_n\,\hbox{\rm prime} \} \,=\, \sqrt{  { 2\over \pi
n}}  \sum_{ k\in\Z}  \Big( \sum_{d\le n\atop  P^+(d )\le \sqrt n}   \m(d) \,e^{-{ (2kd-n)^2\over 2 n}}\Big)+ {\mathcal O}\Big({ \log^{5/2} n  \over n}\Big).\end{equation*} 
\end{theorem}  
We first prove
\begin{proposition} \label{p.eta(n;y)B}  For any positive integer $n$ and any positive real $y$,
 we have 
  \begin{equation}\label{eta(n;y)B.est}\P\{P^-(B_n )> y\} \,=\,  \sum_{ d\le n\,,\, P^+(d )\le y}{  \m(d)\over d}\Theta(d,n)+ {\mathcal O}\Big( { y\,\log^{5/2} n  \over n^{ 3/2}}\, 
e^{-\frac{\log n}{\log y}/2 }
\Big).
\end{equation} 
  \end{proposition}
 
We need a Lemma.
 \begin{lemma} \label{l.eta(n;y)B}For any positive integer $n$ and any positive real $y$,
 we have
 \begin{equation*} \P\{P^-(B_n )> y\} =\sum_{  P^+(d )\le y} \m(d)\P\{ d|B_n \}
.
\end{equation*}
 \end{lemma}

\begin{proof} Recall that the characteristic function of the set of integers $n$ such that $P^-(n )> y$,
\begin{equation*}\eta(n;y) =\begin{cases} 1 \quad & \ \hbox{ if $P^-(n )> y$} \cr
0 \quad & \ \hbox{ if $P^-(n )\le  y$}.
\end{cases}
\end{equation*}
is multiplicative. Thus by the   M\"obius inversion formula
\begin{equation}\label{eta(n;y)}\eta(n;y) =\sum_{d|n\atop P^+(d )\le y} \m(d) \qq \qq (n\ge 1).
\end{equation}
 See  Tenenbaum \cite[p.\,395-396]{Te1}. 
  Therefore by integration the Lemma follows.   \end{proof}

\begin{proof}[Proof of Proposition \ref{p.eta(n;y)B}] A direct consequence of  Lemma \ref{l.eta(n;y)B} and Estimate  \eqref{unif.est.theta} of  Theorem \ref{estPdlBn} is  that for $n\ge y\ge 2$,
 \begin{equation*} \P\{P^-(B_n )> y\} \,=\,  \sum_{ d\le n\,,\, P^+(d )\le y}{  \m(d)\over d}\Theta(d,n)+ {\mathcal O}\Big(\sum_{ d\le n\,,\, P^+(d )\le y}{ | \m(d)|\over d}\,{ \log^{5/2} n  \over n^{ 3/2}}\, 
  \Big).
\end{equation*} 
\vskip 3 pt Now let $\Psi(x,y)= \#\{m\le x:P^+(m)\le y\}$. By Theorem 1 p.\,359 in \cite{Te1}, we have 
\begin{eqnarray*}
\sum_{ d\le n\,,\, P^+(d )\le y} |\m(d)|\,\le \, \Psi (n,y) \,\ll \, y\,e^{-(\frac{\log n}{\log y})/2}\qq \quad (n\ge y\ge 2).\end{eqnarray*}
Therefore 
\begin{equation*} \P\{P^-(B_n )> y\} \,=\,  \sum_{ d\le n\,,\, P^+(d )\le y}{  \m(d)\over d}\Theta(d,n)+ {\mathcal O}\Big(y\,{ \log^{5/2} n  \over n^{ 3/2}}\, 
e^{-\frac{\log n}{\log y}/2 }
\Big).
\end{equation*} 
\end{proof}


\begin{proof}[Proof of Theorem \ref{eta(n;y)Bprime}]It follows from Proposition \ref{p.eta(n;y)B} by taking $y=\sqrt n$, and  since $B_n$ is prime  if and only if $P^-(B_n )> \sqrt n$, that 
\begin{equation*}
\P\{ B_n\,\hbox{  prime} \} \,=\,  \sum_{ d\le n\,,\, P^+(d )\le \sqrt n}{  \m(d)\over d}\Theta(d,n)+ {\mathcal O}\Big( {  \log^{5/2} n  \over n }  
 \Big).
\end{equation*} 
   The proof is completed by  applying \eqref{poisson.comp.theta.llt}.
\end{proof}
 

\vskip 9 pt

  Let $a$ be real and  $\s_{a}(n)= \sum_{d|k} d^a  $.  In a similar way we get for $a=-1$,
  \begin{theorem}\label{sigma-1.Bn}
\begin{equation*}\E \s_{-1}( B_n) \,=\,   \sqrt{  { 2\over \pi
n}} \sum_{d\le n}  \frac1d \sum_{ 
z\equiv 0\, (d)} e^{-{ (2z-n)^2\over 2 n}} 
+ {\mathcal O}\Big({ \log^{3} n   \over n^{3/2}}\Big) .\end{equation*} 
\end{theorem}

 \vskip 4 pt \noi (2) {\it Estimates related to Erd\"os and  Zaremba function.}
Let $\Phi(n)=\sum_{d|n} \frac{\log d}{d}$.  This function appears in the study of good lattice points in numerical integration. Erd\"os and  Zaremba  \cite{EZ} showed that
$$ \limsup_{n\to \infty} \frac{\Phi(n)}{(\log\log n)^2}=e^\g,$$ $\g$ being Euler's constant. The proof
   is based on the identity
 \begin{eqnarray}\label{formule}\Phi(n)
=\sum_{i=1}^r \sum_{\nu_i=1}^{\a_i}\frac{\log p_i^{\nu_i}}{p_i^{\nu_i}}\sum_{\d|n p_i^{-\a_i}}\frac{1}{\d} , \qq\qq ( n= p_1^{\a_1}\ldots p_r^{\a_r}),
\end {eqnarray}

 Consider  the function 
 $$\Psi(n)= \sum_{d|n} \frac{(\log d )(\log\log d)}{d} .$$  In this case  a  formula similar to \eqref{formule}   no longer holds, the "log-linearity" being  lost due to the extra factor $\log\log d$. 
We showed in the recent work  \cite{W11}
that
  \begin{eqnarray*}  \limsup_{n\to \infty}\, \frac{\Psi(n)}{(\log \log n)^2(\log\log\log n)}\,=\, e^\g.\end{eqnarray*}
  
This required  a new approach, and the proof is considerably more complicated. 
  \begin{corollary}
 $$  \E \Psi(B_n)= \sum_{d\le n} \frac{(\log d )(\log\log d)\Theta(d,n)}{d^2} + {\mathcal O}\Big(\frac{(\log n)^{7/2}(\log\log n)}{n^{ 3/2}}
   \Big)$$
  \end{corollary}
  \begin{proof}Writing $\sum_{d| B_n} =\sum_{d\le n\,,\,d|B_n}$ we get,
$$\E \Psi(B_n)=\sum_{d\le n}  \frac{(\log d )(\log\log d)}{d}\big(\P\big\{d|B_n\big\}-{\Theta(d,n)\over d} +{\Theta(d,n)\over d}\big).$$
 Whence by Theorem \ref{estPdlBn},$$  \E \Psi(B_n)= \sum_{d\le n} \frac{(\log d )(\log\log d)\Theta(d,n)}{d^2} + {\mathcal O}\Big(\frac{(\log n)^{7/2}(\log\log n)}{n^{ 3/2}}
 \Big)
  . $$
 
\end{proof}

 
\subsection{A  way back along Pascal's matrices.}\label{sub.bin.way.back}
 Let $f(j)$ be an arithmetical function. Consider the column vector ${\boldsymbol \nu}=\big(f(1),f(2),\ldots, f(n)\big)$. Let $P_n$ be the $n$-size   lower triangular Pascal matrix:
\begin{eqnarray*}P_n={ \left[\begin{matrix}
1        \cr
1   &1       \cr
1        &2            &1    \cr
1       &3       &3 &1       \cr
 \vdots         &   \vdots                  &\vdots   &\vdots   &\ddots    
 \cr {{n}\choose{1}}& {{n}\choose{2}}&{{n}\choose{3}}&{{n}\choose{4}}  &\cdots  &{{n}\choose{n}}\end{matrix} \right]}
\end{eqnarray*}
Then $P_n{\boldsymbol \nu}=  \,{\boldsymbol w}$ where ${\boldsymbol w}=\big(\E f(B_1),2\E f(B_2),\ldots, 2^n\E f(B_n)\big)$. Call and Velleman (\cite{CV},\,Th.\,1)  
 identified $P^{-1}_n$, which is $DP_nD^{-1}$ with $D$ as indicated:
\begin{eqnarray*}P^{-1}_n={ \left[\begin{matrix}
1        \cr
-1  & 1       \cr
1        &-2            &1    \cr
-1       &3       &-3 &1       \cr
 \vdots         &   \vdots                  &\vdots   &\vdots   &\ddots    
 \cr {{n}\choose{1}}& {{n}\choose{2}}&{{n}\choose{3}}&{{n}\choose{4}}  &\cdots  &{{n}\choose{n}}\end{matrix} \right]}
\end{eqnarray*}
\begin{eqnarray*}D={ \left[\begin{matrix}
1        \cr
  &-1       \cr
       &{}      &1    \cr
  &{}          &{}         &-1       \cr
 &{}          &{}   &{}          &{}  &\ddots    
 \cr &{}          &{}   &{}          &{}    &   &(-1)^{n+1}\end{matrix} \right]}
\end{eqnarray*}
 So that knowing ${\boldsymbol w}$  we have -at least theoretically- a way back to know ${\boldsymbol \nu}$. In  condensed form     
$P_n^{-1}=Q_n$,  where  
$Q_n$ is the $n\times n$ matrix defined by   \begin{equation}Q_n(i,j)=\begin{cases} (-1)^{i-j}  {{i-1} \choose{j-1}} \quad&\hbox{if $i\ge j$,}\cr
 0 \quad&\hbox{otherwise.}
 \end{cases}
 \end{equation}  
   \begin{proposition}\label{way.back}Let $f(j)$ be an arithmetical function. Then
\beq\label{} f(i)=\sum_{j=1}^i  (-1)^{i-j}2^j{{i-1}\choose{j-1}}\E f(B_j),\qq\quad i\ge 2.
\eeq \end{proposition}
  


 \subsection{Extension to all residue classes mod\,($d$).}  
     Theorem \ref{estPdlBn} was recently  extended to all residue classes mod\,($d$) in   Weber \cite{W9} with application to divisors in iid square integrable random walks.
   The  estimate   obtained  is   uniform over  all residue classes. 
 
 \begin{theorem}[\cite{W9}, Th.\,3.1]\label{estPdlBn.u} 
There exist  two absolute constants $C$ and $n_0$ such that for all $n\ge n_0$,
\begin{equation*}
 \sup_{u\ge 0}\,\sup_{d\ge 2}\Big|\P\big\{  d|   B_{ n}+u \big\}-  {1\over d}\sum_{ 0\le |j|< d }
e^{i\pi (2un){j\over d}}\  e^{  -n  
{\pi^2j^2\over 2d^2}}\Big|\le C\, (\log n)^{5/2}n^{-3/2}.
\end{equation*}

\end{theorem}

  Put 
 \begin{equation}\label{theta.u.}
\Theta_u(d,n)  =  \sum_{\ell\in \Z}  e^{i\pi (2u+n){j\over d}}\  e^{  -n  
{\pi^2j^2\over 2d^2}}.
 \end{equation}
 Note that $\Theta_0(d,n)=\Theta(d,n)$.
As a corollary we get,
\begin{corollary}\label{cor.estPdlBn.u} 
For some absolute constant $C$, we have  
\begin{equation*} \sup_{u\ge 0}\, \sup_{d\ge 2}\Big|\P\big\{  d|   B_{ n}+u \big\}-{\Theta_u(d,n)\over d} \Big| \le C \,(\log
n)^{5/2}n^{-3/2} ,
\end{equation*}
for all $n\ge n_0$.
\end{corollary}
    The proof is based on the lemma below. 

 \begin{lemma}\label{reduction.sym} For any integers $d\ge 2$, $n\ge 2$ and $u\ge 0$, 
 \begin{eqnarray*}
  \P\big\{ d| B_n+u\big\}&=& {1\over d} + {2\over d}\sum_{1\le j< d/2} \cos\big(\pi (2u+n){j\over d} \big)\big(\cos { \pi j\over d}\big)^{n}
.
\end{eqnarray*}
 \end{lemma}

\section{Divisibility properties in  i.i.d.   random walks.} \label{s4}
    Theorem \ref{estPdlBn} extends to     
i.i.d.  square integrable random walks.  \begin{theorem}[Weber \cite{W9},\,Th.\,2.1]\label{uniform.semi.local.limit.theorem} Let $X  $  be  a square integrable  random variable
      taking values in the  lattice $\mathcal L(v_0,1)=\big\{v_k=v_0+ k,k\in \Z\big\}$, and let $f(k)=\P\{X=v_k\}$, for each $k$.   
Assume that 
   \beq \label{bp}\t_X   :=\sum_{k\in \Z}\min(f(k), f(k+1))>0.
   \eeq  
   Let $  X_i$, $i\ge 1 $  be  independent, identically distributed random variables having same law than $X$, and let $S_n=\sum_{j=1}^nX_j$, for each $n$.  
\vskip 2 pt   Let $ \bar \m= \{ \m_k, k\in \Z\}$ be a   sequence of non-negative reals such that  $0<\m_k < f(k) $, if $f(k)>0$,   $\m_k=0$ if $f(k)=0$,    let  $ \m= \sum_{k\in \Z}\m_k$, and assume that $1- \m< \t_X$.   
   Let also non-negative reals $\bar \tau=  \{ \tau_k, k\in \Z\}$ be  solutions of the (solvable) equation:   
 $ {\tau_{k-1}+\tau_k\over 2} =f(k)  - \m_k,
 $ for all $k \in \Z$.  
Let   $\t = \sum_{k\in \Z}  \tau_k = 1- \m.$ Further let $s(t) =\sum_{k\in \Z} \m_k\, e^{ 2i \pi   v_kt}$ and    $\rho$ be such  that  $1-\t<\rho<1$.   
  \vskip 1 pt      Let $\widetilde X$ be the   random variable associated  to $X$  and $\bar \m$, and defined by the relation
 $ \P\{ \widetilde X =v_k\}= {\tau_k}/{\t}$, for all  $k\in \Z   $.
 Then the following results hold:
\vskip 2 pt {\rm (i)}
There exists  
$\theta=\theta(\rho,\t)$ with  $ 0< \theta <\t$,  
$C$  and $N$ 
 such that for $  n \ge N$,
\begin{align*}   \sup_{u\ge 0}\,\sup_{d\ge 2} \Big| \P \{ d|S_n+u  \}   -  {1\over d}-  {1\over d}\sum_{ 0< |\ell|<d }& \Big( e^{ (i\pi {\ell\over   d }-{ \pi^2\ell^2\over 2 d^2}) } \t\,
 \E \,e^{2i \pi {\ell\over   d }\widetilde X  }   +s\big(  {\ell\over   d }\big)\Big)^n \Big|
  \cr  &\le    \frac{C }{ \theta ^{3/2}}\   \frac{(\log  n)^{5/2}}{ n^{3/2}}+2\rho^n.
\end{align*}
 \vskip 2 pt   {\rm (ii)} Let   $\mathcal D$ be a test set of divisors $\ge 2$,   $\mathcal D_\p$ be the section of $\mathcal D$ at height $\p$  and $|\mathcal D_\p|$ denote its   cardinality. Then,
    \begin{eqnarray*}
      \sum_{n=N}^\infty \   \sup_{u\ge 0} \, \sup_{\p\ge 2}\, {1\over |\mathcal D_\p |} \sum_{d\in \mathcal D_\p } \,\Big| \P \{d|S_n+u  \}   - {1\over d}\Big|
    & \le  &    \frac{C_1}{\t}
 \,  
         +  \frac{C_2 }{ \theta ^{3/2}} +\frac{2\rho^2}{1-\rho},
\end{eqnarray*}
where $C_1=\frac{2e^{     { \pi^2/ 4 }  }}
     {(1-e^{   -    {  \pi^2/ 16} }   ) }$, $C_2= C\,  \sum_{n=N}^\infty  \frac{(\log {n})^{5/2} }{ ({n})^{3/2}}$\,.
     \end{theorem} 
    Condition \eqref{bp}  defines an already very large class of random variables. A random variable $X$ is said to have a {\it Bernoulli part} if  condition \eqref{bp} is satisfied.  That notion is due to  McDonald \cite{M}, who introduced an ingenious coupling method  in the study of a famous limit theorem: the local limit theorem  which has  interface  with Number Theory.
    This method is particularly relevant in our proof. Concretely, let $  \{ \tau_k, k\in \Z\}$     of   non-negative reals such that
\begin{equation}\label{basber0}  \tau_{k-1}+\tau_k\le 2f(k), \qq  \qq\sum_{k\in \Z}  \tau_k =\t.
\end{equation}
For instance  
 $\tau_k=  \frac{\t}{\nu_X} \, \min(f(k), f(k+1))  $ is suitable.
   Define   a pair of random variables $(V,\e)$   as follows:
  \begin{eqnarray}\label{ve} \qq\qq\begin{cases} {\mathbb P}\{ (V,\e)=( v_k,1)\}=\tau_k,      \cr
 {\mathbb P}\{ (V,\e)=( v_k,0)\}=f(k) -{\tau_{k-1}+\tau_k\over
2}    .  \end{cases}\qq (\forall k\in \Z)
\end{eqnarray}
 Then we have the following decomposition which iterates well to iid sums $S_n$.    \begin{lemma} \label{bpr} Let $L$
be a Bernoulli random variable    which is independent of  $(V,\e)$, and put
 $Z= V+ \e  L$.
We have $Z\buildrel{\mathcal D}\over{ =}X$.
\end{lemma}

\begin{remark}[{\tf Reduction of binomial case   to  standard Bernoulli case.}] \rm Obviously binomial random variables have a Bernoulli part.  For  binomial random variables, the above decomposition is   direct. 
Let  $\b $  be a binomial random variable with  $\P\{\b  =1\}= \a = 1-\P\{\b  =0\}$. Let also $\varsigma$  be standard Bernoulli random variable  ($
\P\{\varsigma  =1\}= 1/2=\P\{\varsigma  =0\}$). 

Assume $0<\a<1/2$.  Let $\e $    be such that    $\P\{\e  =1\}= 2\a = 1-\P\{\e  =0\}$ and $\b, \e, \varsigma$ are independent. Trivially  $\e\varsigma \buildrel{\mathcal L}\over{=}\b$. 

 If now $1/2<a<1$, let $\tau_0 $
be  verifying
$0<  \tau_0< 2(1-\a)=2\min (\a,1-\a)$. 
Define a pair of random variables
$(V,\e)$   as follows.  
\begin{eqnarray*}   \begin{cases} \P\{ (V,\e)=(1,1)\}=0  
      \cr   \P\{ (V,\e)=(1,0)\}=\a -{\tau_0 \over
2}    . \end{cases}  \qq \begin{cases}   \P\{ (V,\e)=(0,1)\}=\tau_{0 } 
      \cr 
 \P\{ (V,\e)=( 0,0)\}=1-\a -{ \tau_0\over
2}    . \end{cases} 
\end{eqnarray*}
  Let $ \varsigma$ be  independent from $(V,\e)$.     Then 
$V+\e  \varsigma\buildrel{\mathcal L}\over{=}\b$.  

\end{remark}
  
   Consequently:     
    {\it      The study of the  limit properties (CLT,\, LLT,\,...) of the sequence
$S_n= k_1\b_1+
\ldots + k_n\b_n$, $n\ge 1$,  where 
$\b_j$ are independent binomial random variables, and   $\{k_j, j\ge 1\}$   a given increasing sequence of positive integers,   reduces to the case when these random variables are   i.i.d      standard Bernoulli.}

     \section{Divisibility properties of  non i.i.d. square integrable random walks.} \label{s5}
      Let $X=\{X_i , i\ge 1\}$ be  a sequence of  independent  variables taking values in $\Z$, 
and let $S_n=\sum_{k=1}^n X_k$, for each $n$. 
To our knowledge there is no study of divisors, in particular of their value distribution, in general independent   random walks. However divisibility is present in the study   the local limit theorem.  The sequence $X$ is  \emph{asymptotically uniformly distributed}, in short a.u.d., if   for any $d\ge 2$ and $m = 0,1,\ldots,d-1$, we
have
 \begin{equation} \label{aud1}\lim_{n\to \infty}\ \P\{S_n \equiv m \,{\rm (mod)}\,d\}=\frac1d.
  \end{equation}
 
  Let us assume   that the random variables $X_i$ take values in a common integer lattice $\mathcal L(v_{0},D )$, namely defined by the
 sequence $v_{ k}=v_{ 0}+D k$, $k\in \Z$,   $D >0$, and are  square integrable,  
and let     
\begin{equation} \label{MnBn}M_n= {\mathbb E\,} S_n , \qq B_n^2={\rm Var}(S_n)\to \infty.
\end{equation}
 
  We say  that the local limit theorem (in the usual form) is \emph{applicable} to   $X$    if
 \begin{equation}\label{def.llt.indep}    \sup_{N=v_0n+Dk }\Big|B_n\, {\mathbb P}\{S_n=N\}-{D\over  \sqrt{ 2\pi } }e^{-
{(N-M_n)^2\over  2 B_n^2} }\Big| = o(1), \qq \quad n\to\infty.
\end{equation} 


By   well-known Rozanov's result \cite{Ro}, 
the local limit theorem is applicable to   $X$  {\it only} if $X$ satisfies the a.u.d. property.

 

   \vskip 10pt     In a recent work     based on   Poisson summation formula, 
the stronger result    is proved,  establishing an explicit link between the local limit theorem and  the a.u.d. property, through a     quantitative   estimate of  the difference ${\mathbb P} \{ S_n\equiv\,  m\! \hbox{\rm{ (mod $h$)}} \}- {1}/{h}$.

\begin{theorem}[Weber \cite{W10},\,Th.\,1.4] \label{l1a}    Let $X=\{X_i , i\ge 1\}$ be  a sequence of  independent  variables taking values in $\Z$,  
and let $S_n=\sum_{k=1}^n X_k$, for each $n$. Assume that, 
\vskip 3 pt {\rm ($\mathcal A$) \it For some function $1\le \phi(t)\uparrow \infty $ as $t\to \infty$, and some constant $C$, we have for all $n$} 
 \begin{equation}\label{phi.cond}
\sup_{m\in \Z}\Big|B_n\P\big\{ S_n=m\big\}- {1\over   \sqrt{ 2\pi } }\ e^{-
{(m-M_n)^2\over  2 B_n^2} }\Big|\,\le \,  {C\over \,\phi(B_n)}.
 \end{equation}
\vskip 3 pt   Then there exists a numerical constant $C_1$, such that for all $0<\e \le 1$, all  $n $ such that $B_n\ge 6$, and all $h\ge 2$, 
\begin{align*}
\sup_{\m=0,1,\ldots, h-1} \,&\Big|{\mathbb P}\big\{ S_n\equiv\, \m\ \hbox{\rm{ (mod $h$)}}\big\}- \frac{1}{h}\Big|
\cr   &\le   {1\over \sqrt{2\pi}\,  B_n   }+\frac{2C}{h\,\sqrt{\e}\,\phi(B_n)}
+  {\mathbb P}\Big\{   \frac{|S_n -M_n |}{B_n}> \frac{1}{\sqrt \e}\Big\}+C_1 \,e^{-1/(16\e)}.
\end{align*}
\end{theorem}
    It follows from the proof that $C_1=2e\sqrt{\pi}$ is suitable.
    Choosing $\e= \phi(B_n)^{-2/3}$ and using Tchebycheff's inequality, we get the following
 \begin{corollary}\label{}For all   $n $ such that $B_n\ge 6$, and all $h\ge 2$, 
we have 
\begin{align}\label{eps.phi}
\sup_{\m=0,1,\ldots, h-1} \,  \Big|{\mathbb P}\big\{ S_n\equiv\,  \m\ \hbox{\rm{ (mod $h$)}}\big\}- \frac{1}{h}\Big|
 \le H_n ,
\end{align}
with
\begin{align}\label{eps.phi.Hn}
   H_n=  {1\over \sqrt{2\pi}\,  B_n   }+\frac{1+ 2 {C}/{h}
}{ \phi(B_n)^{2/3} } 
+  C_1 \,e^{-(1/ 16 )\phi(B_n)^{2/3}}.
\end{align}
\end{corollary}
 Theorem \ref{l1a}  implies Rozanov's result,  since by  definition such a function $\phi$ exists   if the local limit theorem is applicable to   $X$.          
 
\bigskip\par
  Instead of assumption ($\mathcal A$), consider now the following assumption  
  \vskip 5 pt ($\mathcal B$)    {\it $\t_{X_j} =\displaystyle{\sum_{k\in \Z}}{\mathbb P}\{X_j= k\}\wedge{\mathbb P}\{X_j= k+1 \} >0$,  for each $j$, and 
 $\nu_n =\displaystyle{\sum_{j=1}^n} \t_j\uparrow \infty$}.
 
   \vskip 3 pt We have the following sharp results.
 \begin{theorem}[\cite{W10},\,Th.\,4.3]\label{saud1} Let $D=1$. Suppose that assumption ($\mathcal B$) is fulfilled.   Let $\a\!>\!\a'\!>\!0$, $0\!<\!\e\!<\!1$. Then for   each $n$ such that 
$$|x|\le\frac12 \sqrt{  \frac{ 2\a\log (1-\epsilon)\nu_n}{ (1-\epsilon)\nu_n }}\qq \Rightarrow \qq{\sin x\over
x}\ge (\a^\prime/\a)^{1/2},$$
   we have
 \begin{eqnarray*} \sup_{u\ge 0}\,\sup_{d<  \pi  \sqrt{   (1-\epsilon)\nu_n \over 2\a\log (1-\epsilon)\nu_n}} \ \Big| \P \{d|S_n+u  \}   -  {1\over d} \Big|
    &\le &2 \,e^{- \frac{\epsilon^2 }{2}\nu_n}+
 \,\big( (1-\epsilon)\nu_n\big)^{-\a'}  .
\end{eqnarray*}
\end{theorem}
   
 \begin{theorem}[\cite{W10},\,Th.\,4.6]\label{saud2} Let $D=1$. Suppose that assumption ($\mathcal B$) is fulfilled. Let $0<\rho<1 $ and   $0<\e<1$. Then for   each $n$ such that 
$$|x|\le\frac12 \,\sqrt{  \frac{ 2 }{ ((1-\epsilon)\nu_n)^{1-\rho} }}\qq \Rightarrow \qq{\sin x\over
x}\ge \sqrt{1-\e}$$
we have\begin{eqnarray*} 
\sup_{u\ge 0}\,\sup_{d<  (\pi/\sqrt 2) ((1-\e)\nu_n)^{(1-\rho)/2} }\ \big| \P \{d|S_n+u  \}   -  {1\over d} \big|  
   &\le  &  2 e^{- \frac{\epsilon^2 }{2}\nu_n}+e^{- ( (1-\epsilon)\nu_n)^\rho}
  .
\end{eqnarray*}
\end{theorem}
 \vskip 2 pt We refer the reader to \cite{W10} for the proofs of these results.
 We also refer to our  recent monograph   with    Szewczak \cite{SW} for more.

\section{Distribution of   divisors of ${ B_nB_m}$.}\label{s6}
This task is more complicated.
 A key function is  the  (multiplicative) function
\begin{eqnarray}\label{rho}\varrho_k(D):=\#\big\{1\le y\le D: \  D| y^2+ky \big\} ,\qq   k=0,1,\ldots
\end{eqnarray} 
 which  it is not so surprising in view of the elementary decomposition
$B_nB_m= B_n^2 +B_n(B_m-B_n)$. 
\vskip 2 pt
As  mentioned in section \ref{s3}, that study is used in   
  \cite{W14} in the  control of   the correlation function \eqref{delta.def}, more precisely in the case      $n<m\le n+n^c$, $c>0$, $c$ being small, where the following intermediate bound appears (hence also $\varrho_k(.)$),
\begin{eqnarray*}\P \{ 
d|B_n\, ,\, \d|B_m \}
& =&\P \{  d|B_n\, ,\, \d|B_m\, ,\,B_m-B_n=0 \}
 +\P \{ 
d|B_n\, ,\, \d|B_m\, ,\,B_m-B_n>0 \}
\cr & =&2^{-(m-n)}\,\P \{  d|B_n\, ,\, \d|B_n \}  
 +\P \{ 
d|B_n\, ,\, \d|B_m\, ,\,B_m-B_n>0 \}
\cr &\le& 2^{-(m-n)}\,\P \{ [d, \d]|B_n  \}
+\P \{ 
d\d|B_n B_m\, ,\,B_m-B_n>0 \}.
\end{eqnarray*} 
Note that $\P \{ 
d|B_n\, ,\, \d|B_m \}>0$, only if $m-n\ge (d,\d)$. 
 \vskip 3 pt  
  By using independence,  
  the characteristic function of $B_nB_m$ expresses as follows
$$\E e^{i\upsilon B_nB_m}={1\over 2^n2^{m-n}}\sum_{k=0}^{m-n} {{m-n}\choose {k}}\sum_{h=0}^n   {{n}\choose{h}} e^{i\upsilon (h^2+kh)}.$$
And so, we begin with 
\begin{eqnarray}\label{basicnm}\qq   \P \{ D|B_nB_m \} &=& \E\Big({1\over D} \sum_{j=0}^{D-1} e^{2i\pi  {j \over D}B_nB_m}\Big) 
\cr &=&  {1\over
2^{m-n}}\sum_{k=0}^{m-n}{{m-n}\choose {k}}{1\over D}\sum_{j=0}^{D-1}\Big({1\over
2^n}\sum_{h=0}^n   {{n}\choose{h}} e^{2i\pi{j\over D}  (h^2+kh)}\Big). 
\end{eqnarray}  

 \begin{theorem}\label{p4.5} For any $\e>0$, there exists a constant $C_\e$ depending on $\e$ only, such that  for any
 positive  integers
 $n,m, D$ with 
 $m>n$
\begin{equation}\label{4.11} \Big|\P \{ D|B_nB_m \}-  {1\over D  2^{ m-n}} \sum_{k=0}^{m-n} { {{m-n}\choose {k}}     } \varrho_k(D) 
\Big|
\le C_\e
\left({     D  ^{1 + \e}\over   n }\right)^{1/2 }. 
\end{equation} 
\end{theorem}
  The function    $\varrho_k(D)$ will be  made  explicit    in Theorem \ref{romult1}. A combination of   both results   will lead to  a useful bound of $\P \{ D|B_nB_m
\}$  stated in   Theorem \ref{bnmbound}. 
\vskip 2 pt
 We first prove Theorem \ref{p4.5}.  The main idea of the proof consists with   controlling the difference
$$\D=  {1\over
2^n }  \sum_{h=0}^n   {{n}\choose{h}} e^{2i\pi{j\over D}  (h^2+kh)}- {1\over n}\sum_{h=0}^{n-1}  e^{2i\pi{j\over D} 
(h^2+kh)},$$
 in which the first average   is the  last Euler sum   in (\ref{basicnm}), and  the second one,  its Ces\`aro means counterpart.  Let 
  $0\le
\ell
\le n$ be some integer,   
$\ell=LD+\m$,
$0\le \m< D$. Obviously,
\begin{eqnarray}\label{sumS} S  := {1\over \ell}\sum_{h=0}^{\ell-1}  e^{2i\pi{j\over D} 
(h^2+kh)}  & =& {1\over D}\sum_{h= 0}^{D-1}e^{2i\pi{j\over D} 
(h^2+kh)}  +   {1\over \ell}\sum_{h= 0}^{\m}e^{2i\pi{j\over D}  
(h^2+kh)} 
\cr & &  +L\big({1\over \ell}-{1\over LD}\big)\sum_{h= 0}^{D-1}e^{2i\pi{j\over D} 
(h^2+kh)}.
 \end{eqnarray}

It follows from (\ref{sumS}),    that 
  for   $j=1,\ldots, D$, 
\begin{eqnarray*} \Big|S-{1\over D}\sum_{h= 0}^{D-1}e^{2i\pi{j\over D} 
(h^2+kh)} \Big|  &\le  &  {1\over  \ell} \Big(\big|\sum_{h= 0}^{D-1}e^{2i\pi{j\over D} 
(h^2+kh)} \big|  +  \big|\sum_{h= 0}^{\m}e^{2i\pi{j\over D} 
(h^2+kh)}\big|\Big).
\end{eqnarray*} 
   We   rewrite this for   later use under the following form.
  \begin{lemma} \label{basicnm2}Let
$\ell\ge 1$, $\ell\equiv \m \, {\rm mod}(D)$ and
$0\le \m< D$.  For   $j=1,\ldots, D$, 
\begin{eqnarray*} \Big| \sum_{h=0}^{\ell-1}  e^{2i\pi{j\over D} 
(h^2+kh)}-{\ell\over D}\sum_{h= 0}^{D-1}e^{2i\pi{j\over D} 
(h^2+kh)} \Big|   
  & \le   &  \big|\sum_{h= 0}^{D-1}e^{2i\pi{j\over D} 
(h^2+kh)} \big|    +  \big|\sum_{h= 0}^{\m}e^{2i\pi{j\over D} 
(h^2+kh)}\big| .
\end{eqnarray*}  \end{lemma}
  Thus the study of $\D$ reduces to the one of the difference
$$ \D'= {1\over
2^n }  \sum_{h=0}^n   {{n}\choose{h}} e^{2i\pi{j\over D}  (h^2+kh)}- {1\over D}\sum_{h=0}^{D-1}e^{2i\pi{j\over D} 
(h^2+kh)} .
$$ 

 By  linearity of Euler summation, it suffices to have at disposal a bound for Euler sums, and we have a  classical tool from the theory of divergent series.  
\begin{lemma}  \label{Eulerbound}Let
$\{a_k, k\ge 0\}$ be a  sequence of reals. Put $A_\ell= \sum_{k=0}^\ell a_k$,  $\ell\ge 0$.
 There exists an   absolute constant $C$ such that for every positive integer $n$
\begin{equation} \label{4.6}\big| \sum_{h=0}^n   2^{-n}{{n}\choose{h}} a_h\big|\le {C\over \sqrt n}\max_{\ell=0}^n\big|A_\ell\big|
.\end{equation}
   \end{lemma}
  This is easily seen by using Abel summation and well-known property of the binomial distribution (Feller \cite{F}, p.\,151). (A  consequence is   Theorem 149 in Hardy \cite{H}.)
\vskip 4 pt
Applying   it  with the choice $$a_h=  e^{2i\pi{j\over D} (h^2+kh)}-
{1\over D}\sum_{y= 0}^{D-1}e^{2i\pi{j\over D}  (y^2+ky)}  \qq h=0,1,\ldots,n,$$
  we get in view of   Lemma \ref{basicnm2},
 \begin{eqnarray}
\D'&\le & { C\over \sqrt n}\max_{\m=0}^{D-1}  \big|\sum_{h= 0}^{\m}e^{2i\pi{j\over D} 
(h^2+kh)}\big|  . 
\end{eqnarray}
 
 Whence the   intermediate estimate.
 \begin{corollary} For some absolute constant $C$, 
 \begin{align}  \label{4.81}   \Big|\sum_{h=0}^n   2^{-n}{{n}\choose{h}} e^{2i\pi{j\over D} (h^2+kh)}-& {1\over D}\sum_{h=
0}^{D-1}e^{2i\pi{j\over D}  (h^2+kh)} \Big|
\cr &\le  {C\over \sqrt n}\bigg(\max_{\m=0}^{D'-1} \big|\sum_{h=0}^{\m} e^{2i\pi{j'\over D'} 
(h^2+kh)} 
\big| \bigg), 
\end{align} where  
$j'=  j/(j,D)
$,
$D'=D/(j,D)$.
 \end{corollary}
 The Lemma below  is just a slight reformulation of Sark\"osy's estimate \cite{Sa}. 
\begin{lemma}  Let  $\a$ be a real number and    $a,q$ be positive integers such that $(a,q)=1$ and
$|\a-a/q|<1/q^2$. Then, for any positive integer  $M$, 
\begin{eqnarray}\label{sarko}
\sup_{k\ge 0}\,\big|\sum_{x=1}^M e^{2i\pi
\a (x^2 +kx)}\big|^2&\le &\sum_{u=1-M}^{M-1}\Big|\sum_{y=\max(1-u,1) }^{\min(M,M-u)}e^{4i\pi \a uy}\Big|
\cr &\le &49
\Big\{{M^2\over  q}+(M\log q) + q\log q \Big\}.
\end{eqnarray} 
\end{lemma}

\begin{proof}  The last inequality in \eqref{sarko} is precisely what is established in the proof of Lemma 4 p.\,128 in \cite{Sa}.
Let $T=\sum_{x=1}^M e^{2i\pi
\a (x^2 +kx)}$. On expanding $|T|^2$, we get
\begin{eqnarray*}|T|^2  &=&
\sum_{y=1}^M\sum_{u=1-y}^{M-y} e^{2i\pi \a u(u+2y+  k )} =\sum_{u=1-M}^{M-1}\sum_{y=\max(1-u,1) }^{\min(M,M-u)}  e^{2i\pi \a u(u+2y+  k )}
\cr &\le &
\sum_{u=1-M}^{M-1}\Big|\sum_{y=\max(1-u,1) }^{\min(M,M-u)}  e^{4i\pi \a   uy }\Big|.
\end{eqnarray*}So that,
$$\sup_{k\ge 0}\Big|\sum_{x=1}^M\sum_{y=1}^M e^{2i\pi \a (x^2-y^2 +k(x-y))}\Big|^2\le 49
\Big\{{M^2\over  q}+(M\log q) + q\log q \Big\},$$
or
$$\sup_{k\ge 0}\Big|\sum_{x=1}^M e^{2i\pi
\a (x^2 +kx)}\Big| \le 7
\Big\{{M \over  \sqrt q}+\sqrt{ M\log q} + \sqrt{q\log q} \Big\}.$$
\end{proof}

 We can now pass to the proof of Theorem \ref{4.11}.
 \begin{proof}[Proof of Theorem \ref{4.11}]
Returning to \eqref{4.81}, we get  
\begin{eqnarray*}\sup_{k\ge 0}\big|\sum_{y=1}^{\m} e^{2i\pi{j'\over D'} 
(y^2+ky)} 
\big|&\le & 7\Big\{{\m \over  \sqrt{ D'}}+\sqrt{ \m\log D'} + \sqrt{D'\log D'} \Big\}
 \cr (\m \le
D')\quad & \le & 7  \Big\{ \sqrt{ D'}+2\sqrt{ D'\log D'} \Big\}
\cr   & \le& C_\e  (D')^{1/2+ \e}\le C_\e  (D )^{1/2+ \e}, 
\end{eqnarray*}
 where  
$j'=  j/(j,D)
$,
$D'=D/(j,D)$.
Thus
\begin{align}  \label{4.81a}   \Big|\sum_{h=0}^n   2^{-n}{{n}\choose{h}} e^{2i\pi{j\over D} (h^2+kh)}-& {1\over D}\sum_{h=
0}^{D-1}e^{2i\pi{j\over D}  (h^2+kh)} \Big|
\cr &\le  {C\over \sqrt n}\bigg(2+\max_{\m=1}^{D'-1} \big|\sum_{h=1}^{\m} e^{2i\pi{j'\over D'} 
(h^2+kh)} 
\big| \bigg)
  \le  {C_\e  (D )^{1/2+ \e}\over \sqrt n}.
\end{align} 
 
 Whence, 
\begin{equation}\label{4.9} \sup_{ 0\le j< D }\sup_{0\le k\le m-n}\Big|\sum_{h=0}^n   2^{-n}{{n}\choose{h}} e^{2i\pi{j\over D} (h^2+kh)}- {1\over D}\sum_{h=
0}^{D-1}e^{ i\pi{2j\over D}  (h^2+kh)} \Big|\le C_\e \left({     D  ^{1 + \e}\over   n }\right)^{1/2 } \! \!.
\end{equation}  
As
\begin{equation}\label{rho.} {1\over  D^2}\sum_{j=0}^{D-1} \sum_{h= 0}^{D-1}e^{ i\pi{2j\over D} 
(h^2+kh)}={\#\big\{1\le y\le D: D|y^2+ky\big\}\over D }
={\varrho_k(D)\over D } \!, 
\end{equation} 
it follows  
\begin{equation}\label{4.10}  \sup_{0\le k\le m-n} \bigg|{1\over D}\sum_{   j=0}^{D-1}\sum_{h=0}^n   2^{-n}{{n}\choose{h}} e^{2i\pi{j\over D} (h^2+kh)}
   -{\varrho_k(D)\over D } \Bigg|  \le C_\e \Big({     D  ^{1 + \e}\over   n }\Big)^{1/2 }\!\!
.
\end{equation} 
By combining now (\ref{4.10}) with (\ref{basicnm}), we obtain
$$ \Big|\P \{ D|B_nB_m \}-   \sum_{k=0}^{m-n} { {{m-n}\choose {k}} \over 2^{ (m-n)}  } {\varrho_k(D)\over D } 
\Big|
\le C_\e\,
\Big({     D  ^{1 + \e}\over   n }\Big)^{1/2 }. 
$$
\end{proof}

     \medskip \par

 It remains to compute   $\varrho_k(D)$.
  Introduce some   notation. Let   $v_p(d)$ denote the $p$-adic valuation of $d$, $p$ being a prime number, namely $p^{v_p(d)} || d$,   also \begin{eqnarray}\label{D.demi}D_{1/2} \,:=\,\prod_{p|D} p^{\lfloor {v_p(D)\over 2}\rfloor }. 
\end{eqnarray}

\begin{theorem} \label{romult1}   We have  
  \begin{equation*} \varrho_k(D)=\begin{cases}  \  \,
D_{1/2}\qq & \hbox{ if \ $k=0$,}\cr
\  \, 2^{\#\{p\,:\, 1\le v_p(k)<v_p(D)/2 \}}\, (k, D_{1/2})\qq &\hbox{ if \ $k\ge 1$.}
\end{cases}
\end{equation*}
     \end{theorem}
  
The proof is given in Appendix. 
We deduce
\begin{corollary}\label{rokest} For any positive integers  $D,k$,   
$$\varrho_k(D)\le 2^{\o ((k, D_{1/2}))}\, (k, D_{1/2})
 \qq
  {\rm and}\qq
\varrho_0(D)\le  
\sqrt D.
$$   
\end{corollary}  

\begin{proof}Immediate. 
\end{proof}
\begin{theorem} \label{bnmbound} We have for any positive integers $n,m,D$,  
\begin{eqnarray*}   \P \{ D|B_nB_m \} 
 &\le &{1  \over D }\sum_{d| D_{1/2} } 2^{\o (d)}d
 \Big(\sum_{u|  (D_{1/2}/d )}\m(u)\, \P\{ud|B_{m-n}\}\Big)
\cr &  &   +{1\over 2^{  m-n }\sqrt D } +C_\e
\left({     D  ^{1 + \e}\over   n }\right)^{1/2 }. 
\end{eqnarray*} 
   \end{theorem}  
\begin{proof} By Theorem \ref{p4.5},
 for any $\e>0$, there exists a constant $C_\e$ depending on $\e$ only, such that  for any
 positive  integers
 $n,m, D$ with 
 $m>n$
\begin{equation}\label{4.11a} \Big|\P \{ D|B_nB_m \}-  {1\over D  2^{ m-n}} \sum_{k=0}^{m-n} { {{m-n}\choose {k}}     } \varrho_k(D) 
\Big|
\le C_\e
\left({     D  ^{1 + \e}\over   n }\right)^{1/2 }. 
\end{equation} 
 By Corollary \ref{rokest}, 
\begin{eqnarray}\label{323}
& & {1\over D  2^{ m-n}}  \sum_{k=0}^{m-n}   {{m-n}\choose {k}} \,  \varrho_k(D)
\ = \  {\varrho_0(D)\over 2^{ 
m-n }D }+{1\over D  2^{ m-n}} 
\sum_{k=1}^{m-n}   {{m-n}\choose {k}}\,
 \varrho_k(D)  
\cr &  \le  &  {1\over 2^{  m-n }\sqrt D }+  {1  \over D 2^{  m-n }}  \sum_{k=1}^{m-n}   {{m-n}\choose {k}} 2^{\o ((k, D_{1/2}))}\, (k, D_{1/2})
\cr & =  &  {1\over 2^{  m-n }\sqrt D }+ 
 {1  \over D 2^{  m-n }}\sum_{d| D_{1/2}} 2^{\o (d)}d
 \sum_{k=1\atop (k, D_{1/2})=d}^{m-n}   {{m-n}\choose {k}}  
 .\end{eqnarray}
Let $\d$ be the arithmetical function defined by
\begin{eqnarray} \label{Delta}   \d(n)=\begin{cases}1\quad & {\rm if}\ n=1,\cr
0\quad & {\rm if}\ n \not=1. \end{cases} \end{eqnarray} Recall that 
\begin{eqnarray} \label{sp} \sum_{d|n}\m(d)=\d(n) . \end{eqnarray} 
 
 Consider first the sub-sum corresponding to $d=1$. We have  
\begin{eqnarray*}
 {1  \over  2^{  m-n }}
 \sum_{k=1\atop (k, D_{1/2})=1}^{m-n}   {{m-n}\choose {k}}
 &=&
 {1  \over   2^{  m-n }}
 \sum_{k=1}^{m-n}   {{m-n}\choose {k}}\d\big((k, D_{1/2})\big)
 \cr&=& 
 {1  \over D 2^{  m-n }}
 \sum_{k=1}^{m-n}   {{m-n}\choose {k}}\Big(\sum_{u|(k, D_{1/2} )}\m(u)\Big)
\cr&=& 
 \sum_{u|  D_{1/2} }\m(u)
 {1  \over   2^{  m-n }} \sum_{k=1\atop u|k}^{m-n}   {{m-n}\choose {k}}\cr&=& 
 \sum_{u|  D_{1/2} }\m(u)\,\big(\P\{u|B_{m-n}\}-{1  \over  2^{  m-n }}
\big)\end{eqnarray*}
Now let $(k, D_{1/2})=d>1$. Let $k=\k d$.
Then $(k,D_{1/2}/d)=1$, and so
\begin{eqnarray*}
& & {1  \over   2^{  m-n }}
 \sum_{k=1\atop (k, D_{1/2})=d}^{m-n}   {{m-n}\choose {k}}
\ = \ 
 {1  \over   2^{  m-n }}
 \sum_{1\le \k\le (m-n)/d\atop 
 (\k,D_{1/2}/d)=1}     {{m-n}\choose {\k d}}
\cr&=&
 {1  \over   2^{  m-n }}
 \sum_{1\le \k\le (m-n)/d}   {{m-n}\choose {\k d}}\d\big((\k, D_{1/2}/d)\big)
 \cr&=&{1  \over   2^{  m-n }}
 \sum_{1\le \k\le (m-n)/d}   {{m-n}\choose {\k d}}\Big(\sum_{u|(\k, D_{1/2}/d) }\m(u)\Big)
 \cr&=&\sum_{u|  D_{1/2}/d }\m(u)
 {1  \over   2^{  m-n }} \sum_{1\le \k\le (m-n)/d\atop u|\k}   {{m-n}\choose {\k d}}
    \cr &\buildrel{(\k = u\tau)}\over{=}& 
 \sum_{u|  D_{1/2} /d}\m(u)
 {1  \over   2^{  m-n }} \sum_{1\le ud\tau \le  m-n}     {{m-n}\choose {ud\tau}}
 \cr &=&
 \sum_{u|  (D_{1/2}/d ) }\m(u)
 {1  \over   2^{  m-n }} \sum_{1\le x \le  m-n\atop 
 ud|x}    {{m-n}\choose {x}}\cr&=& 
  \sum_{u|  (D_{1/2}/d )}\m(u)\,\big(\P\{ud|B_{m-n}\}-{1  \over  2^{  m-n }}
\big)
  \cr&=&
  \sum_{u|  (D_{1/2}/d )}\m(u) \P\{ud|B_{m-n}\}-{\d(D_{1/2}/d ) \over  2^{  m-n }}.\end{eqnarray*} 
Consequently,
 \begin{eqnarray*}
& &{1  \over  2^{  m-n }}\sum_{d| D_{1/2}} 2^{\o (d)}d
 \sum_{k=1\atop (k, D_{1/2})=d}^{m-n}   {{m-n}\choose {k}}  
\cr &=&
 \sum_{d| D_{1/2}} 2^{\o (d)}d
 \Big(\sum_{u|  (D_{1/2}/d )}\m(u)\,\big(\P\{ud|B_{m-n}\}-{\d(D_{1/2}/d )  \over  2^{  m-n }}\Big)
 \cr &=&
 \sum_{d| D_{1/2} } 2^{\o (d)}d
 \Big(\sum_{u|  (D_{1/2}/d )}\m(u)\, \P\{ud|B_{m-n}\}\Big)-{1  \over  2^{  m-n }} .
 \end{eqnarray*}
We observe   that $\d(D_{1/2}/d ) =0$ as long as $d< D_{1/2}$, and if $d=D_{1/2}$, then the corresponding summand   produces the contribution $2^{\o (D_{1/2})}D_{1/2}   \P\{ D_{1/2}|B_{m-n}\}-{1  \over  2^{  m-n }}$  (recalling that $\m(1)=1$), whence the last line. 
\vskip 3 pt
By reporting this in \eqref{323} we get
\begin{eqnarray}\label{323a} 
 {1\over D2^{ m-n }  } \sum_{k=0}^{m-n}  {{m-n}\choose {k}} \, \varrho_k(D) 
 & =  & {1  \over D 2^{  m-n }}\sum_{d| D_{1/2}} 2^{\o (d)}d
 \sum_{k=1\atop (k, D_{1/2})=d}^{m-n}   {{m-n}\choose {k}}  +{1\over 2^{  m-n }\sqrt D }
\cr & =  &  {1  \over D }\sum_{d| D_{1/2} } 2^{\o (d)}d
 \Big(\sum_{u|  (D_{1/2}/d )}\m(u)\, \P\{ud|B_{m-n}\}\Big)
 \cr & &  +{1\over 2^{  m-n }\sqrt D }-{1  \over  D2^{  m-n }} 
   .\end{eqnarray}
By inserting now this into  \eqref{4.11a}, we finally get
\begin{eqnarray}\label{4.11b}  \P \{ D|B_nB_m \} 
&\le &{1  \over D }\sum_{d| D_{1/2} } 2^{\o (d)}d
 \Big(\sum_{u|  (D_{1/2}/d )}\m(u)\, \P\{ud|B_{m-n}\}\Big)
\cr &  &   +{1\over 2^{  m-n }\sqrt D }-{1  \over D 2^{  m-n }}+C_\e
\left({     D  ^{1 + \e}\over   n }\right)^{1/2 }
\cr &\le &{1  \over D }\sum_{d| D_{1/2} } 2^{\o (d)}d
 \Big(\sum_{u|  (D_{1/2}/d )}\m(u)\, \P\{ud|B_{m-n}\}\Big)
\cr &  &   +{1\over 2^{  m-n }\sqrt D } +C_\e
\left({     D  ^{1 + \e}\over   n }\right)^{1/2 }. 
\end{eqnarray} 
This achieves the proof of Theorem \ref{bnmbound}.
\end{proof}

\vskip 30 pt

\section{Divisibility and primality in the Rademacher random walk}\label{s8}
The primality properties in that case  are more studied than divisibility. The necessary value distribution results  of divisors     are taken from  \cite{W7a}, which contains an application to the functional equation of the
 Zeta-Riemann function. An extremal divisor case is thoroughly investigated in \cite{W6}. The proof deeply relies on mixing properties of the system $\big\{ \chi\{d|R_n\},    \    2\le d\le n,\ n\ge 1\big\}$.
 
 \subsection{Divisibility.}Put, 
\begin{eqnarray}\label{radem(2.7)}\qq \qq  
\Theta_1(\d,M) & =& 2\ \displaystyle{\sum_{\ell\in \Z} e^{-{2M\pi^2\ell^2\over \d^2}},} \cr   
\Theta_2(\d,M)&=&2\ \displaystyle{\sum_{\ell=0}^\infty \big\{e^{-{ M\pi^2(2\ell )^2/ 2\d^2}}-e^{-{
M\pi^2(2\ell+1)^2/ 2\d^2}}\big\}-1.}
 \end{eqnarray}
We note that
\begin{equation}\Theta_1(\d,M)=2\,\d /\sqrt{  {2M\pi } } +\mathcal O(1)
\end{equation}
when $\d\to \infty$, uniformly in any interval of   $M$.

 \vskip 3 pt  The  Theorem below provides precise estimates of $\P\big\{ \d | R_{M}\big\}$.
\begin{theorem}[\cite{W7}, Th.\,2.1]\label{radem.t2.1}
 Let $\a >\a^\prime>3/2$. There exist constants
$C$, $\d_0$  and  
$M_0$, depending on $\a,\a^{\prime}$ only, such that for any $M\ge M_0$, $M\ge \d\ge \d_0$:
\vskip 3 pt  
--- if $\d$ and $M$ are even,
\begin{eqnarray} \Big|\P\big\{ \d | R_{M}\big\} -{ 2 + 4 \cos^M {2\pi  \over \d}  \over \d}\Big| \le  M^{- \a^\prime },& \qquad 
\hbox{if} \ \  \d <
2\pi 
\sqrt{M\over {2\a\log M}} ,\cr
 \Big|  \P\big\{ \d | R_{M}\big\}-  {1\over \d}\Theta_1(\d,M)  \Big|\le C  {
\log^{5/2} M 
\over M^{3/2}},& \qquad \hbox{if} \ \d \ge  2\pi  \sqrt{M\over {2\a\log M}}. 
\end{eqnarray}
\vskip 3 pt   
--- if $\d$ and $M$ are odd,
\begin{eqnarray}  \Big|\P\{ \d|R_M\}-{1+ 2 (\cos {2\pi  \over \d})^M-2(\cos{\pi  \over \d})^M \over \d}\Big| 
\le M^{- \a^\prime }, & &  \hbox{if} \ \ \d <
2\pi 
\sqrt{M\over {2\a\log M}} ,\cr
  \Big|\P\{ \d|R_M\} - { \Theta_2(\d,M) \over \d}\Big|\le C   {\log^{5/2} M \over M^{3/2}},  & & \hbox{if} \ \d \ \ge  2\pi  \sqrt{M\over {2\a\log
M}}.
\end{eqnarray} 
\end{theorem}

\bigskip\par 
\subsection{An extremal divisor case.} Put  successively, 
$$ N_1 =1, \qquad N_k =\inf\{\, N>N_{k-1}  \, :\, N\   {\rm even \  and} \  \, N | R_{N^2}\, \},\qquad(k>1). $$ 
The fact that this random sequence  is well defined,  can be deduced from our result below. 

\vskip 3 pt The following fine asymptotic result    exhibits an exponential growth of the sequence
$(N_k)_k$. 

\begin{theorem}[Weber \cite{W6}]\label{NkRn} Put $s= 2 \sum_{j\in \Z}  e^{-2\pi^2j^2}$. For any $\tau>7/8$,
$$ \     \log N_k = {k\over s} + {\mathcal O}\big(  k^{\tau} \big) \qquad\qquad  \hbox{\it almost surely.} $$
\end{theorem}
 
  The result extends to Bernoulli sequences $\b= \{\b_i , i\ge 1\}$. Write $B_N  = \b_1+\ldots +\b_N$, $N=1,2,\ldots$ Since $
\e_i\buildrel{{\mathcal D}}\over{=}2\b_i-1 $, Theorem \ref{NkRn} gives the same estimate for the sequence 
 $$ M_1 =1, \qquad M_k =\inf\{\, M>M_{k-1}  \ :\ M\  \textstyle{even \  and} \quad M | \,2B_{M^2} \, \},\qquad(k>1) $$
 
The proof is intricate and  long.  We only give a sketch of the main steps. First for $N$ even,  one has the asymptotic estimate
\begin{equation}N\P\big\{ N | R_{N^2}\big\} = s + {\mathcal O}\Big( { \log^{5/2} N  \over N^2}\Big) . 
\end{equation}

  Next  it is   necessary to have at disposal asymptotic estimates of  
  the
correlation 
$$\P \{ N | R_{N^2},  M | R_{M^2} \} - \P \{  N |
R_{N^2} \}\P \{  M | R_{M^2} \}.$$
This is a considerably more   delicate task.
The next
statement is the crucial step towards the proof of Theorem \ref{NkRn}.
\begin{lemma}[\cite{W6}] \label{NkRnl2}
There exists an absolute constant
$C$, and for every $\e >0$, a constant $C_\e$ depending on $\e$ only, such that for any even positive  integers
$N\ge M$ large enough,
\vskip 3pt a) if $\  M\le N\le M+  N/(\log N)^{1/2}$ 
 \begin{align*} \Big|\P\big\{ N | R_{N^2},  M | R_{M^2}\big\}- \P\big\{  N |
R_{N^2}\big\}\P\big\{  M | R_{M^2}\big\}\Big|\le C\,{(\log N)^{1/2}\over N } {(\log M)^{1/2}\over M}
.\end{align*}
\vskip 3pt  b) if $  N>  M+  N/(\log N)^{1/2}$ 
\begin{align*}  \Big|\P\big\{ N | R_{N^2},  M | R_{M^2}\big\}-& \P\big\{  N |
R_{N^2}\big\}\P\big\{  M | R_{M^2}\big\}\Big|
\cr &\le C_\e\,\Big\{  {(\log M)^{5/2}(\log N)^{1/2}\over M^2N}+ {1\over N^2}(\log
N)^{3/4\, \,+\e} \Big\}.
\end{align*}
\end{lemma}

The next tool is Lemma 10 p.\,45  in Sprind\v zuk \cite{Sp}, a typical criterion of almost sure convergence, providing  in addition an estimate of the  remainder term. 

\begin{lemma} \label{NkRnl4} 
Let $ \{f_l, l\ge 1\}$ be a sequence
 of non negative random variables, and let $\p=\{\p_l, l\ge 1\}$, $m=\{m_l, l\ge 1\}$ be two sequences of
non negative reals such that 
\begin{equation}\label{RNRM.2.49}
 0\le \p_l\le m_l \le 1\qquad (l\ge 1).
\end{equation}
Write 
$$ M_n =\sum_{1\le l\le n}  m_l.$$

\noi  Assume that $M_n\uparrow\infty$ and that the following condition is satisfied 
\begin{equation}\label{RNRM.2.50}\E\  \big| \sum_{i\le l \le j}(f_l-\p_l)\big|^2 \,\le
C\sum_{i\le l \le j} m_l ,
\qquad  0\le i\le
 j<\infty. 
 \end{equation}
    Then, for every $a>3/2$,   
\begin{equation}\label{RNRM.2.51} \sum_{1\le l\le n}  f_l \ \buildrel{a.s.}\over{=} \sum_{1\le l\le n}  \p_l +{\mathcal O}\Bigl( M_n^{1/2}  {\log}^{ a}
M_n \Big). 
\end{equation}
 \end{lemma}
\subsection{Primality.}

\begin{theorem}[\cite{W5a}]\label{P-Rn} There exist two absolute constants $C_1$ and  $C_2$ such that for any integer $n\ge 3$,
$$ {  C_1  \over    \log
n }\le \P\{R_n\ \hbox{is prime} \}\le { C_2  \over    \log
n }. $$   
In fact, for every $\e>0$,
$${4    \over  \sqrt{ 2e\pi  }\log n}-{C _\e  \over  n^{  
    1/2-\e }}\le \P\{R_n\ \hbox{is prime} \}\le   \sqrt{{2
\over     \pi }  }\Big(\sum_{k=1}^{\infty}
{2^{ k }    }e^{-   2^{2k-1  } }\Big)  {1
\over \log     n  }+{C'_\e  \over  n^{  
    1/2-\e }},$$
where the constants $C_\e$, $C_\e '$ depend on $\e$ only.\end{theorem}
The constants obtained are certainly not the best, and the  limit  $\displaystyle{\lim_{n\to \infty}} (\log n)\,\P\{R_n\, \hbox{prime} \}$ may  exists.

\begin{proof}
Put  
$$I_k(n) =[-2^{ k}\sqrt n, 2^{ k}\sqrt n], \qq k=0,1,\ldots $$
next 
\begin{equation} \label{jkn}
  J_k(n)=  \begin{cases}  {\mathcal P}\cap  I_0(n) & \qq \hbox{if $k=0$},\cr         {\mathcal P}\cap\big(I_k(n) \backslash I_{k-1}(n) \big)  & \qq \hbox{if $k\ge 1 $},\end{cases}
\end{equation} 
and let  $k_n= \max\{ k: 2^{k-1}\le
a \sqrt{ 2\log \log n} \}$, where $a>0$ will be  chosen later on. 
Then 
$$\P\{R_n\ \hbox{ prime} \}=\sum_{k=0}^\infty \P\{R_n\in  J_k(n) \}=\sum_{k=0}^{k_n} \P\{R_n\in  J_k(n) \}+\rho_n . $$
There exists a numerical constant $c$ ($c={4 \over \sqrt {3\pi} }$ is suitable) such that if $a>c\sqrt 2$, 
\begin{eqnarray*} \P\{  |R_n|> a\sqrt{2 n\log \log n}\}
&=&\P\{  |{R_n\over \sqrt n }|> a\sqrt{2  \log \log n}\}
\cr  & \le &e^{ -(2a^2/c^2) \log \log
n}\E e^{ {|R_n|^2/(
 c^2 n )}}  
 \cr  & \le & {2\over (\log n)^{2a^2/c^2}} .
 \end{eqnarray*} 
Thus
$$\rho_n\le \P\{|R_n|\ge a\sqrt{2n\log\log n}\} \ll {1\over  \log n  }. $$

Besides, 
$$\P\{R_n\in  J_k(n) \}= \P\{R_n\in  J_k(n) \}- {1\over  \sqrt{ 2\pi n}} \sum_{p\in  J_k(n)}e^{-  {p^2\over  2 n } }+{1\over  \sqrt{ 2\pi n}} \sum_{p\in 
J_k(n)}e^{-  {p^2\over  2 n } }.$$
By means of Theorem \ref{radem.t2.1}, 
$$\sum_{k=0}^{k_n}\Big|\P\{R_n\in  J_k(n) \} - {1\over  \sqrt{ 2\pi n}} \sum_{p\in  J_k(n)}e^{-  {p^2\over  2 n } }\Big|  \le   {C_\e\over n^{  
   3/2-\e }} \sum_{k=0}^{k_n}\#(
J_k(n)) \le C_\e   n^{  
   -1/2+\e }.$$
 
The following estimates are classical: (\cite{RS}, Th.\,2 and Cor.\,1, p.\,69, and  \cite{DR}, Lemma 5, p.\,439)
\begin{eqnarray*}x/\log x \le \pi(x)\le x/(\log x-3/2), & &\qq \hbox{for any $x\ge 67$,}   \cr 
 \pi(2x)-\pi(x)\le {x\over \log x}, &   & 
 \qq\hbox{for any $x>
1$}.
\end{eqnarray*}
Then $$  {4    \over  \sqrt{ 2e\pi  }\log n}\le{2\pi(\sqrt n)\over  \sqrt{ 2e\pi n}}\le {1\over  \sqrt{ 2\pi n}} \sum_{p\in  J_0(n)}e^{-  {p^2\over  2 n
} }\le   {2\pi(\sqrt n)\over  \sqrt{ 2\pi n}}\le   {4   \over ( \log n-3 ) \sqrt{ 2\pi
 }}.$$
And
\begin{eqnarray*}{1\over  \sqrt{ 2\pi n}}\sum_{k=1}^{k_n}  e^{-  {p^2\over  2 n } }\  \le \ {2\over  \sqrt{ 2\pi n}}\sum_{k=1}^{k_n} {2^{ k  }\sqrt
n\over \log (2^{ k } \sqrt n)}e^{-   2^{2k-1  } } \  \le \  \sqrt{{2
\over     \pi }  }\Big(\sum_{k=1}^{\infty}
{2^{ k }    }e^{-   2^{2k-1  } }\Big)  {1
\over \log     n  }
 .
 \end{eqnarray*}
We deduce
$$ {4    \over  \sqrt{ 2e\pi  }\log n}\le{1\over  \sqrt{ 2\pi n}}\sum_{k=0}^{k_n}  e^{-  {p^2\over  2 n } }\le  \sqrt{{2
\over     \pi }  }\Big(\sum_{k=1}^{\infty}
{2^{ k }    }e^{-   2^{2k-1  } }\Big)  {1
\over \log     n  }.$$
From there, it follows immediately that
$${4    \over  \sqrt{ 2e\pi  }\log n}-{C _\e  \over  n^{  
    1/2-\e }}\le \P\{R_n\ \hbox{ prime} \}\le   \sqrt{{2
\over     \pi }  }\Big(\sum_{k=1}^{\infty}
{2^{ k }    }e^{-   2^{2k-1  } }\Big)  {1
\over \log     n  }+{C'_\e  \over  n^{  
    1/2-\e }}.$$
Whence the claimed result.
\end{proof}

\bigskip\par
\bigskip\par

The following Theorem is complementing the above.

\begin{theorem}[\cite{W5a}]\label{P-RnRm} Let $m>n$ be two odd positive integers. Then for some absolute constant $C$, and every $\e>0$,
$$\P\{R_n\ {\rm and } \ R_m\ \hbox{are prime} \}\le  {C\over \sqrt{ m-n } \,\log n} +C_\e n^{-1/2+\e} + {C \log\log^+    (m-n  )\over (\log
n)^2}     .$$\end{theorem}

\begin{proof}  Let  $m>n$ be two odd positive integers. We have
\begin{eqnarray}\label{P-RnRm.1}\P\{R_n\ {\rm and} \ R_m\hbox{ prime} \}&=&\P\{R_n\ {\rm and} \ R_n +(R_m-R_n)\ \hbox{prime} \}
\cr &=& \sum_{|p|\le  n}\P\{R_n=p \}\Big(\sum_{|v|\le m-n\atop
p+v\in {\mathcal P}} \P\{R_{m-n}=v \}\Big)
\cr &=&\sum_{|v|\le m-n}\Big(\sum_{|p|\le  n\atop
p+v\in {\mathcal P}}\P\{R_n=p \}\Big)\P\{R_{m-n}=v \}.
\end{eqnarray} 
The summation term corresponding to $v=0$ equals to
\begin{eqnarray}\label{v=0} \sum_{|p|\le  n\atop
p \in {\mathcal P}}\P\{R_n=p \} \P\{R_{m-n}=0 \}& =&\P\{R_n  \, \hbox{prime}\} \P\{R_{m-n}=0 \} 
\cr &\le& {C\over \sqrt{ m-n } }\,\P\{R_n  \, \hbox{prime}\}   \,\le \,  {C\over \sqrt{ m-n } \,\log n}
\end{eqnarray} 
since $\P\{R_{m-n}=0 \}=2^{-(m-n)}{{m-n}\choose{{ m-n\over2}}}\sim{2\over \sqrt{2\pi(m-n)}}$.
\vskip 2 pt We next write for $|v|\le m-n$,
$$ \sum_{|p|\le  n\atop
p+v\in {\mathcal P}}\Big|\P\{R_n=p \}-{1\over  \sqrt{ 2\pi n}}   e^{-  {p^2\over  2 n } } \Big|\le C_\e n^{1-3/2+\e}= C_\e n^{-1/2+\e}.$$
Further,
$$  {1\over  \sqrt{ 2\pi n}} \sum_{|p|\le  n\atop
p+v\in {\mathcal P}}  e^{-  {p^2\over  2 n } } = {1\over  \sqrt{ 2\pi n}} \sum_{k=0}^\infty \sum_{p\in J_k(n)\atop
p+v\in {\mathcal P}} e^{-  {p^2\over  2 n } },$$
where $J_k(n)$ are defined in \eqref{jkn}. Recall  Siebert's estimate \cite{Si}: let $a\neq0$, $b\neq 0$ be integers with $(a,b)=1$, $2|ab$. Then for $x>1$,
 \begin{equation}\label{Siebert.est.}
\sum_{p\le x \atop   ap+b\in {\mathcal P} }\le 16\o{x \over (\log x)^2}  \prod_{p|ab\atop p>2}\, \Big({p-1\over p-2}\Big), \qq \o= \prod_{p>2}\Big(1-\frac{1}{(p-1)^2}\Big) .
\end{equation}
   Daboussi and   Rivat's observation  \cite{DR}, p.\,440, that
$$\frac{p-1}{p-2}= \Big(1-\frac{1}{p}\Big)^{-1}  \Big(1-\frac{1}{(p-1)^2}\Big)^{-1} $$   
and when $h$ is even
$$\prod_{p|h\atop p\ge 3}\Big(1-\frac{1}{p}\Big) = 2\frac{\p(h)}{h}, 
$$ 
imply 
that for $x>1$ and  any positive even number $h$,
\begin{equation}\label{Siebert.est}
\sum_{p\le x \atop    p+h\in {\mathcal P} }1\le 8\prod_{p\atop (p,h)=1}\Big(1-\frac{1}{(p-1)^2}\Big){ x\over (\log x)^2} {h\over \p(h)} \le 8\,{ x\over (\log x)^2} {h\over \p(h)} .\end{equation}

In    \eqref{P-RnRm.1}, the  sums are indexed over   pairs $(p,p+v)$ with $|p|\le n$, $|v|\le m-n$.   By symmetry, 
it     suffices to restrict to the case $p>0$, $v<0$, namely $(p,p-|v|)$.  
\vskip 4 pt Consider three sub-cases: (a) $|v|< p$, (b) $|v|\ge 2p$, (c) $p\le  |v|<2p$.  
\begin{eqnarray*}
{\rm (a)}& &
\hbox{$(p-|v|,p)$ with  $1\le p-|v|\le n$, $|v|\le m-n$.}
\cr {\rm (b)} & & \hbox{
$(p, |v|- p)=(p,p+|v|-2p)$  with $1\le p\le n$, ${0\le |v|-2p\le m-n}.^{(*)}$}
   \cr {\rm (c)} & &
    \hbox{$(|v|-p, p)=(|v|-p,|v|-p+2p-|v|)$ with $1\le |v|-p\le 2p-p=p\le n$,}     \cr & & \hbox{
      $1\le 2p-|v|\le 2|v|-|v|=|v|\le m-n$.}
\end{eqnarray*}
 ${}^*$(The case $2p-|v|=0$ in (b) is treated in  \eqref{v=0}.) 
  Thus  the pairs $(p-|v|,p)$ are pairs $(\a, \a+h)$ of positive integers   all satisfying  $\a\le n$, $h\le m-n$. So that estimate \eqref{Siebert.est} remains applicable to this case, up to a multiplicative factor 4.
\vskip 2pt

It follows that  $$\#\{ p\in J_k(n) : p+v\in {\mathcal P} \}\le \#\{ |p|\le 2^k\sqrt n : p+v\in {\mathcal P} \}\le C{ 2^k\sqrt n\over (\log 2^k\sqrt n)^2}
{|v|\over \p(|v|)}. $$ 
Hence,
$${1\over  \sqrt{ 2\pi n}} \sum_{k=0}^\infty \sum_{p\in J_k(n)\atop
p+v\in {\mathcal P}} e^{-  {p^2\over  2 n } }\le {C\over (\log n)^2}{|v|\over \p(|v|)} \sum_{k=0}^\infty 2^k e^{-2^{2k-3}}\le {C\over (\log
n)^2}{|v|\over \p(|v|)}.$$
 Consequently
$$\sum_{|p|\le  n\atop
p+v\in {\mathcal P}} \P\{R_n=p \}\le C_\e n^{-1/2+\e} + {C\over (\log
n)^2}{|v|\over \p(|v|)}.$$
By reporting,
$$\P\{R_n\ {\rm and} \ R_m\hbox{ prime} \}  =\sum_{|v|\le m-n}\Big(\sum_{|p|\le  n\atop
p+v\in {\mathcal P}}\P\{R_n=p \}\Big)\P\{R_{m-n}=v \} $$
$$ \le {C\over \sqrt{ m-n } \,\log n} +C_\e n^{-1/2+\e} + {C\over (\log
n)^2}\sum_{0<|v|\le m-n} {|v|\over \p(|v|)}\P\{R_{m-n}=v \} .$$
$$ \le {C\over \sqrt{ m-n } \,\log n} +C_\e n^{-1/2+\e} + {C\over (\log
n)^2}\ \E \, {|R_{m-n}|\over \p(|R_{m-n}|)} {\bf 1}_{\{|R_{m-n}|\ge 1\}} .$$
We have obtained 
\begin{eqnarray*}\P\{R_n\ {\rm and} \ R_m\hbox{ prime} \}&\le & {C\over \sqrt{ m-n } \,\log n} +C_\e n^{-1/2+\e} 
\cr & &\qq \quad + {C\over (\log
n)^2}\ \E \, {|R_{m-n}|\over \p(|R_{m-n}|)} {\bf 1}_{\{|R_{m-n}|\ge 1\}}.
\end{eqnarray*}
We have the following estimates concerning $\p$ (Rosser and Schoenberg \cite{RS},  inequality 3.40 and 3.41), for $h\ge 3$
\begin{eqnarray*}&& \sum_{k=1}^n {k\over \p(k)}={\mathcal O}(n)  ,
\cr & & {|v|\over \p(|v|)}\le e^\gamma \log\log |v| + {2,50637\over \log\log |v|}\le C\log\log |v|.
 \end{eqnarray*}
We deduce
$$\P\{R_n\ {\rm and} \ R_m\hbox{ prime} \}\le  {C\over \sqrt{ m-n } \,\log n} +C_\e n^{-1/2+\e} + {C \log\log^+    (m-n  )\over (\log
n)^2}     .$$
\end{proof}

\bigskip\par
 
\subsection{Primality mod($k$).}

Theorem \ref{P-Rn} can be extended to a   larger class of random variables under moderate integrability assumption (only the upper bound part is concerned).  Indeed, by using the Brun-Titchmarsh  theorem and the local limit theorem, a similar result is valid under a stronger form,   the constraint made on $k$ below, is weak but the bound obtained, although sharp, should admit  improvments (maybe the $\log\log$ term can be dropped).

\begin{theorem}[\cite{W5a}]\label{P-Sn} Let $X$ be a symmetric random variable such that   
 $\E X =0=\E X ^3$,  $\E X ^2=1$ and $\E |X| ^{3+\e }<\infty$  for some $\e >0$. Let $X_i$, $i\ge 1$ be independent copies of $X$ and for each positive integer $n$, $S_n= X_1+\ldots + X_n$. Let $k$ be some positive integer and  $\ell$ be an  odd integer such that
   $(\ell,k)=1$ and $
\sqrt{ n  \log\log n}/k\to
\infty$. Then for  every $B>1$,   
$$ \P\Big\{S_n \  \hbox{is prime},\,  S_n\equiv \ell \ ({\rm mod}\, k)\Big\}={\mathcal O}_B\Big\{   {\sqrt{   
\log\log n}\over \varphi(k)\log    n   } +  {  1\over (\log n)^{B} }\Big\}.
$$ 
\end{theorem}

The proof uses the following local limit theorem due to Ibragimov, which characterizes the speed of convergence in the i.i.d. case. 
 
\begin{proposition}[\cite{P},  p.\,216]\label{llt.indep.ibra.g.moments} Let  $\{X_k, k\ge 1\}$ be a sequence of centered i.i.d. $\Z$-valued random variables, with common distribution function $F$, and unit variance. 
 Put $$r_n= \sup_{k}\Big|  \sqrt n \P\{S_n=k\}-{1\over  \sqrt{ 2\pi}}e^{-  {k^2\over  2 n } }\Big|  $$

Then   $ r_n ={\mathcal O}\big(n^{-(m+\d)/2}\big) $ 
for some integer  $m\ge 0$  and a real $\d$,  $\d\in ]0,1[$, if and only if  $$\sup\big\{h>0; \exists a\in
\Z:
\P\{X_1\in a+h\Z\}=1\big\}= 1,$$  and the two following conditions are realized:
\smallskip\par {\rm (i)} All moments of  $X_1$ up to the order  $m+2$ included exist and coincide with those of the standard normal law.  
\smallskip\par {\rm (ii)}
$$\int_{|x|\ge v}|x|^{m+2} F(dx)={\mathcal O}\big(v^{-\d}\big) \qq  \qq v\to \infty. $$
\end{proposition}
\begin{proof}[Proof of Theorem \ref{P-Sn}]

  As there exists   $\e>0$ such that $\E |X|^{3+\e}<\infty $, $X$ satisfies the requirements of  Proposition \ref{llt.indep.ibra.g.moments} with 
$m=1$ and
 $0<\d<1$, $\d= \d(\e)$. Then
\begin{align*}
\sup_{k}\big|  \sqrt n \P\{S_n=k\}-{1\over  \sqrt{ 2\pi}}e^{-  {k^2\over  2 n } }\big| ={\mathcal O}\big(n^{-(1+\d)/2}\big),\end{align*}
and
\begin{align}\label{P-Sn(3.3).} \big| \sum_{k\in J}\P\{S_n=k\}-{1\over  \sqrt{ 2\pi n}}\sum_{k\in J}
e^{-  {k^2\over  2 n } }\big|\le C  {\#\{J\}\over   n^{1+(\d/2)}}.
\end{align}

Besides, by means of a result of  Robbins and Siegmund (\cite{P}, p.\,314), if $\a>1$, $b>1$ are two arbitrary but fixed reals,
there exists a constant $C$ depending on   $\a$, $b$ and $X$, an integer  $n(\a)$ depending on  $\a$ only, such that for any $n\ge
n(\a)$,
\begin{align}\label{P-Sn(3.4)} \P\Big\{ \exists m\ge n : S_m>\sqrt{2\a(1+b)m  \log\log m}\Big\}\le C (\log n)^{-b} (\log\log n)^{-1/2}.
\end{align}
As  $X$ is symmetric, $\P\{|S_n|>u\}\le 2\P\{ S_n >u\}$; and thus it follows from  \eqref{P-Sn(3.4)} that 
\begin{align}\label{P-Sn(3.5)} \P\Big\{  |S_n|>\sqrt{2\a(1+b)n  \log\log n}\Big\}\le 2C (\log n)^{-b} (\log\log n)^{-1/2}.\end{align}
 Let  $k$, $\ell$
be coprime integers. Put
$$J=\big\{2\le |p|\le \sqrt{2\a(1+b)n  \log\log n}: p\equiv \ell \ ({\rm mod}\,
k)\big\} .$$  
Using \eqref{P-Sn(3.3).} and \eqref{P-Sn(3.5)} we have
\begin{align}\label{P-Sn(3.6)} \bigg| \P\Big\{S_n \ \hbox{ prime},\,  S_n\equiv \ell \ ({\rm mod}\, k)\Big\}&- \sqrt{2\over       \pi n }\sum_{j\in J}e^ {-  { j^2
/  2n  }  }\bigg| \cr &\le  \  {C\#\{J\}\over   n^{1+(\d/2)}}+  {2C\over (\log n)^{ b} (\log\log n)^{ 1/2}}.
\end{align}
Whence
\begin{align}\label{P-Sn(3.7)} \P\Big\{S_n \ \hbox{ prime},\, S_n\equiv \ell \ ({\rm mod}\, k)\Big\}\le\   {\#(J)\over      \sqrt{ \pi n} }   + C' (\log n)^{-b}
(\log\log n)^{-1/2}.\end{align}
Using the notation  $\pi(x;\ell, k):= \#\big\{p\le x: p\equiv \ell \ ({\rm mod}\, k)\big\} $, we have
 $$    \#\big\{2\le  p \le
\sqrt{2\a(1+b)n  \log\log n}:  p\equiv \ell \ ({\rm mod}\, k) \big\}  =  \pi(\sqrt{2\a(1+b)n  \log\log n};\ell, k)  . $$ 
By means of  the Brun-Titchmarsh theorem (see for instance \cite[p.\,83]{Te1}), 
$$\pi(x+y;\ell, k)-\pi(x;\ell, k) \le (2+o(1)) {y\over \varphi(k)\log (y/k)},
$$
 as $y/k$ tends to infinity, where  $\varphi$ is Euler's totient function. \vskip 2 pt 
 
 Letting $x=2$ and $y=\sqrt{2\a(1+b)n  \log\log n} -2$,  we have in particular, as   $
\sqrt{ n  \log\log n}/k\to
\infty$, 
\begin{eqnarray}\label{P-Sn(3.8)} {\pi(\sqrt{2\a(1+b)n  \log\log n};\ell, k)\over \sqrt
  n}&\le& { (2+o(1)) \over \sqrt
  n}  {\sqrt{2\a(1+b)n  \log\log n}\over \varphi(k)\log  (\sqrt{2\a(1+b)n  \log\log n}/k)   }\cr &\le  &C''   {\sqrt{   
\log\log n}\over \varphi(k)\log    n   }.
\end{eqnarray}

By combining  \eqref{P-Sn(3.8)} with \eqref{P-Sn(3.7)}, we deduce that for all  $n$ large enough, if  $
\sqrt{ n  \log\log n}/k\to
\infty$ 
 $$ \P\Big\{S_n \  \hbox{ prime},\,  S_n\equiv \ell \ ({\rm mod}\, k)\Big\}\le  C''   {\sqrt{   
\log\log n}\over \varphi(k)\log    n   } + C'{  1\over (\log n)^{ b}(\log\log n)^{ 1/2} }.
 $$

Consequently, for any   $B>1$, there exists a constant $C(B)$, such that, as $n$ is large and  $
\sqrt{ n  \log\log n}/k\to
\infty$ 
$$ \P\Big\{S_n \  \hbox{ prime},\,  S_n\equiv \ell \ ({\rm mod}\, k)\Big\}\le C(B)\Big\{   {\sqrt{   
\log\log n}\over \varphi(k)\log    n   } +  {  1\over (\log n)^{B} }\Big\}.
 $$
 \end{proof}

\section{Divisibility and primality in the Cram\'er random walk}\label{s9}
Cram\'er's probabilistic model basically consists with  a sequence of independent random variables $\xi_i$, defined for
$i\ge 3$ by
\begin{equation}\label{C.xin}\P\{\xi_i= 1\}= \frac1{\log i}, \qq \quad \P\{\xi_i= 0\}= 1-\frac1{\log i}.
\end{equation}
Let   $m_n=\E S_n=\sum_{j=3}^n \frac1{\log j}$, $B_n={\rm Var }S_n= \sum_{j=3}^n \frac1{\log j}(1-\frac1{\log j})$. 
Cram\'er's well-known twin prime conjecture, which seems largely accepted,   does not however assert that there are any primes in  the \lq primes\rq\,  of the model,  which are  the instants  of jump   
of the random walk  $\{S_n , n\ge 1\}$,  recursively defined  by 
  \begin{equation}\label{C.Pn} P_1= \inf\{ n\ge 3: \xi_n=1\},\qq   P_{\nu +1} =\inf\{ n>P_\nu: \xi_n=1\}  \qq \quad \nu\ge 1  .
\end{equation}
 We refer   to Granville \cite{G},   Pintz \cite{P}  and Montgomery and Soundararajan \cite{MS}.    Our recent paper Weber \cite{W8} offers a     probabilistic complementary study, in particular     of   the \lq primes\rq~ $P_\nu$      and of the primality properties of  $\{S_n, n\ge 1\}$.
Recall    
that  the standard limit theorems
are fulfilled,  
   and  the    invariance principle   holds, derived from Sakhanenko's. The law of the iterated logarithm implies that 
\begin{equation}\label{cramer.pnt}
 \#\big\{ \nu\in \mathcal C : \nu\le x\big\} = \int_2^x \frac{\dd t}{\log t} + \mathcal O\big( \sqrt{x\log\log x}\big) ,
\end{equation}
with probability one.          Cram\'er's  easier prediction on  the order of magnitude of the remainder   in the \lq   prime number theorem\rq, cannot be separated from the fact that this one should  also be sensitive to the subsequence on which it is observed, as    shown  in \cite{W8}. Such a property  is hardly conceivable  and was nowhere emphasized.  
 In a work in progress, that  study is  transposed to other random models for  prime numbers.

 \subsection{Divisibility}The distribution of divisors of $S_n$ turns up to share similarities with the one of Bernoulli sums. We prove the following result, which was only stated  in   \cite{W8}.
 \begin{theorem}\label{Cramer.div.Sn}
\begin{align*}   
\P\{d|S_n\}   \, = \frac{1}{d}  + \frac{1}{d}\sum_{  1\le  j\le \frac{d}{\pi}( \log n)  ( { \a \over 2 n} )^{1/2}} &  e^{ 2i \pi m_n(\frac{j}{d})  - 2\pi^2  B_n\,(\frac{j}{d})^2  }  +\mathcal O \Big(   \frac{d^3 \a^3(\log n)^4}{n^{3/2}}   \Big)
.
\end{align*}
\end{theorem}

  Let $\p_k(t)$ be the characteristic function of $\xi_k$, $\p_k(t)= \E e^{2i\pi \xi_k t}$. Let also 
$$\Phi_n (t)=\E e^{2i\pi tS_n}= \prod_{k=1}^n\p_k(t) ,  $$
be the characteristic function of $S_n$.  We  prove the following estimate with explicit constants.
\begin{proposition}\label{Cramer.Phin} We have
$|\Phi_n(t)|\le \exp\big\{- 2B_n\sin^2\pi t  \big\} $.
Further
\begin{equation*}
  \Phi_n (t)= e^{ 2i \pi tm_n -  2   B_n  (\pi t)^2  + E_n(t)  },
    \qq \quad {\text with}\qquad |E_n(t)| \,\le\,  12\,m_n\, (\pi  |t|)^3.
\end{equation*}
In particular,  
\begin{equation*}
  \E e^{ iy\, \frac{S_n-m_n}{\sqrt {B_n}} }= e^{    - \frac{y^2}{2}    + E_n(y)  }, \qq \quad {\text and}\qquad |E_n(y)| \, \le  2 \, \frac{\sqrt {\log n}\,|y|^3}{\sqrt {n}}   .
 \end{equation*}

 \end{proposition} 
The proof crucially uses the Lemma below. 
  
  \begin{lemma}\label{lfp1} Let $m  $ be a positive real and $p$ be a real such that $0<p<1$. Let $\b$ be a random variable defined by ${\mathbb P}\{ \b=0\}= p$, $
{\mathbb P}\{
\b=m\}= 1-p=q$.  Let $\p(t) ={\mathbb E\,} e^{2i\pi t \b}$. Then
 we have the following estimates,

\smallskip
\hspace*{1.2mm}{\rm (i)} For all real $t$,
 $|{\varphi}(t)|\le \exp\big\{- 2pq\sin^2\pi t m\big\} $

\smallskip
{\rm (ii)} If $q|\sin \pi t m|\le 1/3$, then
 \begin{eqnarray*}  \p(t)&=& e^{  q      2i\pi mt   - 2qp (  \pi mt)^2 +E}
 , 
\end{eqnarray*}
with  
$|E|\, \le \,   12 q (\pi m|t|)^3$.
    \end{lemma}
      For   sake of completeness we include the proof.
   \begin{proof} (i) One verifies that $|\p(t)|^2=1- 4 pq \sin^2\pi mt.$ As moreover $1-\t\le e^{-\t}$ if  ${\vartheta} \ge 0$, we obtain $|\p(t)|^2\le e^{-4 pq \sin^2\pi mt}$. 

\vskip 3 pt  (ii) Let  $|u|\le u_0<1$. From the series expansion of $\log(1+u)$,  it follows that
$$\log(1+u)=u-\frac{u^2}{2}+R, \qq \qquad|R|\le |u|^3\,\sum_{j=0}^\infty \frac{|\theta|^j}{3+j}\le \frac{|u|^3}{3(1- u_0 )}.$$
 Then $ 1+u= \exp\{ u -\frac{u^2}{2} +B\}$,  with $|B|\le C_0\,|u|^3$ 
and  $C_0=\frac{1}{3(1- u_0 )}$.

 Writing that  $\p(t)= 1 +
q\big( e^{2i\pi mt} -1\big)=1 + u$,
  where $|u|= 2q|\sin\pi mt|\le  2q\big(1\wedge   \pi m|t| \big) $, we obtain 
  \begin{equation}\label{phi.u} \p(t)= 1+u=e^{q ( e^{2i\pi mt} -1 ) }\, e^{  -\frac{u^2}{2} +B}, \qq 
\quad |B|\le 8C_0\,q^3 |\sin\pi mt|^3.
\end{equation}
   
   In order to   estimate $u^2$, we let  $A(t)= e^{2i\pi mt} -1-  2i\pi mt+ \frac{(2 \pi mt)^2 }{2}$, and    write  $ (  e^{2i\pi mt} -1 )^2$ under the form 
\begin{align*} 
   &\   A(2t)-2A( t) -\big\{  -1-  4i\pi mt+ 8 ( \pi mt)^2\big\} +2\big\{ -1-  2i\pi mt+ \frac{(2 \pi mt)^2 }{2}\big\}+1
\cr &=
   A(2t)-2A( t) -     ( 2\pi mt)^2.
\end{align*}
Then
$ ( e^{2i\pi mt} -1 )^2+     ( 2\pi mt)^2 =    A(2t)-2A( t)$,
 and so   
  $\frac{u^2}{2}=-  2q^2(  \pi mt)^2+  \frac{q^2}{2}(A(2t)-2A( t)).$ 
 
\vskip 3 pt
Let $u_0=\frac{2}{3}$ so that $C_0=1$. We assumed $q|\sin \pi t m|\le 1/3$, thus $|u|\le 2/3$. We consequently  get with \eqref{phi.u}, 
\begin{eqnarray}\label{phi.bound} \p(t)&=&  e^{q ( e^{2i\pi mt} -1 )  -\frac{u^2}{2} +B }\, =\, 
e^{q ( e^{2i\pi mt} -1 ) +    2q^2(  \pi mt)^2-  \frac{q^2}{2}(A(2t)-2A( t))+B}
\cr &=&e^{  q      2i\pi mt   - 2qp (  \pi mt)^2 +H+B}
, 
\end{eqnarray}
with $H=q A(t) -    \frac{q^2}{2}(A(2t)-2A( t))$. By using 
 the   estimate (Lemma 4.14 in Kallenberg \cite{Ka})
 \begin{equation}\label{expit}\Big|e^{ix} -\sum_{k=0}^n \frac{(ix)^k}{k!}\Big|\,\le\, \frac{2|x|^n}{n!}\wedge \frac{ |x|^{n+1}}{(n+1)!},
 \end{equation}
valid for any $x\in\R$ and $n\in \Z_+$, and    letting $\d(x)=x^2\big(1\wedge   \frac{|x|}{6} \big)$, we get 
\begin{eqnarray}\label{At.est}
 |A(t)|&\le&
 \d(2\pi m|t|)
.\end{eqnarray}
Using the rough bound  $\d(x)\le  \frac{|x|^3}{6} $, and   $|B|
\le\,8q^3\big(1  \wedge \pi m | t|\big)^3$, we get 
 \begin{equation}\label{HB.bound} |H|+|B|\, \le \, 2q|A(t)|+\frac{q^2}{2}|A(2t)| +|B|\, \le \, \big(\frac83 q+\frac43q^2+8q^3\big)(\pi m|t|)^3
\, \le \,  12 q (\pi m|t|)^3.
\end{equation}
 We conclude by inserting estimate  \eqref{HB.bound} into \eqref{phi.bound}.  \end{proof}
 
 \begin{proof}[Proof of Proposition \ref{Cramer.Phin}]
   
We apply Lemma \ref{lfp1}. Here we have $m=1$, $q=\frac{1}{\log k}$. As $|{\varphi}_k(t)|\le \exp\big\{- 2(1-\frac{1}{\log k})(\frac{1}{\log k})\sin^2\pi t  \big\}$, we get 
 $$|\Phi_n(t)|\le \exp\big\{- 2\sum_{k=3}^n(1-\frac{1}{\log k})(\frac{1}{\log k})\sin^2\pi t  \big\}.$$
Further condition $q|\sin \pi t m|\le 1/3$ in Lemma \ref{lfp1}  reduces for $\p_k(t)$ to 
$ \frac{1}{\log k}|\sin \pi t  |\le 1/3$.  Thus 
\begin{equation}
  \label{chf.xik}
  \p_k (t)= e^{ 2i \pi    ( \frac{1}{\log k}) t - 2\pi^2  (1-\frac{1}{\log k})(\frac{1}{\log k} ) \,t^2  + C_k(t)  },
\end{equation}
where
 $|C_k(t)|
\le\, \frac{12  (\pi  |t|)^3}{\log k}$. 
  Recalling that $m_n =\sum_{j=3}^n \frac1{\log j}$, $B_n= \sum_{j=3}^n \frac1{\log j}(1-\frac1{\log j})$, it follows that 
\begin{equation}
  \label{chf.sn}
\E e^{2i\pi tS_n}=  \Phi_n (t)= e^{ 2i \pi tm_n -  2\pi^2  B_n  t^2  + D_n(t)  },
\end{equation}
and
$|D_n(t)| \,\le\,  12\,m_n\, (\pi  |t|)^3  $. 
In particular,  
 \begin{equation}
  \label{chf.sn1}
 \E e^{ iy\, \frac{S_n-m_n}{\sqrt {B_n}} }= e^{    - \frac{y^2}{2}    + E_n(y)  },
 \end{equation}
with  $|E_n(y)| \,\le\,  \frac{3m_n}{2} \, (\frac{y}{\sqrt {B_n}})^3\,\le\,  2 \, \frac{\sqrt {\log n}|y|^3}{\sqrt {n}}  $. 
  This achieves the proof.
\end{proof}

\begin{proof}[Proof of Theorem \ref{Cramer.div.Sn}]  We note that 
\begin{eqnarray}\label{div.Sn}\P\{ d|S_n\}
&=& \E \Big(\frac{1}{d}  \sum_{j=0}^{d-1} 
e^{2i\pi S_n \frac{j}{d}}\Big) \ = \ 
\frac{1}{d}  +\frac{1}{d}  \sum_{j=1}^{d-1} \Phi_n \big(  \frac{j}{d}\big)
\cr &= &\frac{1}{d}+ \frac{1}{d}  \sum_{j=1}^{d-1} 
e^{ 2i \pi m_n(\frac{j}{d})  - 2\pi^2  B_n\,(\frac{j}{d})^2+\mathcal O (m_n(\frac{j}{d})^3 ) } 
.\end{eqnarray}

The    elliptic Theta function $$  \Theta (d,m_n,B)  =  \sum_{\ell\in \Z} e^{im\pi{\ell\over   d }-{B\pi^2\ell^2\over 2 d^2}}, 
 $$
 appears in the above estimate of $\P\{d|S_n\}$. 
Here $m=m_n$, $B=B_n$ and we have $m_n\sim B_n\sim \frac{n}{\log n}$. 
\vskip 3 pt 
 Let 
 $\a>\a'>3/2$. Let
 $$\p_n= ( \log n) \big( { \a \over 2 n}\big)^{1/2},
\qquad\qquad \tau_n= {\sin\p_n/2\over
\p_n /2}.
$$ 
We assume
$n$ sufficiently large, say $n\ge
n_0$, for $\tau_n$ to be greater than $(\a^\prime/\a)^{1/2}$.  Consider two sectors
  $$A_n = ]0,\p_n[,  \qq\qq A'_n=[ \p_n, {\pi\over 2} [.$$ 

 We write 
 \begin{align*}   \sum_{j=1}^{d-1} 
&e^{ 2i \pi m_n(\frac{j}{d})  - 2\pi^2  B_n\, (\frac{j}{d})^2 +\mathcal O (m_n(\frac{j}{d})^3 ) }\cr &= \Big( \sum_{ j\ge 1\, :\, { \pi j\over d}\in A_n}+\sum_{ j\ge 1\, :\, { \pi j\over d}\in A'_n}
\Big)
e^{ 2i \pi m_n(\frac{j}{d})  - 2\pi^2  B_n\,(\frac{j}{d})^2+\mathcal O (m_n(\frac{j}{d})^3 ) } . 
\end{align*}  

  Concerning the second sub-sum,
the inequality $\p_n
\le {\pi j\over d}<{\pi \over 2}$ implies that  
\begin{eqnarray*}  \sum_{ j\ge 1\, :\, { \pi j\over d}\in A'_n} e^{ - 2\pi^2  B_n\,(\frac{j}{d})^2}&\le&
\sum_{ j\ge 1\, :\, { \pi j\over d}\in A'_n} e^{ - 2c   \frac{n}{\log n}\,(\frac{\pi j}{d})^2}
\cr &\le&\sum_{ j\ge 1\, :\,{ \pi j\over d}\in A'_n} e^{ -  c\a   \,  \log n       }\le
dn^{-  \a'  }.
\end{eqnarray*}
Therefore 
\begin{align*}  \Big| \sum_{ j\ge 1: { \pi j\over d}\in A'_n}  e^{ 2i \pi m_n(\frac{j}{d})  - 2\pi^2  B_n\,(\frac{j}{d})^2+\mathcal O (m_n(\frac{j}{d})^3 ) }\Big|
  &\le    \sum_{ j\ge 1:{ \pi j\over d}\in A'_n}  e^{  
 - 2\pi^2  B_n\,(\frac{j}{d})^2}\big(1   + \mathcal O ( m_n\,d ) \big)\cr &\le 
Cm_nd \sum_{ j\ge 1:{ \pi j\over d}\in A'_n} e^{ - 2c   \frac{n}{\log n}\,(\frac{\pi j}{d})^2}
\cr &\le Cm_nd\sum_{ j\ge 1: { \pi j\over d}\in A'_n} e^{ -  c\a   \,  \log n       }
\cr &\le C\big( \frac{n}{\log n}\big)d^2
 n^{-  \a'  }\, \le \, C d^2\,n^{1-  \a'  }.
\end{align*}
Now, concerning the first sub-sum, 
\begin{eqnarray*}& &  \sum_{ j\ge 1\, :\, { \pi j\over d}\in A_n}  e^{ 2i \pi m_n(\frac{j}{d})  -  2\pi^2  B_n\,(\frac{j}{d})^2+\mathcal O (m_n(\frac{j}{d})^3 ) } 
\cr &= &\sum_{  1\le  j\le \frac{d}{\pi}( \log n)  ( { \a \over 2 n} )^{1/2}}e^{ 2i \pi m_n(\frac{j}{d})  - 2\pi^2  B_n\,(\frac{j}{d})^2+\mathcal O (m_n(\frac{j}{d})^3 ) }
\cr &=& \sum_{  1\le  j\le \frac{d}{\pi}( \log n)  ( { \a \over 2 n} )^{1/2}}e^{ 2i \pi m_n(\frac{j}{d})  - 2\pi^2  B_n\,(\frac{j}{d})^2  }\big(\,1+ \mathcal O (m_n(\frac{j}{d})^3 )\,\big)
.
\end{eqnarray*}
We have 
\begin{equation*} m_n  \sum_{  1\le  j\le \frac{d}{\pi}( \log n)  ( { \a \over 2 n} )^{1/2}}(\frac{j}{d})^3\,\le\, C\, \frac{d^2 \a^2(\log n)^3}{n} .
\end{equation*}
Thus
\begin{eqnarray*}&&\sum_{  1\le  j\le \frac{d}{\pi}( \log n)  ( { \a \over 2 n} )^{1/2}}
 e^{ 2i \pi m_n(\frac{j}{d})  - 2\pi^2  B_n\,(\frac{j}{d})^2  }\big(1+ \mathcal O (m_n(\frac{j}{d})^3 )\big) 
\cr &=&\sum_{  1\le  j\le \frac{d}{\pi}( \log n)  ( { \a \over 2 n} )^{1/2}}e^{ 2i \pi m_n(\frac{j}{d})  - 2\pi^2  B_n\,(\frac{j}{d})^2  }  
\cr &&\quad +\mathcal O \Big( \sum_{  1\le  j\le \frac{d}{\pi}( \log n)  ( { \a \over 2 n} )^{1/2}}e^{    - 2\pi^2  B_n\,(\frac{j}{d})^2  }  \Big)\Big( m_n  \sum_{  1\le  j\le \frac{d}{\pi}( \log n)  ( { \a \over 2 n} )^{1/2}}(\frac{j}{d})^3 \Big)
\cr &=&\sum_{  1\le  j\le \frac{d}{\pi}( \log n)  ( { \a \over 2 n} )^{1/2}}e^{ 2i \pi m_n(\frac{j}{d})  - 2\pi^2  B_n\,(\frac{j}{d})^2  }  
\cr &&\quad +\mathcal O \Big(   \frac{d^2 \a^2(\log n)^3}{n} \sum_{  1\le  j\le \frac{d}{\pi}( \log n)  ( { \a \over 2 n} )^{1/2}}e^{    - 2\pi^2  B_n\,(\frac{j}{d})^2  }  \Big)
\end{eqnarray*}
Whence, 
\begin{align*}  \sum_{ j\ge 1\, :\, { \pi j\over d}\in A_n} &e^{ 2i \pi m_n(\frac{j}{d})  -  2\pi^2  B_n\,(\frac{j}{d})^2+\mathcal O (m_n(\frac{j}{d})^3 ) } 
\cr &=\sum_{  1\le  j\le \frac{d}{\pi}( \log n)  ( { \a \over 2 n} )^{1/2}}e^{ 2i \pi m_n(\frac{j}{d})  - 2\pi^2  B_n\,(\frac{j}{d})^2  }  
  +\mathcal O \Big(   \frac{d^3 \a^3(\log n)^4}{n^{3/2}}   \Big) .
\end{align*}
Combining both estimates gives
\begin{align*}   
\P\{d|S_n\} \,=\, \frac{1}{d}  + \frac{1}{d}\sum_{  1\le  j\le \frac{d}{\pi}( \log n)  ( { \a \over 2 n} )^{1/2}} &  e^{ 2i \pi m_n(\frac{j}{d})  - 2\pi^2  B_n\,(\frac{j}{d})^2  }  +\mathcal O(d^2 n^{-\a'})
 \cr &\qq +\mathcal O \Big(   \frac{d^3 \a^3(\log n)^4}{n^{3/2}}   \Big)
 \cr \, = \frac{1}{d}  + \frac{1}{d}\sum_{  1\le  j\le \frac{d}{\pi}( \log n)  ( { \a \over 2 n} )^{1/2}} &  e^{ 2i \pi m_n(\frac{j}{d})  - 2\pi^2  B_n\,(\frac{j}{d})^2  }  +\mathcal O \Big(   \frac{d^3 \a^3(\log n)^4}{n^{3/2}}   \Big)
.
\end{align*}
This achieves the proof.
 \end{proof}

Now  as $\sum_{j\ge d/2} e^{-A\frac{j^2}{2d^2}}\le e^{-cA }$, $c$ an absolute constant, for all $d\ge 2$ and $A\ge 2$, we have
\begin{align*}   
 \sum_{  j> \frac{d}{\pi}( \log n)  ( { \a \over 2 n} )^{1/2}}    e^{    - 2\pi^2  B_n\,(\frac{j}{d})^2  }& \le  \sum_{  j> \frac{d}{\pi}( \log n)  ( { \a \over 2 n} )^{1/2}}    e^{    - A(\frac{n}{\log n})(\frac{j}{d})^2  } 
 \cr &\le e^{    - A'(\frac{n}{\log n})(\frac{1}{d})^2(d( \log n))^2  ( { \a \over 2 n} ) )  } \cr & \le e^{    - A''\a \log n } =n^{- A''\a}.
\end{align*}
 
 We therefore also have
\begin{corollary}
\begin{eqnarray*}\Big|\P\{d|S_n\}-\frac{\Theta (d,m_n,B_n)}{d}\Big|\le C\ \frac{d^3 \a^3(\log n)^4}{n^{3/2}}  .
\end{eqnarray*}
\end{corollary}

We end  the paper with recent results related to   primality in the Cram\'er model.

\subsection{Quasiprime and prime property.}
We   obtained in   \cite{W8} -without assuming RH- a  sharp estimate  of $\P \{S_n\ \text{prime}\, \}$, for almost all $n$, namely for all $n$, $n\to\infty$ through a set $\mathcal S$  of natural density $1$. 
\begin{theorem}\label{Snprime.M} 
{\rm (i)} For any constant $b>1/2$, 
\begin{equation}\label{Snprime.i}   \P \{S_n\ \text{prime}\, \}
\, =\, \frac{ 1 }{ \sqrt{2\pi  B_n  } }\,\int_{m_n-\sqrt{ 2bB_n\log n}}^{m_n+\sqrt{ 2bB_n\log n}}
 e^{- \frac{(t- m_n)^2}{    2  B_n   } } \, \dd \pi(t)  
    + \mathcal O\Big( \frac{(\log n)^{3/2 }}{\sqrt n} \Big),
  \end{equation}
as $n\to\infty$, where   $\pi(.)$ is the prime counting function.
  \vskip 3 pt {\rm (ii)} There exists a set of integers $\mathcal S$ of density 1, such that  \begin{equation}\label{abel.base,}  
 \P \{S_n\ \text{prime}\, \}
\, =\, \frac{ (1+ o( 1) )}{ \sqrt{2\pi  B_n  } } 
\int_{m_n-\sqrt{ 2bB_n\log n}}^{m_n+\sqrt{ 2bB_n\log n}}
  \,  e^{- \frac{(t- m_n)^2}{    2  B_n   } }\, \dd \pi(t),  \end{equation} 
as $n\to \infty$,  $n\in \mathcal S$. Further,
   \begin{equation}\label{abel.base,,}  
\liminf_{  \mathcal S\ni n\to\infty}\ (\log n)\P \{S_n\ \hbox{prime}  \}
\ge  \frac{1  }{  \sqrt{2\pi e}\, } .
   \end{equation}
 \end{theorem} 
The proof  allows one (Remark \ref{rem.Snprime.M}) to treat other cases, so we   included it. \begin{proof}[Proof of Theorem \ref{Snprime.M}.]
(i) By Lemma 7.1 p.\,240 in \cite{P}, for $0\le x\le B_n$
\begin{eqnarray*} \P \{|S_n-m_n|\ge x \}&=& \P \{S_n-m_n\ge x \}
+\P \{-(S_n-m_n)\ge x \}
\cr &\le &2\,\exp\Big\{-\frac{x^2}{2B_n}\Big(1-\frac{x}{2B_n}\Big)\Big\},
\end{eqnarray*}
noticing that   $\{-\xi_j\}_j$ also satisfies the conditions of Kolmogorov's Theorem. Let $b>b'>1/2$.
Then for all sufficiently large $n$, since $\log  B_n\sim \log n$,
\begin{equation} \P \{|S_n-m_n|\ge  \sqrt{2bB_n\log   n} \}\le
2\,  n^{-b'}
.
\end{equation}

 We have \begin{align*}   \Big|  \P \{S_n \in\mathcal P \}&  -  \P \{S_n \in \mathcal P\cap[m_n- \sqrt{ 2bB_n\log n},m_n+\sqrt{ 2bB_n\log n}] \} \Big| 
\cr &\le    \P \{|S_n-m_n|\ge \sqrt{ 2bB_n\log n} \}
\cr &\le   n^{-b'}.
   \end{align*}
Further,
\begin{eqnarray*}& & \Big|  \P \{S_n \in \mathcal P\cap[m_n-\sqrt{ 2bB_n\log n},m_n+\sqrt{ 2bB_n\log n}] \}
\cr & &-\sum_{\k\in\mathcal P\cap[m_n-\sqrt{ 2bB_n\log n},m_n+\sqrt{ 2bB_n\log n}]} \frac{ e^{- \frac{(\k- m_n)^2}{    2  B_n   } } }{ \sqrt{2\pi  B_n  } } 
\Big|
 \cr &\le & \sum_{\k\in\mathcal P\cap[m_n-\sqrt{ 2bB_n\log n},m_n+\sqrt{ 2bB_n\log n}]}\Big|\P \{S_n =\kappa \} -\frac{ e^{- \frac{(\k- m_n)^2}{    2  B_n   } } }{ \sqrt{2\pi B_n   } }     \Big|   
  \cr     & \le &    C\,\#\big\{\mathcal P\cap[m_n-\sqrt{ 2bB_n\log n},m_n+\sqrt{ 2bB_n\log n}] \big\}\cdot
 \frac{(\log n)^{3/2}}{n} 
  \cr     & \le & 
  C\, \sqrt b\ \frac{(\log n)^{3/2 }}{\sqrt n} 
.
  \end{eqnarray*}
Therefore
\begin{equation} \label{3} \Big|  \P \{S_n \in \mathcal P  \}
 -\sum_{\k\in\mathcal P\cap[m_n-\sqrt{ 2bB_n\log n},m_n+\sqrt{ 2bB_n\log n}]} \frac{ e^{- \frac{(\k- m_n)^2}{    2  B_n   } } }{ \sqrt{2\pi  B_n  } } 
\Big|
     \, \le \, 
   C\,\sqrt b\  \frac{(\log n)^{3/2 }}{\sqrt n} 
.
  \end{equation}
 By    expressing the inner sum as a Riemann-Stieltjes integral  \cite[p.\,77]{A}, we get

\begin{equation}\label{abel.base}   \P \{S_n \in \mathcal P  \}
\, =\, \int_{m_n-\sqrt{ 2bB_n\log n}}^{m_n+\sqrt{ 2bB_n\log n}}
  \,\frac{ e^{- \frac{(t- m_n)^2}{    2  B_n   } } }{ \sqrt{2\pi  B_n  } }\, \dd \pi(t)
    + \mathcal O\Big( \frac{(\log n)^{3/2 }}{\sqrt n} \Big)
.
  \end{equation}

\vskip 5 pt (ii) We note that 
  \begin{eqnarray} \label{int.min}
  \int_{m_n-\sqrt{ 2bB_n }}^{m_n+\sqrt{ 2bB_n }}
  \,\frac{ e^{- \frac{(t- m_n)^2}{    2  B_n   } } }{ \sqrt{2\pi  B_n  } } \dd \pi(t)&\ge & L \  \frac{ \pi(m_n+\sqrt{ 2bB_n })-\pi(m_n-\sqrt{ 2bB_n })}{ \sqrt{   B_n  }},
\end{eqnarray}
with $L=\frac{e^{-b }}{  \sqrt{2\pi }}$. 
\vskip 3 pt
 We use  Theorem 4 in  Selberg  \cite{Se}. Let    $\Phi(x)$ be positive and increasing and such that $\frac{\Phi(x)}{x} $  decreasing for $x>0$. Further assume that 
 \begin{eqnarray} \label{phi.ab}
{\rm (a)}\quad \lim_{x\to\infty}  \frac{\Phi(x)}{x}= 0 \qq\qq {\rm (b)}\quad\liminf_{x\to\infty}\frac{\log \Phi(x)}{\log x}>\frac{19}{77}.
\end{eqnarray}
Then there exists a   set $\mathcal S$   of positive reals of density one such that
 \begin{equation}  \label{phi.ab.enonce}\lim_{\mathcal S\ni x\to \infty}\frac{\pi(x+\Phi (x))-\pi(x)}{({\Phi(x)}/{\log x})}  =1.
\end{equation}
 Let $\Phi(x) = \sqrt{2b  x}$. Then the   requirements in \eqref{phi.ab} are fulfilled, and so \eqref{phi.ab.enonce} holds true. 
  Now let $C$ be some possibly large but fixed positive number, as well as some positive real $\d<1/2$.    
    \vskip 5 pt   By \eqref{phi.ab.enonce}, the  set of $x>0$, call it $\mathcal S_\d$, such   that 
    \begin{equation}\label{pi.phi.d}
\pi(x+\Phi(x))-\pi(x) \, \ge  \,(1-\d)\, \frac{ \Phi(x)}{\log x}.
 \end{equation}
  has    density 1. 
  Note that   $\mathcal S_{\d'}\subseteq \mathcal S_{\d''}$ if  $\d'\le \d''$.   
 \vskip 3 pt Pick  $x\in\mathcal S_\d$ and let $\D(x)= \pi(x+\Phi(x))-\pi(x)$. Note that if $|y-x|\le C$,   $\big|\Phi(y)-\Phi(x)\big|=o(1)$ for $x$   large.
 Thus  for every $y\in [x-C,x+C]$, $\big|\D(y)-\D(x)\big|\, \le C'$, and so 
$$\D(y)   \, \ge  \,(1-\d)\, \frac{ \Phi(x)}{\log x}-C'.$$
 the constant $C'$ depending on $C$ only. 
 As $\big|\frac{\Phi(x)}{\log x}- \frac{\Phi(y)}{\log y}\big|\le 
\frac{C}{2\sqrt x \log x}$, we have 
\begin{equation*}
\D(y)   \, \ge  \,(1-\d)\, \frac{ \Phi(y)}{\log y}-C' -\frac{C}{2\sqrt x \log x}.
 \end{equation*} 
  
Thus every $y\in [x-C,x+C]$ also satisfies \begin{equation}\label{pi.phi.d.a}
\D(y)   \, \ge  \,(1-2\d)\, \frac{ \Phi(y)}{\log y},
 \end{equation} 
 if  $x $ is large enough.
 
\vskip 3pt   
Let $\nu=\nu(x)$ be the  unique integer   such that $m_{\nu-1}< x\le m_{\nu}$. As $m_{\nu}-m_{\nu-1}=o(1)$,  $\nu \to \infty$,  it follows that $m_{\nu}\in [x-C,x+C]$ provided that   $x$ is large enough, in which case we have by   \eqref{pi.phi.d.a},
  \begin{equation}\label{pi.phi.d.b}
 \frac{\pi(m_{\nu}+\Phi(m_{\nu}))-\pi(m_{\nu})}{\Phi(m_{\nu})}  \ge  \frac{1-2\d}{\log m_{\nu}}.
 \end{equation} 
   Let $X\ge 1$ be a large positive integer and $\e$  a small positive real. The number $N(X)$ of intervals $]\m-1,\m]$, $\m\le X$ such that $\mathcal S_\d \cap ]\m-1,\m]\neq \emptyset$ verifies $N(X)/X\sim 1$, $X\to \infty$, since $\mathcal S_\d$ has density 1. 
 
 Given such an $\m\le X$, pick $x\in \mathcal S_\d \cap ]\m-1,\m]$. 
 We know (recalling that $m_{\nu}-m_{\nu-1}=o(1)$,  $\nu \to \infty$) that some $m_{\nu}$, $\nu=\nu(x)$ belongs to   $ ]\m-1-\e,\m+\e]$, and that \eqref{pi.phi.d.b} is satisfied.   The union of these intervals $[\m-1-\e,\m+\e ]$ is contained in $[1-\e , X+\e ]$. It follows that the number of $\nu$ such that \eqref{pi.phi.d.b} is satisfied, forms a set of density 1. 
  
 \vskip 4 pt  We now use an induction argument in order to replace $2\d$ in  \eqref{pi.phi.d.b} by a quantity $\e(\nu)$ which  tends to 0 as $\nu$ tends to infinity  along some other set of density 1, which we shall build explicitly.    Let $\mathcal T_n$ be the set   of $\nu$'s of density 1, corresponding to $\d=\frac1n$, $n\ge 3$. Let $X_3$ be large enough so   that $\#\{\mathcal T_3\cap [1, X ]\}\ge X (1-1/3)$ for all $X\ge X_3$. Next let  $X_4>X_3$ be sufficiently large  so   that $\#\{\mathcal T_4\cap [X_3, X ]\}\ge X (1-1/4)$ for all $X\ge X_4$. Like this we manufacture an increasing sequence $X_j$,   verifying for all $j\ge 3$,
  $$\#\{\mathcal T_j\cap [X_{j-1}, X ]\}\ge X (1-1/j), \qq \qq {\rm for\ all}\  X\ge X_j .$$
  The resulting set 
  $$ \mathcal T= \bigcup_{j=3}^\infty \mathcal T_j\cap [X_{j-1}, X_j ]$$
  has density 1 and further we have the inclusions 
  $$\mathcal T\cap [X_{l-1}, \infty )=\bigcup_{j=l}^\infty \mathcal T_j\cap [X_{j-1}, X_j ]\subset \mathcal T_l\cap\bigcup_{j=l}^\infty   [X_{j-1}, \infty )=\mathcal T_l\cap [X_{l-1},\infty),\qq\quad l\ge 4,$$
as   the sets $\mathcal T_j$ are decreasing with $j$ by definition.   

\vskip 3 pt We finally have  by \eqref{int.min},
  \begin{equation}\label{pi.phi.d.b.1}
 \frac{\pi(m_{\nu}+\Phi(m_{\nu}))-\pi(m_{\nu})}{\Phi(m_{\nu})}  \ge  \frac{1-\e(\nu)}{\log m_{\nu}},
 \end{equation}
along $\mathcal T$, for some sequence of reals $\e(\nu)\downarrow 0$ as $\nu\to \infty$.
 
\vskip 3 pt Therefore 
  \begin{eqnarray}\label{base} 
  \int_{m_\nu-\sqrt{ 2bB_\nu }}^{m_\nu+\sqrt{ 2bB_\nu }}
  \,\frac{ e^{- \frac{(t- m_\nu)^2}{    2  B_\nu   } } }{ \sqrt{2\pi  B_\nu  } } \dd \pi(t)&\ge & L \  \frac{ \pi(m_\nu+\sqrt{ 2bB_\nu })-\pi(m_\nu-\sqrt{ 2bB_\nu })}{ \sqrt{   B_\nu  }}\cr & \ge &  L\, \frac{1-\e(\nu)}{\log m_{\nu}},
\end{eqnarray}
for all $\nu \in\mathcal T$, recalling that $L=\frac{e^{-b }}{  \sqrt{2\pi }}$.

Also  \begin{equation}\label{abel.base,proof}  
\liminf_{  \mathcal T\ni \nu\to\infty} (\log \nu)\P \{S_\nu\ \hbox{prime}  \}
\ge  \frac{1  }{ \sqrt{2\pi e}\, } .
 \end{equation}
  It further follows from \eqref{abel.base} that  
\begin{equation}\label{abel.base,,proof}  
 \P \{S_\nu\ \hbox{prime}  \}
\, =\, \big(1+ o( 1)\big)
\int_{m_\nu-\sqrt{ 2bB_\nu\log \nu}}^{m_\nu+\sqrt{ 2bB_\nu\log \nu}}
  \,\frac{ e^{- \frac{(t- m_\nu)^2}{    2  B_\nu   } } }{ \sqrt{2\pi  B_\nu  } }\, \dd \pi(t),
  \end{equation}
   all $\nu\in \mathcal T$. This completes the proof of Theorem \ref{Snprime.M}.
   \end{proof}
\begin{remark}\label{rem.Snprime.M}\rm 
The interested reader will have  observed that the property $m_\nu-m_{\nu-1}= o(1)$ is not fully used, in fact it suffices that  $m_\nu-m_{\nu-1}= \mathcal O(1)$, selecting at the beginning of the proof $C$ sufficiently large.  Thus the proof of Theorem \ref{Snprime.M} also adapt to the Bernoulli case. In particular there exists a set of integers $\mathcal U$ of density 1 such that
\begin{equation}\label{bernoulli.d1}  
\liminf_{  \mathcal U\ni \nu\to\infty} (\log \nu)\P \{B_\nu\ \hbox{prime}  \}
>0 .
 \end{equation}
\end{remark}
\vskip 4  pt

  Let $\Pi_z=\prod_{p\le z}p$. According to Pintz \cite{Pi}, an integer $m$ is   $z$-quasiprime, if $(m, \Pi_z)=1$. Let $S'_n= \sum_{j= 8}^n \xi_j$, $n\ge 8$. Note that the introduction of $S'_n$ in place of $S_n$ is not affecting Cram\'er's conjecture. The probability that $S'_n$   be $z$-quasiprime is   for all $n$ large enough is studied in \cite{W8}, Theorem 2.6.
   The proof is based on a randomization argument.
     
 \begin{theorem}\label{cramer.quasi.prime.P}
 We have for any $0<\eta<1$, and all $n$ large enough  and      $ \zeta_0\le  \zeta\le \exp\big\{ \frac{c\log n}{\log\log n}\big\}$, 
\begin{eqnarray*}  \P\big\{  S'_n\hbox{\ $\zeta$-quasiprime}\big\}   \,\ge  \, (1-\eta)   \frac{ e^{-\gamma} }{  \log
\zeta }. \end{eqnarray*}
 where $\gamma$ is  Euler's constant and $c$ is a positive constant.\end{theorem} 

\vskip 3 pt  
 \subsection{Primality of ${P_n}$.} 
We also showed in \cite{W8} that when the \lq primes\rq~  $ P_\nu$   are observed   along    moderately growing subsequences,  then with probability 1, they  ultimately avoid       any given infinite set of primes satisfying a reasonable  tail's  condition.  We also test which infinite sequences of primes  are ultimately avoided  by   the \lq primes\rq~  $ P_\nu$, with probability 1. 

\begin{theorem}\label{mathfrak}  Let $\mathcal K$ be an increasing sequence of naturals such that  
the series  $\sum_{k\in \mathcal K}  k^{-\b} $ converges for some   $\b\in ]0, \frac12[$.
Let $\mathfrak P$  be an increasing sequence of primes such that for some  $b>1$,
 $$\sup_{k\in\mathcal K} \frac{\#\{\mathfrak P\cap [k,  b k] \}}{  k^{\frac12-\b}}<\infty. $$
Let also  $\{\D_k,k\ge 1\}$ be the instants of jump of the Bernoulli sequence $ \{B_k,k\ge 1\}$.
Then \begin{eqnarray*}
   \P\big\{   \D_k\notin\mathfrak P, \quad k\in \mathcal K\ \hbox{ultimately}\big\}=1.
    \end{eqnarray*}
Further, 
\begin{eqnarray*}
   \P\big\{   P_\nu\notin\mathfrak P, \quad \nu\in \mathcal K\ \hbox{ultimately}\big\}=1.
    \end{eqnarray*}
Moreover (case $\b=1/2$), let  $\mathfrak P$ be such that 
$\sum_{p\in \mathfrak P , \, p>y} p^{-1/2} = \mathcal O \big(y^{-1/2}\big)$, and $\mathcal K$ be   such that
  $\sum_{k\in \mathcal K}  k^{-1/2}<\infty$.
Then 
\begin{eqnarray*}
   \P\big\{   P_\nu\notin\mathfrak P, \quad \nu\in \mathcal K\ \hbox{ultimately}\big\}=1.
    \end{eqnarray*}
    \end{theorem} 

  \section{Concluding Remarks.}\label{s10}
At the period we were interested in studying the divisors properties of the Bernoulli random walk, the topic, the study of arithmetical properties of binomial random walks was not much attracting attention, and we  worked   isolated.
\vskip 2 pt
The study of $\E f(B_n)$ for instance is quite tractable because there are exact formulas for $\P\{B_n=k\}$ or $\P\{d|B_n=k\}$,  $\P\{P^-(B_n)>\zeta \}$, $\P\{  B_n \,{\rm prime} \}$, \ldots, and standard tools from 
Fourier analysis    can be used to obtain results on $\E f(B_n)$ from standard number theoretic results; but this was not much taken into account. The used methods are intrinsic to the Bernoulli model. See   Remark \ref{unif.est.theta.alpha.rem}. 
As  indicated in subsection \ref{sub.bin.way.back}, there is a way back to the non random setting along Pascal matrices; 
thus a study of the Bernoulli case transfers to the  arithmetical one using inverse of Pascal matrices.  We found Call and Velleman's result only recently.
 For    any arithmetical function $f(n)$
 we have the reverse formula
 $$ f(i)=\sum_{j=1}^i  (-1)^{i-j}2^j{{i-1}\choose{j-1}}\E f(B_j),\qq\quad i\ge 2.
$$ 
Now the research in this area has received a recent growing interest.
We steadily found the subject attractive because of the combined interaction 
 of Arithmetic and Probability  which can just be promising, and are convinced  that much remains to discover, and to do. The topic has large potential of development.
\vskip 2 pt
Random walks are also an important tool of investigation  in Analytic  Number Theory.  This topic is the object of the recent work Weber \cite{W15}.     
  \appendix
   

\section{Erd\"os-R\'enyi's model}\label{s5.a}

    Erd\"os and R\'enyi \cite{ER} have   investigated models for integers involving  infinite sequences of  independent   random
variables $\b_n$, with $\P\{\b_n=0\} =1-\P\{\b= 1\}=\varrho_n
$. They thorougly studied
 their additive properties in relation with classical problems of   additive representation  of integers such as Waring problem. See
Halberstam and Roth \cite{HR}, see also Landreau \cite{L}.   

\section{Kubilius' model} The Bernoulli model  only involves  one
probability space: the infinite product of probability spaces      $(\{0,1\}^\N, \mathcal B(\{0,1\}^\N),
\m^\N )$, where  $\m= \frac{1}{2}(\d_{0}+\d_{1})$,  $\d_x$ denoting the Dirac mass at point $x$. By construction,  the Kubilius model  relies upon all
probability spaces
$$\hbox{$(\{ 1,2, \ldots ,n \}, \m_n)$}$$
  $\m_n$ being  the normalized counting measure.
This is an obstacle for the application of general criteria of  almost sure convergence, typically G\'al-Koksma, Stechkin or M\'oricz theorems. 
   However,
     a bridge exists in Kubilius theory with some standard infinite  product of probability  spaces, namely with some inherent random walk built up from
independent non identically distributed random variables. 

  Let
$\{Y_p, p\ge 1\}$ be a sequence of independent binomial random variables such that
$\P\{Y_p=1\}=1/p$ and
$\P\{Y_p=0\}=1-1/p$. We can view $Y_p$ as modelling whether   an integer taken at random is divisible by $p$ or not. Let 
$$ T_n=\sum_{p\le n}  Y_p.  $$
By Mertens estimate
 $s_n^2=\E (T_n-\E T_n)^2= \sum_{p\le n} {1\over p}-{1\over p^2}=\log\log n +\mathcal O(1)$.  
 The sequence $\{T_n, n\ge 1\}$ is known to asymptotically behave as the truncated prime divisor function
$$\o(m,t)=\#\{ p\le t : p|m\} ,  $$
  at least when $t$ is not too close to $m$.  More precisely, let
$$\o_r(m)=\big( \o(m,1), \ldots, \o(m,r)\big), $$
where $r$ is some integer with $2\le r\le x$, and put $u={\log x\over \log r}$. The bridge mentionned before is given by the following sharp estimate.
\begin{theorem} \label{km}  Given $c<1$ arbitrary, we have uniformly in
$x,r
$ and
$Q\subset \Z^r$,
$$  \frac{\#\{ m\le x: \o_r(m)\in Q\}}{ x}=\P\big\{(T_1, \ldots, T_r)\in Q\big\}+ \mathcal O\big( x^{-c}+ 
e^{-u\log u}\big).
$$
\end{theorem} 
   We refer   to Elliott \cite{E} p.119-122 and   Tenenbaum \cite{Te} where a thorough study is presented.
 \vskip 2pt  The error term $e^{-u\log u}$  
naturally    brings  restrictions in   applications   to asymptotic problems.  To make it small,
it requires     if
$r=r(x)$ that
$  r(x)=\mathcal O _\e (x^\e)$   for all $\e>0$.   This amounts to truncate  the prime divisor function $\o(m )$  at level
$\mathcal O _\e (x^\e)$, which is satisfactory as long as $m\ll x$. However, these integers have a negligible contribution on the size
of    the left-term in Theorem  \ref{km}.   The  model is therefore mostly adapted to the analysis of the distribution of   small prime divisors
of an integer (see notably \cite{E} p.122). 
 \section{Proof  of Theorem \ref{romult1}.}   
 
  The following lemma  is   classical.  
\begin{lemma} \label{romult}Let $f\in \Z(X)$ and put $\varrho_f(d)=\#\big\{0\le y<d: d|f(y)\big\}$. Then $\varrho_f$ is a
multiplicative function.
\end{lemma}
 
\vskip 4 pt 
\noi   {\it Comment:}  One way to \lq expedite\rq \, the proof is to say that it is \lq\lq well-known and trivial\rq\rq, or to say that \lq\lq it is proved using the Chinese Remainder Theorem\rq\rq. We have here a proof ready and  are of mind we can't get out of it so. Further the present paper also address to probabilists, who are not necessarily    acquainted  with  arithmetical arcana.   
 We provide a simple and complete proof. We couldn't find it in the existing literature,  where however,  different but sometime sketchy proofs are  available.
  \begin{proof}
  Write $f(x)= a_0+a_1x+\ldots + a_nx^n$, $a_j\in \Z$, $0\le j\le n$. Let $d=d_1d_2$ with $(d_1,d_2)=1$. We first
show that
$\varrho_f(d_1)\varrho_f(d_2)\le
\varrho_f(d )$. Let
$(y_1, y_2)$ be such that $0\le y_i<d_i$ and   $d_i|f(y_i)$, $i=1,2$. 
\vskip 3 pt
Now as $(d_1,d_2)=1$,   there exists an 
 integer $y$, $0\le y<d$ such that (\cite[Th.\,59]{HW}),
\begin{equation}\label{equiv}\hbox{$y\equiv y_i\  {\rm mod}(d_i)$,\qq$i=1,2$.}
\end{equation}
Writing that  $y=y_1+\ell d_1$ we have
 \begin{eqnarray*}f(y)&=& a_0+a_1(y_1+\ell d_1)+\ldots + a_n(y_1+\ell d_1)^n \cr & =&a_0+(a_1 y_1+d_1 A_1  ) +\ldots + (a_n y_1^n +  d_1A_n)=
f(y_1)+ d_1 F ,
\end{eqnarray*}
where $A_1, \ldots, A_n$, $F $ are integers. As   $d_1|f(y_1)$ by assumption, it follows that   $d_1|f(y)$. Writing now that  $y=y_2+m d_2$, we similarly get  $d_2|f(y)$, whence $d|f(y)$. 
\vskip 3 pt
Now let $(y'_1, y'_2)$ be such that
$$\hbox{$0\le y'_i<d_i$ and   
$d_i|f(y'_i)$, \qq $i=1,2$.}$$ 

Let also $y'$ integer, be 
  satisfying 
$$\hbox{$0\le y'<d$, \qq $y'\equiv
y'_i\  {\rm mod}(d_i)$,\quad $i=1,2$,}$$ 
   thereby implying that  $d|f(y')$ by the above argument. 
\vskip 3 pt
 
 We have the implication: $y=y'\Rightarrow (y_1, y_2)=(y'_1, y'_2)$. Indeed as $y=y'$,  we have
$$y= y_1+\ell d_1=y_2+k d_2=y'_1+\ell' d_1=y'_2+k' d_2.$$
And so $y_1-y'_1=  (\ell'-\ell) d_1$. Since $0\le y_1, y'_1<d_1$, this implies $y_1=y'_1$. Similarly we get $y_2=y'_2$, so that $(y_1,
y_2)=(y'_1, y'_2)$. 
Therefore $\varrho_f(d_1)\varrho_f(d_2)\le
\varrho_f(d )$. 

\vskip 3 pt Conversely, let $0\le y<d$ be such that $d|f(y)$. Let $y_1, y_2$, $0\le y_i<d_i$ be such that  $y_i\equiv y \, {\rm
mod}(d_i)$, $i=1,2$. Then, in the same fashion
\begin{eqnarray*}f(y_1)&=& a_0+a_1(y +\ell d_1)+\ldots + a_n(y +\ell d_1)^n \cr &=&a_0+(a_1 y +d_1 B_1  ) +\ldots + (a_n y ^n +  d_1B_n)\cr & =&
f(y )+ d_1 G.
\end{eqnarray*}
Thus $d_1|f(y_1)$ and  similarly $d_2|f(y_2)$. 

\vskip 2 pt Now let   $0\le y'<d$ be such that $d|f(y')$, and let $(y'_1, y'_2)$ be the corresponding
pair of integers. We shall prove the implication $  (y_1, y_2)=(y'_1, y'_2) \Rightarrow y=y'$. 

Write $y_1=y+\ell d_1$,
$y'_1=y'+\ell' d_1$. If $y_1=y'_1$, then $y-y'=(\ell'-\ell)d_1 $ so that $d_1|y-y'$. Similarly $y_2=y'_2$ implies $d_2|y-y'$. Thus
$d |y-y'$. As $0\le y,y'<d$ this implies that $y=y'$. Hence the requested implication.
We deduce  that
$ 
\varrho_f(d )\le \varrho_f(d_1)\varrho_f(d_2) $. The proof is now complete. 
\end{proof}

\vskip 7 pt


    Before passing to the proof of Theorem \ref{romult1}, recall some related results. In McKee \cite{McK},   counting functions   associated with   quadratic forms $y^2 + by+c$ are investigated, under the restriction $\D= b^2 -4c<0$. As explained in section 2,  the number of solutions to the  congruence
 $$y^2 + by+c\equiv 0 \ {\rm (mod}\, d{\rm )}, \quad 0\le y<d,$$
 equals the number of solutions to the  congruence
 $y^2 \equiv 0 \ {\rm (mod}\, 4d{\rm )}$, $b\le y<b+2d$.  The condition $\D <0$ is not satisfied in the case we consider.   
   \vskip 2 pt
   The case of irreducible polynomials is mainly adressed in the literature and we couldn't find anything corresponding to the solved problem   in this work.    That question was already considered in Weber \cite{W4}.  For results concerning summatory functions associated with counting functions of the number of divisors of irreducible polynomials, we refer to Broughan \cite{Br}, Hooley \cite{Ho}, McKee \cite{McK1} and Scourfield \cite{S}, among  other contributors.  
\begin{proof}[\bf Proof of Theorem \ref{romult1}] 
The proof is  long and technical because  of many sub-cases entering into consideration, which is natural when studying this kind of question, according to sayings of patented number theorists colleagues. It can be used to treat similar questions. First consider the   case  $k>0$, which is the
main case. Recall that
\begin{eqnarray*}
\varrho_k(D)=\#\big\{1\le y\le  D: \  D| y^2+ky \big\} ,\qq \quad k=0,1,\ldots
\end{eqnarray*} 
The function $\varrho_k(D)$ being multiplicative,
it suffices to compute $\varrho_k(p^r)$. Observe to begin that 
$$\varrho_k(p^r)=2, \quad r=1,2,\ldots\qq\quad \hbox{if $p\not| k$}. $$ 
 Indeed, if $(y,p)=1$, then $p^r|y+k$ and there is only one solution given by
$y= \tilde k$ or $y=p^r-\tilde k$ where $\tilde k$ is the residue class of $k$ modulo $p^r$. If $y=p^s Y$, $(Y,p)=1$, $1\le s<r$, then $p^r| y(y+k)\Leftrightarrow p^{r-s} |(p^sY+k)$. And so $p|k$,
which was excluded. There is thus no solution of this kind. Finally, it remains one extra solution $y=p^r$. Thus $\varrho_k(p^r)=2$.
\vskip 5pt     We now assume that $p|k$. We can range the solutions $y$ of the equation $p^r|y(y+k)$ in disjoint classes of type $y=p^sY$,
with
$(Y,p)=1$. When
$r=1,2$ or
$3$, there is a direct simple computation.  
\smallskip\par   {\it Case $r=1$.}  We have $\varrho_k(p)=\#\big\{1\le y\le p: \,  p| y(y+k)  \big\}$. 
 Since $p |k$, it follows that  $p| y(y+k)\Leftrightarrow p|y^2$. As $y\le p$, the prime divisors of $y$ are less or equal to $p$ and    so are the prime divisors of $y^2$. There is thus only one solution: $y=p$.
Whence $    \varrho_k(p)=1$.
 \smallskip\par    {\it Case $r=2$.}   We have $\varrho_k(p^2)=\#\big\{1\le y\le p^2: \,  p^2| y(y+k)  \big\}
$. As $p\| k$, 
   If $(y,p)=1$, then $p^2| y(y+k)\Leftrightarrow p^2|y+k$ and so $p|y$, which is
impossible. There is no solution of this type. 
 If $y=pY$, with  $(Y,p)=1$, then  $p^2| y(y+k)\Leftrightarrow p|(pY+k)$, which  is satisfied. 
The number of solutions is $\#\{ Y\le p:
(Y,p)=1\}=p-1$. The case $y=p^2Y$, with  $(Y,p)=1$ coincides with $y=p^2$.
 Therefore,  
 $    \varrho_k(p^2)=   p-1+1=p$.
\smallskip\par {\it Case $r=3$.}   We have $\varrho_k(p^3)=\#\big\{1\le y\le p^3: \  p^3| y(y+k)  \big\}. 
$
     If $(y,p)=1$, then $p^3| y(y+k)\Leftrightarrow p^3|y+k$ and so $p|y$, which is
impossible. There is no solution of this type. If $y=pY$, with  $(Y,p)=1$, then  $p^3| y(y+k)\Leftrightarrow p^2|(pY+k)$. If
$v_p(k)\ge 2$, this implies that $p|Y$, which is impossible, and so there is no solution of this type in that case. If $v_p(k)=1$, write
$k=p K$, we get $p|Y+ K$.  The number of solutions is
$\#\{ Y\le p^2 :Y\equiv -K\, {\rm mod}(p)\}=p $.  If $y=p^2Y$, with  $(Y,p)=1$, then  $p^3| y(y+k)\Leftrightarrow p |(p^2Y+k)$, which is
always realized. The number of solutions is $\#\{ Y\le p: (Y,p)=1\}=p-1$. 
  The extra case is
$y=p^3 $.
 Therefore   
\begin{equation}    \varrho_k(p^3)=\begin{cases} 
                                   2p\quad &{\rm if}\ \quad v_p(k)=1
 \cr
                                   p  \quad &{\rm if}\ \quad v_p(k)\ge 2.   
            \end{cases}
\end{equation}
An additional   sub-case thus appears. 
Summarizing
\begin{equation}\varrho_k(p^r)=\begin{cases}  
 1\ &{\rm if} \    r=1 , 
 \cr
 p\ &{\rm if}\  r=2 ,   
  \cr
                                   2p  \quad &{\rm if}\   r=3 , \   v_p(k)=1 . 
\cr
                                   p  \ &{\rm if}\   r=3 , \   v_p(k)\ge 2 . 
\end{cases}  
\end{equation}

 \medskip {\it Case $r\ge 4$.} Put 
 $r' =\lfloor{r\over 2}\rfloor$.  Then
$r' 
\ge 2
$. We have
$\varrho_k(p^r)=\#\big\{1\le y\le p^r:
\  p^r| y(y+k) 
\big\}. 
$  
  If $(y,p)=1$, then $p^r| y(y+k)\Leftrightarrow p^r|y+k$ and so $p|y$, which is
excluded and there is no solution. 

Apart from the trivial solution $y=p^r$, the other possible solutions are of type $y=p^s
Y$,  
$(Y,p)=1$,
$1\le s<r$; and we shall distinguish three cases:   
$$\hbox{(i) $r'<s<r$, \qq (ii) $s=r'$,\qq (iii) $1\le s< r'$.}$$
 
\noi (i) Since $r'<s<r$, then $ r/2 \le  s $, and so $1\le r-s\le s$.  Further 
$p^r| y(y+k)$ means $p^{r-s}|  Y(p^sY+k)$ or
$p^{r-s}|  p^sY +k  $,   which is possible if and only if
$p^{r-s}|  k
$, namely $ r-s \le v_p(k)
$.  Thus 
$$\max(r'+1,r-v_p(k)) \le s<r.$$
 We have $Y\le p^{r-s}$, $(Y,p^{r-s})=1$. Their number
  is   $\phi(p^{r-s})$ where $\phi$ is Euler's function, and since $\phi(p^{r-s})=p^{r-s}(1-{1\over p}) $, 
the corresponding number of solutions is 
\begin{eqnarray*} \sum_{\max(r'+1,r-v_p(k)) \le s\le r -1 } p^{r-s}(1-{1\over p})   &=&   (1-{1\over p})\sum_{1\le  v  
\le (r-r'-1)\wedge v_p(k)}  
p^{v}
\cr &  =&  p\, {p^{(\lfloor {r-1\over 2}\rfloor )\wedge v_p(k)}-1\over p-1}(1-{1\over p}) 
\cr &  =&  
p^{\lfloor {r-1\over 2}\rfloor\wedge v_p(k)} -1   . 
\end{eqnarray*}
   (ii)      We consider   solutions of type    $y=p^{r'}Y$, $(Y,p)=1$. 
\smallskip\par  ---  If $r$ is odd, $r=2r'+1$, then  $p^{2r'+1}| p^{ r' }Y(p^{ r' }Y+k)$ means
$p^{ r'+1}|   (p^{ r' }Y+k)$. So $p^{ r'}|k$ and thereby $v_p(k)\ge r'$.
  If $v_p(k)< r'$, there is   no solution.  If $v_p(k)>  r'$, it implies that  $p|Y$ which is impossible and   there is   again no  
solution. 
 The remainding case $v_p(k)=  r'$ is the only one providing solutions. Write $k=p^{ r' }K $, $(K,p)=1$, then $p|Y+K$. Since
$(K,p)=1$, the solutions are the numbers
$Y$ such that
$1\le Y\le  p^{ r' +1}$ and 
$Y\equiv -K\, {\rm mod}(p)$. Let $1\le \kappa<p$ be such that $K  \equiv \kappa\, {\rm mod}(p)$.
The number of solutions is
\begin{equation} \#\big\{Y\le  p^{ r' +1} : Y\equiv -\kappa\  {\rm mod}(p)\big\} =\#\big\{ - \kappa  +jp: 1\le j\le p^{ r' }\big\}= p^{
r'  }.
\end{equation}
   --- If $r$ is even, $r=2r'$, then  $p^{2r' }| p^{ r' }Y(p^{ r' }Y+k)$ reduces to $p^{ r' }|  (p^{ r' }Y+k)$, so $p^{
r'}|k$. If $v_p(k)< r'$, there is no solution.  If $v_p(k)\ge   r'$,  write $k=p^{v_p(k) }K$, $(K,p)=1$, this is always realized and   the
number of solutions is
$$\#\big\{Y\le  p^{ r' } :  (Y,p)=1\big\} =\phi(p^{ r' })=p^{ r' }(1-{1\over p}) .$$ 
 \par \noi (iii) We consider the last type of  solutions: $y=p^sY$, $(Y,p)=1$,    $1\le s< 
r'$.  Notice first,  since $s<r'$ that  $s<r/2$, and so $r-s>r/2>s$.
As
$p^r| y(y+k)$ means 
 $p^{r-s}|   Yp^s+k $, we deduce that $p^s|k$, namely $ s\le v_p(k) $. 
 If $  v_p(k)<s $, there is no solution.  
  If $   v_p(k)>s $, then $p|Y$ which is impossible.  

If $s= v_p(k) $, which requires $v_p(k)<r'$, write $k=p^{v_p(k)}K$, $(K,p)=1$. Then
we get the equation $p^{r-2v_p(k)}|   Y +K $; so  $Y  
\equiv -K \, {\rm mod}(p^{r-2v_p(k)})$.  
  Let $1\le \kappa<p^{r-2v_p(k)}$ be such that $K  \equiv \kappa\, {\rm mod}(p^{r-2v_p(k)})$.
The number of solutions is
\begin{eqnarray*} \#\big\{Y\le p^{r-v_p(k)}:    Y   &\equiv& -\k \  {\rm mod}(p^{r-2v_p(k)})\big\}
\cr &=&\#\big\{  -\kappa +jp^{r-2v_p(k)}:  1\le j\le p^{v_p(k)}\big\}
\cr &=&p^{v_p(k)}.
\end{eqnarray*}
  \vskip 7pt
 Summarizing the case $r\ge 4$, if $r$ is odd, $r=2r'+1$ 
\begin{equation}     \varrho_k(p^r) = \begin{cases}   2p^{v_p(k)} & \quad\hbox{\rm if   $1\le v_p(k)\le r'$},
   \cr 
        p^{r'} & \quad\hbox{\rm if   $v_p(k)> r'$}.
   \end{cases}
\end{equation}
And if $r$ is even, $r=2r'$  
\begin{equation}       \varrho_k(p^r) =  \begin{cases}    
   2p^{v_p(k)} & \quad\hbox{\rm if  $1\le v_p(k)< r'$},
   \cr
 p^{r'} & \quad\hbox{\rm if   $v_p(k)\ge r'$} .
\end{cases}
\end{equation}
This remains true if $r=1,2,3$. Observe that ($v_p(k)< r'$, $r$ even) or ($v_p(k)\le  r'$, $r$ odd) are equivalent to $v_p(k)<{r\over
2}$.  
\smallskip\par Therefore, for
all $r\ge 2$ 
 
\begin{equation}        \varrho_k(p^r) =  \begin{cases}    2  & \quad {\rm if}  \   p\not| k  ,
   \cr  2p^{  v_p(k)} & \quad\hbox{\rm if}\  1\le v_p(k)<{r\over
2} ,
   \cr 
        p^{\lfloor{r\over
2}\rfloor}   \ & \quad\hbox{\rm if} \  v_p(k)\ge {r\over
2} .
     \end{cases}
\end{equation}

Consequently \begin{equation} \label{valro}\varrho_k(D)= \prod_{p|D  }\varrho_k(p^{  v_p(D) })=  \prod_{ v_p(k)<{v_p(D)/
2}}(2p^{v_p(k)})\cdot\prod_{   v_p(k)\ge {v_p(D)/
2} }  p^{\lfloor{v_p(D)\over
2}\rfloor} .
\end{equation} 
Recall that 
$$  D_{1/2} \,=\,\prod_{p|D} p^{\lfloor {v_p(D)\over 2}\rfloor }.$$

We consider again several cases:
\vskip 3 pt 
--- If $p$ is such that $v_p(D)=1$, then it produces in the writing of $\varrho_k(D)$ only a factor 1, since $v_p(k)<1/2$, so $v_p(k)=0$   and $\lfloor{v_p(D)\over
2}\rfloor=0$. 
\vskip 2 pt 
--- Now let $p$ be such that $v_p(D)\ge 2$, and compare $v_p(D)/2$ with $\lfloor{v_p(D)\over
2}\rfloor$. If $v_p(D)$ is even, both quantities are equal. If $v_p(D)$ is odd, $v_p(D)= 2z+1$, then 
$v_p(D)/2=z + 1/2$,  
$\lfloor{v_p(D)\over
2}\rfloor= z$. 
\vskip 2 pt
Thus in the latter case:
\vskip 2 pt
 
a) If $v_p(k)<{v_p(D)/2}$, namely $v_p(k)\le z$, we get the factor $2p^{v_p(k)}= 2p^{\min(v_p(k),z)}=2p^{  v_p((k,D_{1/2})) }$. If $v_p(k)\ge{v_p(D)/2}$, namely $v_p(k)> z$, we have the factor $p^z=p^{\min(v_p(k),z)}=p^{  v_p((k,D_{1/2})) }$.

\vskip 2 pt
 
b) In the case $v_p(D)$ is even, if $v_p(k)<{v_p(D)/2}$, namely $v_p(k)< z$, we have the factor $2p^{v_p(k)}=2p^{\min(v_p(k),z)}=2p^{  v_p((k,D_{1/2})) }$. If $v_p(k)\ge {v_p(D)/2}= z$, we have the factor $p^{z}=p^{\min(v_p(k),z)}=p^{  v_p((k,D_{1/2})) }$.

\vskip 2 pt 

  Therefore the  contribution of 
 a prime $p$ such that $v_p(D)\ge 2$, is   either $2p^{  v_p((k,D_{1/2})) }$ if $p$ is part of the first product, or $p^{  v_p((k,D_{1/2})) }$ if $p$ is part of the second. 
 
 A prime number $p$ is part of the first product, if and only if the inequation $1\le v_p(k)<v_p(D)/2$ is fulfilled.
 We 
therefore get the condensed form, 
\begin{eqnarray} \label{valro1}\varrho_k(D)= \prod_{p|D  }\varrho_k(p^{  v_p(D) })&= & \prod_{ v_p(k)<{v_p(D)/
2}}(2p^{v_p(k)})\cdot\prod_{   v_p(k)\ge {v_p(D)/
2} }  p^{\lfloor{v_p(D)\over
2}\rfloor}\cr  &=&2^{\#\{p\,:\, 1\le v_p(k)<v_p(D)/2 \}}\,\prod_{p} p^  {\min (v_p(k),v_p(D)/2)} 
\cr  &=&2^{\#\{p\,:\, 1\le v_p(k)<v_p(D)/2 \}}\, (k, D_{1/2}).
\end{eqnarray}

\vskip 5 pt Finally,   consider the case  $k=0$, namely       $\varrho_0(p^r)=\#\big\{1\le y\le p^r: \  p^r| y^2   \big\}  $. Notice that 
$\varrho_0(p)=\#\big\{1\le y\le
p: \  p| y   \big\}=1 $. Let
$r>1$, and write   $y=p^sY$, $(Y,p)=1$ and $1\le s\le r$. If $s=r$, there is the trivial unique solution $y=p^r$. If $1\le s<r$,
 $p^r| y^2\Rightarrow 2s\ge r$, in which case, the number of solutions is  
$$\#\big\{Y\le p^{r- s}: (Y,p)=1\big\}=\phi(p^{r- s})=p^{r- s}(1-{1\over p}).  $$
 Consequently
 $\varrho_0(p^r)= 1+\sum_{r/2\le s< r}p^{r- s}(1-{1\over p})$ . 
If $r$ is even, $r=2r'$
$$\varrho_0(p^r)= 1+\sum_{\s =1}^{r' } p^{\s}(1-{1\over p})= 1+p\sum_{u =0}^{r'-1 } p^{u}(1-{1\over p})= 1+p\Big({p^{r' }-1\over p-1}\Big)
 ( {p-1\over p})=p^{r' } ,$$
whereas if $r$ is odd, $r=2r'+1$,
 $\varrho_0(p^r)= 1+\sum_{\s =1}^{r'} p^{\s}(1-{1\over p})=p^{r' }$.
Thus 
$$\varrho_0(p^r)=p^{\lfloor {r\over 2}\rfloor }. $$
It follows that 
$$\varrho_0(D)=\prod_{p|D} p^{\lfloor {v_p(D)\over 2}\rfloor }\, =\, D_{1/2}.  $$
The proof is now complete. 
\end{proof}

 \vskip 7 pt
 \noindent {\bf Acknowledgments} 
 \vskip 3 pt  I    thank  Sanying Shi   for  exchanges     around divisibility  in Bernoulli random walk.


 
 \end{document}